\documentclass[12pt, A4paper, leqno]{amsart}

\usepackage{verbatim}
\usepackage{amssymb}
\usepackage{amsbsy}
\usepackage{amscd}
\usepackage{amsmath}
\usepackage{amsthm}
\usepackage[mathscr]{eucal}

\usepackage{graphicx}
\usepackage{epstopdf}
\usepackage{sfmath}
\usepackage{comment}

\usepackage[small, nohug, heads-vee]{diagrams} 
\diagramstyle[labelstyle=\scriptstyle]

\usepackage{url}

\usepackage[backref,bookmarks=true,colorlinks=true,pdfstartview=FitV,linkcolor=blue, citecolor=red,urlcolor=green]{hyperref}  

\newtheorem{theorem}{Theorem}[section]
\newtheorem{lemma}[theorem]{Lemma}
\newtheorem{corollary}[theorem]{Corollary}
\newtheorem{proposition}[theorem]{Proposition}

\theoremstyle{remark}
\newtheorem{remark}[theorem]{Remark}
\newtheorem{hypothesis}[theorem]{Hypothesis}
\newtheorem{example}[theorem]{Example}

\newtheorem{definition}{Definition}\numberwithin{definition}{section}

\newcommand\bR{{\mathbb{R}}}
\newcommand\bC{{\mathbb C}}
\newcommand\bZ{{\mathbb Z}}

\newcommand\Hom{{\rm Hom}}
\newcommand\dev{{\bf dev}}
\newcommand\SI{{\mathbb{S}}}

\newcommand\Bd{{\rm bd}}
\newcommand\clo{{\rm Cl}}
\newcommand\bdd{{\mathbf{d}}}

\newcommand\ra{\rightarrow}

\newcommand\emp{\emptyset}
\newcommand\eps{\epsilon}

\newcommand\Aff{{\mathbf{Aff}}}
\newcommand\ovl{\overline}
\newcommand\Aut{{\mathbf{Aut}}}
\newcommand\Idd{{\rm I}}

\newcommand\bv{{\mathbf{v}}}

\newcommand\CN{{\mathcal{N}}}
\newcommand\Pgl{{\mathrm{PGL}}(n+1, \bR)}
\newcommand\Ag{{\mathrm{Ag}}}

\newcommand\SL{{\mathsf{SL}}}

\newcommand\SO{{\mathsf{SO}}}

\newcommand\PGL{{\mathsf{PGL}}}
\newcommand\SLnp{{\mathsf{SL}}_\pm(n+1, \bR)}
\newcommand\SLn{{\mathsf{SL}}_\pm(n, \bR)}
\newcommand\SLf{{\mathsf{SL}}_\pm(4, \bR)}

\newcommand\GL{{\mathsf{GL}}}

\newcommand\GLnp{{\mathsf{GL}}(n+1, \bR)}
\newcommand\PGLnp{{\mathsf{PGL}}(n+1, \bR)}

\newcommand\orb{\mathcal{O}} 
\newcommand\torb{\tilde{\mathcal{O}}}
\newcommand\bGamma{{\boldsymbol \Gamma}}
\newcommand\leng{{\mathrm{length}}}

\newcommand\cwl{{\mathrm{cwl}}}

\newcommand\SLpm{{\mathrm{SL}}_{\pm}(n+1, \bR)}

\usepackage[usenames,dvipsnames,svgnames,table]{xcolor}

\title[The classification of ends]
{A classification of radial or totally geodesic ends of real projective orbifolds I: a survey of results}

\author{Suhyoung Choi}
\address{ Department of Mathematics \\ KAIST \\
{Daejeon 34141, South Korea }
}
\email{schoi@math.kaist.ac.kr}

\date{\today}

\thanks{This work was supported by the National Research Foundation
of Korea (NRF) grant funded by the Korea government (MEST) (No.2010-0027001).} 

\begin{document}

\begin{abstract} 
Real projective structures on $n$-orbifolds are useful in understanding the space of 
representations of discrete groups into $\SL(n+1, \bR)$ or $\PGL(n+1, \bR)$. 
A recent work shows that many hyperbolic manifolds 
deform to manifolds with such structures not projectively equivalent to the original ones. 
The purpose of this paper is to understand 
the structures of ends of real projective $n$-dimensional orbifolds for $n \geq 2$.
In particular, these have the radial or totally geodesic ends. Hyperbolic manifolds with cusps 
and hyper-ideal ends are examples.
For this, we will study the natural conditions on eigenvalues of holonomy 
representations of ends when these ends are manageably understandable. 
We will show that only the radial or totally geodesic ends of lens shape
or horospherical ends exist for strongly irreducible properly convex 
real projective orbifolds under some suitable conditions. 
The purpose of this article is to announce these results. 

\end{abstract}


\maketitle

\tableofcontents




\section{Introduction}  


In these series of  
articles, we are interested in real projective structures on orbifolds, which are principally non-manifold ones. 
Orbifolds are basically objects finitely covered by manifolds. 
The real projective structures can be considered as torsion-free projectively flat affine connections on orbifolds. 
Another way to view is to consider these as an immersion from the universal cover $\tilde \Sigma$ of 
an orbifold $\Sigma$ to $\bR P^{n}$ equivariant with respect to a homomorphism $h: \pi_{1}(\Sigma)
\ra \PGL(n+1, \bR)$. 
These orbifolds have ends. We will study the cases when the ends are of {specific type.}
The types that we consider are radial ones, i.e., R-ends, where end neighborhoods are foliated 
by concurrent projective geodesics. 
A {\em hypersurface} is a codimension-one submanifold or suborbifold of a manifold or an orbifold. 
{Another type ones are {\em totally geodesic ones}, or T-ends, when the closures of end neighborhoods can 
be compactified by ideal totally geodesic hypersurfaces in some ambient real projective orbifolds.}

 Kuiper, Benz\'ecri, Koszul,  Vey, and Vinberg might be the first people to consider these objects seriously 
 as they are related to proper action of affine groups on affine cones in $\bR^{n}$. 
 We note here, of course, the older study of affine structures on manifolds
 with many major open questions.

\subsection{Some recent motivations.} 

Recently, there were many research papers on convex real projective structures on manifolds and orbifolds.
(See the work of Goldman \cite{Gconv}, Choi \cite{cdcr1}, \cite{cdcr2}, Benoist \cite{Ben1}, 
Kim \cite{ink}, Cooper, Long, Thistlethwaite \cite{Cooper2006}, \cite{CLT} and so on.)
Topologists will view each of these as
a manifold with a structure given by 
 a maximal atlas of charts to $\bR P^n$ where transition maps are projective. 
Hyperbolic and many other geometric structures will induce canonical real projective structures. 
Sometimes, these can be deformed to real projective structures not arising from such obvious constructions. 
In general, the theory of the discrete group representations and their deformations form very much mysterious subjects still. 
(See the numerous and beautiful examples in Sullivan-Thurston \cite{ST}.)


\begin{figure}
\centerline{\includegraphics[height=4cm]{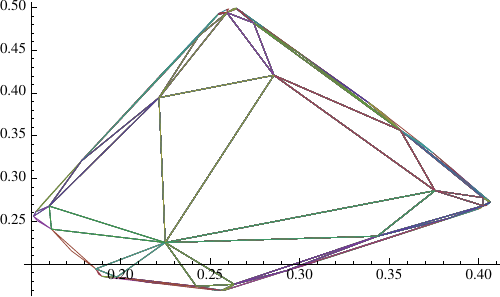}}
\caption{The developing images of convex $\bR P^2$-structures on $2$-orbifolds deformed from hyperbolic ones: $S^2(3,3,5)$.}
\label{fig:good2}
\end{figure}

\begin{figure}

\centerline{\includegraphics[height=4cm]{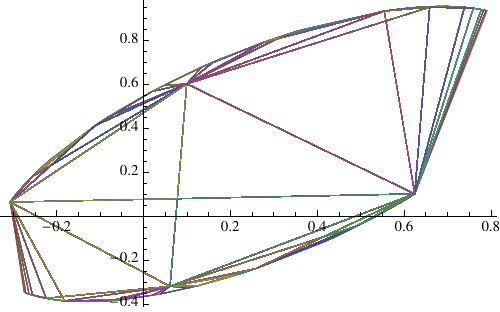}}
\caption{The developing images of convex $\bR P^2$-structures on $2$-orbifolds deformed from hyperbolic ones: 
$D^2(2, 7)$.}
\label{fig:good3}

\end{figure}

Since the examples are easier to construct, even now, we will be studying orbifolds, a natural generalization of 
manifolds. 
Deforming a real projective structure on an orbifold to an unbounded situation results in the actions of 
the fundamental group on affine buildings which hopefully will lead us to some understanding of 
orbifolds and manifolds in particular of dimension three as indicated by
Cooper, Long, Thistlethwaite, and Tillmann.

\subsection{Real projective structures on orbifolds with ends} 
It was discovered by D. Cooper, D. Long, and M. Thistlethwaite \cite{Cooper2006}, \cite{CLT} that 
many closed hyperbolic $3$-manifolds deform to real projective $3$-manifolds. 
Later S. Tillmann found an example of a $3$-orbifold obtained from pasting sides of a single ideal hyperbolic tetrahedron 
admitting a complete hyperbolic structure with cusps and a one-parameter family of 
real projective structure deformed from the hyperbolic one (see \cite{conv}).
Also, Craig Hodgson, Gye-Seon Lee, and I found a few other examples: $3$-dimensional ideal hyperbolic Coxeter orbifolds 
without edges of order $3$ has at least $6$-dimensional deformation spaces in \cite{CHL}. (See Example \ref{exmp-Lee}.)


Crampon and Marquis \cite{CM2} and Cooper, Long, and Tillmann \cite{CLT2} have done similar study with 
the finite volume condition. In this case, only possible ends are horospherical ones. 
The work here studies more general type ends 
while we have benefited from their work. 
We will see that there are examples where horospherical R-ends deform to lens-shaped R-ends  and vice versa
( see also Example \ref{exmp-Lee}.)
More recently, {Ballas}, Cooper, Long, Leitner and Tillmann also have made progresses on the classification of the ends where 
they require that the ends have neighborhoods with nilpotent fundamental groups. (See \cite{CLT3}
and \cite{Leitner1} and \cite{Leitner2}.)


\begin{remark}
A summary of the deformation spaces of real projective structures on closed orbifolds and surfaces is given 
in \cite{Cbook} and \cite{Choi2004}. See also Marquis \cite{Marquis} for the end theory of $2$-orbifolds. 
The deformation space of real projective structures on an orbifold 
loosely speaking is the space of isotopy equivalent real projective structures on
a given orbifold. (See \cite{conv} also.) 
\end{remark}

Also, it seems likely from some examples that these orbifolds with ends deform more easily. 

Our main aim is to understand these phenomena theoretically. It became clear from 
our attempt in \cite{conv} that we needed to understand and classify the types of 
ends of the relevant convex real projective orbifolds. We will start with the simplest ones: 
radial type ones or totally geodesic ones. 


But as Davis observed, there are many other types such 
as ones preserving subspaces of dimension greater than equal to $0$. 
In fact, Cooper and Long found such an example from 
$S/\SL(3, \bZ)$ for the space $S$ of unimodular positive definite bilinear forms. 
Since $S$ is a properly convex domain in $\bR P^{5}$ and $\SL(3, \bZ)$ acts projectively, 
$S/\SL(3, \bZ)$ is a strongly tame properly convex real projective orbifold
by the classical theory of lattices. These types of ends were compactified 
by Borel and Serre \cite{BS} for arithmetic manifolds. The ends are not of type studied here. 
We will not study these here.

In \cite{conv}, we show that the deformation spaces of real projective structures on orbifolds are locally homeomorphic to 
the spaces of conjugacy classes of representations of their fundamental groups 
where both spaces are restricted by some end conditions. 

It remains how to see for which of these types of real orbifolds, 
nontrivial deformations exist or not. 
For example, we can consider examples such as complete hyperbolic manifolds 
and how to compute the deformation space. 
From Theorem 1 in \cite{CHL} with Coxeter orbifolds, we know that 
a complete hyperbolic Coxeter orbifold always deforms nontrivially. 
(See also \cite{Choi2006}.) S. Ballas \cite{Ballas, Ballas2} also produced some results.  
We conjecture that maybe these types of real projective orbifolds with R-ends might be very flexible. 
Ballas, Danciger, and Lee \cite{BDL} announced that they have found much evidence for this very recently in Cooperfest in May 2015.
Also, there are some related developments for the complex field $\bC$ with the Ptolemy module in SnapPy as developed by  
S. Garoufalidis, M. Goerner, D. Thurston, and C. Zickert.  






\subsection{Our settings} 

Given an orbifold, 
recall the notion of universal covering orbifold $\torb$ 
with the orbifold covering map $p_{\orb}:\torb \ra \orb$
and the deck transformation group
$\pi_{1}(\orb)$ so that $p_{\orb}\circ \gamma = p_{\orb}$ for $\gamma \in \pi_{1}(\orb)$. 
(See \cite{Thnote}, \cite{BH}, \cite{Choi2004} and \cite{Cbook}.)
We hope to generalize these theories to noncompact orbifolds with some particular conditions on ends. 
In fact, we are trying to generalize the class of complete hyperbolic manifolds with finite volumes. 
These are $n$-orbifolds with compact suborbifolds 
whose complements are diffeomorphic to intervals times closed $(n-1)$-dimensional orbifolds. 
Such orbifolds are said to be {\em strongly tame} orbifolds. 
An {\em end neighborhood} is a  component of the complement of 
a compact subset not contained in any compact subset of the orbifold. 
An {\em end} $E$ is an equivalence class of compatible exiting sequences of end neighborhoods. 
Because of this, we can associate an $(n-1)$-orbifold at each end and we define 
the {\em end fundamental group} $\pi_1(E)$ as a subgroup of the fundamental group $\pi_1(\orb)$ of the orbifold $\orb$. 
We also put the condition on end neighborhoods being foliated by radial lines or to have totally geodesic ideal boundary. 

We studied some such Coxeter orbifolds with ends  in \cite{Choi2006} already. 
These have convex fundamental polytopes and are easier to understand. 

\begin{figure}
\begin{center}$
\begin{array}{cccc}
\includegraphics[width=3.5in]{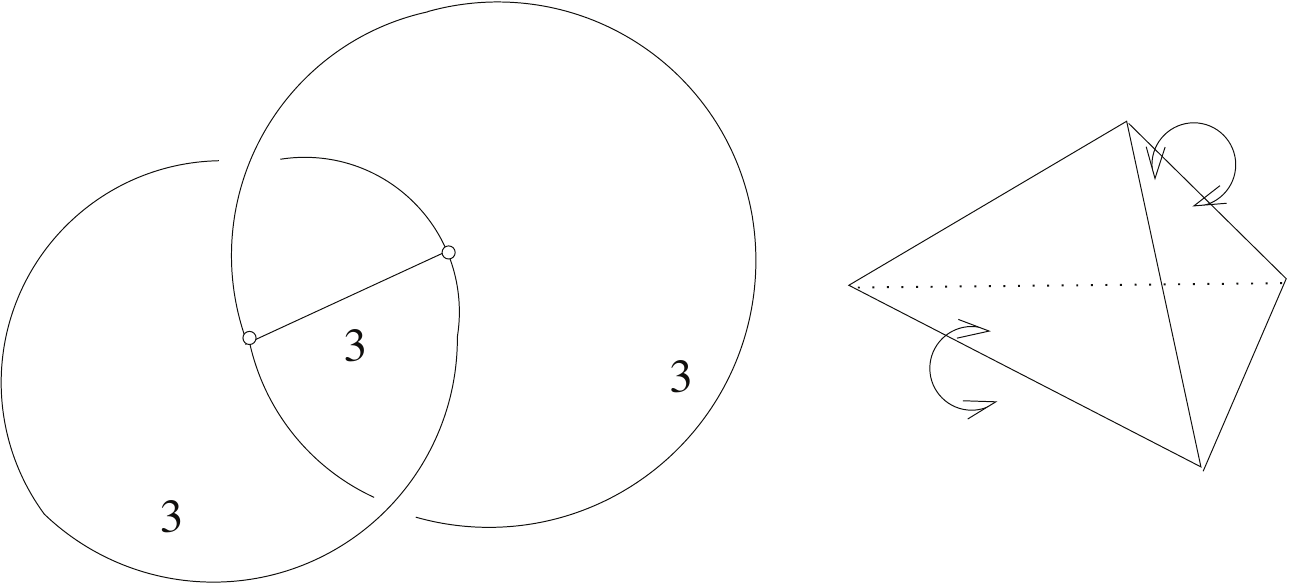} 
\end{array}$

\end{center} 

\caption{The handcuff graph and the construction of $3$-orbifold of Tillmann by pasting faces of  {a complete hyperbolic tetrahedron}.}
\label{fig:cubs}
\end{figure}

\subsubsection{Real projective structures on manifolds and orbifolds with ends.}
In general, the theory of geometric structures on manifolds with ends is not studied very well. 
We should try to obtain more results here and find what the appropriate conditions are. 
This question seems to be also related to how to make sense of the topological structures of ends 
in many other geometric structures such as ones modelled on symmetric spaces and so on. 
(See for example \cite{KL}, \cite{KLP}, and \cite{GKW}.) 



Given a vector space $V$, we let $P(V)$ denote the space obtained by taking the quotient space of \index{$P(\cdot)$}
$V -\{O\}$ under the equivalence relation
\[v\sim w \hbox{ for } v, w \in V  -\{O\}  \hbox{ iff } v = s w, \hbox{ for } s \in \bR -\{0\}.\] 
We let $[v]$ denote the equivalence class of $v \in V  -\{O\}$. 
For a subspace $W$ of $V$, we denote by $P(W)$ the image of $W-\{O\}$ under the quotient map, also said to be a {\em subspace}. \index{subspace} 

Recall that the projective linear group $\PGL(n+1, \bR)$ acts on $\bR P^{n}$, i.e.,
$P(\bR^{n+1})$, in a standard manner. 

Let $\mathcal{O}$ be a noncompact strongly tame $n$-orbifold where the orbifold boundary is not necessarily empty. 
{For convenience, we assume $n \geq 2$ in this article.}
\begin{itemize} 
\item A {\em real projective orbifold} is an orbifold with a geometric structure modelled on $(\bR P^n, \PGL(n+1, \bR))$. \index{real projective orbifold} 
(See \cite{Choi2006} and Chapter 6 of \cite{Cbook}.)
\item A real projective orbifold also has the notion of projective geodesics as given by local charts
and has a universal cover $\tilde{\mathcal{O}}$ where a deck transformation group $\pi_1({\mathcal{O}})$ acting on. 
\item The underlying space of $\mathcal{O}$ is homeomorphic to the quotient space $\tilde{\mathcal{O}}/\pi_1({\mathcal{O}})$.
\end{itemize} 
Let $\bR^{n+1 \ast}$ denote the dual of $\bR^{n+1}$.  \index{$\bR^{n+1 \ast}$}
Let $\bR P^{n \ast}$ denote the dual projective space $P(\bR^{n+1 \ast})$.  \index{$\bR P^{n \ast}$}
$\PGL(n+1, \bR)$ acts on $\bR P^{n \ast}$ by taking the inverse of the dual transformation. 
Then $h: \pi_{1}(\orb) \ra \PGL(n+1, \bR)$ has a dual representation $h^{\ast}: \pi_{1}(\orb) \ra \PGL(n+1, \bR)$ sending 
elements of $\pi_1(\orb)$ to the inverse of the dual transformation of $\bR^{n+1 \ast}$.

A {\em projective map} $f:\orb_{1}\ra\orb_{2}$ from an real projective orbifold $\orb_{1}$ to $\orb_{2}$ 
is a map so that for each $p \in \orb_{1}$, there is a chart $\phi_{1}: U_{1} \ra \bR P^{n}$, $p \in U_{1}$, and 
$\phi_{2}:U_{2} \ra \bR P^{n}$ where $f(p) \in U_{2}$ 
where $f(U_{1}) \subset U_{2}$ and $\phi_{2}\circ \phi_{1}^{-1}$ is a projective map. 

For an element $g \in \PGL(n+1, \bR)$, we denote
\begin{alignat}{2}
g \cdot [w] &:= [\hat g(w)] & \hbox{ for } [w] \in \bR P^n \hbox{ or } \nonumber \\ 
  & := [(\hat g^{T})^{-1}(w)] & \hbox{ for } [w] \in \bR P^{n \ast}
\end{alignat} 
where $\hat g$ is any element of $\SLpm$ mapping to $g$ and $\hat g^T$ the transpose of $\hat g$. \index{$\PGL(n+1, \bR)$}

The complement of a codimension-one subspace of $\bR P^n$ can be identified with an affine space 
$\bR^n$ where the geodesics are preserved. The group of affine transformations of $\bR^n$ are
{the restriction} to $\bR^n$ of the group of projective transformations of $\bR P^n$ fixing the subspace. 
We call the complement an {\em affine subspace}. It has a geodesic structure of a standard affine space.  \index{affine subspace}
A {\em convex domain} in $\bR P^n$ is a convex subset of an affine subspace.  \index{convex domain}
A {\em properly convex domain} in $\bR P^n$ is a convex domain contained in a precompact subset of an affine subspace.  \index{convex domain! properly convex} 
 
The important class of real projective structures are so-called convex ones
where any arc in $\mathcal{O}$ can be homotopied with endpoints fixed to a straight geodesic 
where the developing map $\dev$ is injective to $\bR P^n$ except possibly at the endpoints. 
If the open orbifold has a convex structure, it is covered by a convex domain $\Omega$ in $\bR P^n$. 
Equivalently, this means that the image of the developing map $\dev(\tilde{\mathcal{O}})$ for the universal cover $\tilde{\mathcal{O}}$ of $\mathcal{O}$ 
is a convex domain. 
$\mathcal{O}$ is projectively diffeomorphic 
to $\dev(\tilde{\mathcal{O}})/h(\pi_1(\mathcal{O}))$. In our discussions, since $\dev$ often is an imbedding, 
$\tilde{\mathcal{O}}$ will be regarded as an open domain in $\bR P^n$ and $\pi_1(\mathcal{O})$ a subgroup of
$\PGL(n+1, \bR)$ in such cases. 
This simplifies our discussions. 


Since this paper has many topics, we will outline this paper. 

\subsection{Outline.} 
In this paper, we will survey some results that we obtained for the ends of convex real projective orbifolds. 
There are three parts to expose the work here.
\begin{itemize}
\item[(I)] The preliminary review and examples. We discuss some parts on duality and finish our work on 
complete affine R-ends. 
(This corresponds to the present paper).
\item[(II)] We classify properly convex R-ends and T-ends when they satisfy the uniform middle eigenvalue conditions. 
(See \cite{End2}.)
\item[(III)] We classify nonproperly convex and non-complete affine but convex R-ends (NPNC-ends). (See \cite{End3}.)
\end{itemize} 
{We are currently writing a book \cite{book} generalizing the results in this article and the articles \cite{conv, endclass, End2, End3} simplifying 
many assumptions and results.}


Part I: 
In Section \ref{sec-prelim}, we go over basic definitions. We discuss ends of orbifolds, convexity, 
the Benoist theory on convex divisible actions, 
and so on. 

In Section \ref{sec-duality} we discuss the dual orbifolds of a given convex real projective orbifold. 

In Section \ref{sec-ends}, we will discuss the ends of orbifolds, covering most elementary aspects of the theory. 
For a properly convex real projective orbifold, 
the space of rays in each R-end give us a closed real projective orbifold of dimension $n-1$. 
The orbifold is convex. The universal cover can be a complete affine space 
or a properly convex domain or a convex domain that is neither. 

In Section \ref{sec-exend}, we discuss objects associated with R-ends, 
and examples of ends; horospherical ones, totally geodesic ones, 
and bendings of ends to obtain more general examples of ends. 

In Section  \ref{sub-horo}, 
we discuss horospherical ends. First, they are complete ends and 
have holonomy matrices with only unit norm eigenvalues 
and their end fundamental groups are virtually abelian. 
Conversely, a complete end in 
a properly convex orbifold has to be a horospherical end
or another type that we can classify. 

We begin the part II:

In Section \ref{sec-endth}, we start to study the properly convex R-end theory. 
First, we discuss the holonomy representation spaces.
Tubular actions and the dual theory of affine actions are discussed. We show that distanced actions
and asymptotically nice actions are dual. We explain that the uniform middle eigenvalue condition 
implies the existence of the distanced action. The main result here is the characterization of 
R-ends whose end fundamental groups
satisfy weakly uniform or uniform middle eigenvalue conditions. That is, they are either  lens-shaped R-ends 
or quasi-lens-shaped R-ends. Here, we will classify R-ends satisfying the uniform middle eigenvalue conditions. 
We also define the quasi-lens-shaped R-ends.

We go to Part III. 
In Section \ref{sec-notprop}, we discuss the R-ends that are NPNC. 
First, we show that the end holonomy group for an end $E$ will have an exact sequence 
\[ 1 \ra N \ra h(\pi_1(\tilde E)) \longrightarrow N_K \ra 1\] 
where $N_K$ is in the projective automorphism group $\Aut(K)$ of a properly convex compact set $K$, 
$N$ is the normal subgroup of elements mapping to the trivial automorphism of $K$,
and $K^o/N_K$ is compact. 
We show that $\Sigma_{\tilde E}$ is foliated by complete affine spaces of dimension $\geq 1$. 
We explain that an NPNC-end satisfying the {transverse weak middle eigenvalue} condition for NPNC-ends is a quasi-joined R-end.

\subsection{Acknowledgements} 
We thank David Fried for helping me to understand the issues with the distanced nature of the tubular actions and the duality,
and Yves Carri\`ere with the general approach to study the indiscrete cases for nonproperly convex ends. 
The basic Lie group approach of Riemannian foliations was a key idea here as well as the theory of 
Fried on distal groups. 
We thank Yves Benoist with some initial discussions on this topic, which were very helpful
for Section \ref{sec-endth} and thank Bill Goldman and Francois Labourie 
for discussions resulting in the proof of the distanced nature of the tubular actions. 
We thank Daryl Cooper and Stephan Tillmann 
for explaining their work and help and we also thank Micka\"el Crampon and Ludovic Marquis also. 
Their works obviously were influential 
here. The study was begun with a conversation with Tillmann at ``Manifolds at Melbourne 2006" 
and I began to work on this seriously from my sabbatical at Univ. Melbourne from 2008. 
Also, many technical points were made clear by discussions with Daryl Cooper at numerous occasions. 
We also thank Craig Hodgson and Gye-Seon Lee for working with me with many examples and their 
insights. The idea of radial ends comes from the cooperation with them.
We thank the referee for pointing out missing cases for the complete affine ends in 
{an earlier version.}

\section{Preliminaries} \label{sec-prelim}

In this paper, we will be using the smooth category{;} that is, we will be using smooth maps and smooth charts and so on. 
We explain the material in the introduction again.  

First, we explain what ends mean and define p-ends and p-end fundamental groups. 
Next we will define real projective structures on orbifolds.
We define convexity and explain the Benoist theory \cite{Ben1}, \cite{Ben2}, \cite{Ben3}, and \cite{Ben4}. 
In particular, the decomposition of properly convex orbifolds will be explained.

\subsection{Preliminary definitions} 

If $A$ is a domain of subspace of $\bR P^n$ or $\SI^n$, we denote by $\Bd A$ the topological boundary 
in the subspace. 
The closure $\clo(A)$ of a subset $A$ of $\bR P^n$ or $\SI^n$ is the topological closure in $\bR P^n$ or in $\SI^n$. 
Define $\partial A$ for a manifold or orbifold $A$ to be the {\em manifold or orbifold boundary}. 
Also, $A^o$ will denote the manifold or orbifold interior of $A$.

\subsection{Distances used} \label{sub-Haus}

\begin{definition}\label{defn-Haus} 
Let $\bdd$ denote the standard spherical metric on $\SI^n$ {\rm (}resp. $\bR P^n${\rm )}.
Given two compact subsets $K_1$ and $K_2$ of $\SI^n$ {\rm (}resp. $\bR P^n${\rm ),}
we define the spherical Hausdorff distance $\bdd_H(K_1, K_2)$ between $K_1$ and $K_2$ to be 
\[ \inf\{\eps > 0| K_2 \subset N_\eps(K_1), K_1 \subset N_\eps(K_2)  \}.\] 
The simple distance $\bdd(K_1, K_2)$ is defined as 
\[ \inf\{ \bdd(x, y)| x \in K_1, K_2 \}.\]
\end{definition}

Recall that every sequence of compact sets $\{K_i\}$ in $\SI^n$ (resp. $\bR P^n$) has
a convergent subsequence. We say that $\{K_{i}\}$ \hypertarget{term-geoconv}{{\em geometrically converges}} to a compact set $K$ if 
{$\bdd_{H}(K_{i}, K) \ra 0$} as $i \ra \infty$. 
Also, given a sequence $\{K_i\}$ of compact sets, 
$\{K_i\} \ra K$ for a compact set $K$ if and only if every sequence of points $x_i \in K_i$ has limit points in $K$ only
and every point of $K$ has a sequence of points $x_i \in K_i$ converging to it. 
(These facts can be found in some topology textbooks.)

\subsubsection{Topology of orbifolds and their ends.} \label{sub-ends}
An {\em orbifold} $\orb$ is a topological space with charts modeling open sets by 
quotients of Euclidean open sets or half-open sets 
by finite group actions and compatibly patched with one another. 
\hypertarget{term-bd}{The boundary $\partial \orb$} of an orbifold is defined as the set of points with only half-open sets as models. 
(These are often distinct {from the topological} boundary.)
Orbifolds are stratified by manifolds. 
Let $\orb$ denote an $n$-dimensional orbifold with finitely many ends 
where end neighborhoods are homeomorphic to closed  $(n-1)$-dimensional orbifolds times an open interval. 
We will require that $\orb$ is \hypertarget{term-sttame}{{\em strongly tame}}; that is, $\orb$ has a compact suborbifold $K$ 
so that $\orb - K$ is a disjoint union of end neighborhoods {diffeomorphic to} 
closed $(n-1)$-dimensional orbifolds multiplied by open intervals.
Hence $\partial \orb$ is a compact suborbifold. 

An {\em orbifold covering map} is a map so that for a given modeling open set as above, the inverse 
image is a union of modeling open sets that are quotients by subgroups of the original model.
We say that an orbifold is a manifold if it has a subatlas of charts with trivial local groups. 
We will consider good orbifolds {only, i.e., ones covered by simply connected manifolds. }
In this case, the universal covering orbifold $\torb$ is a manifold
with an orbifold covering map $p_\orb: \torb \ra \orb$. 
The group of deck transformations will be denote by $\pi_1(\orb)$ or $\bGamma$.
They act properly discontinuously on $\torb$ but not necessarily freely. 

By strong tameness, $\orb$ has only finitely many ends $E_1, \ldots, E_m$,
and each end has an end neighborhood diffeomorphic to $\Sigma_{E_i} \times (0, 1)$.
Here, $\Sigma_{E_{i}}$ up to diffeomorphism
may not be uniquely determined. 
(However, our radial or totally geodesic end-orbifolds will have determined $\Sigma_{\tilde E}$ up to 
diffeomorphisms.)

By an {\em exiting sequence} of sets $U_i$ of $\torb$, we mean 
a sequence of  neighborhoods $\{U_i\}$ so that 
$U_i \cap p_{\orb}^{-1}(K) \ne \emp$ for only finitely many indices for each compact subset $K$ of $\orb$. 

Each end neighborhood $U$ diffeomorphic to $\Sigma_{E} \times (0, 1)$ 
of an end $E$ lifts to a connected open set 
$\tilde U$ in $\torb$ 
where a subgroup of deck transformations $\bGamma_{\tilde U}$ acts on $\tilde U$ where 
$p_{\torb}^{-1}(U) = \bigcup_{g\in \pi_1(\orb)} g(\tilde U)$. Here, each component
$\tilde U$ is said to a \hypertarget{term-pendnh}{{\em proper pseudo-end neighborhood}}.
\begin{itemize} 
\item A {\em pseudo-end sequence} is an exiting sequence of proper pseudo-end neighborhoods 
$U_1 \supset U_2 \supset \cdots $. 
\item Two pseudo-end sequences are {\em compatible} if an element of one sequence is contained 
eventually in the element of the other sequence. 
\item A compatibility class of a proper pseudo-end sequence is called a {\em pseudo-end} of $\torb$.
Each of these corresponds to an end of $\orb$ under the universal covering map $p_{\orb}$.
\item For a pseudo-end $\tilde E$ of $\torb$, we denote by $\bGamma_{\tilde E}$ the subgroup $\bGamma_{\tilde U}$ where 
$U$ and $\tilde U$ is as above. We call $\bGamma_{\tilde E}$ a \hypertarget{term-pendfund}{{\em pseudo-end fundamental group}}.
\item A \hypertarget{term-pendn}{{\em pseudo-end neighborhood }} $U$ of a pseudo-end $\tilde E$ is a $\bGamma_{\tilde E}$-invariant open set containing 
a proper pseudo-end neighborhood of $\tilde E$. 
\end{itemize}
(From now on, we will replace ``pseudo-end'' with the abbreviation \hypertarget{term-pend}{``p-end''}.) 





\subsubsection{Real projective structures on orbifolds.} 
Recall the real projective space $\bR P^n$ 
is defined as 
$\bR^{n+1} - \{O\}$ under the quotient relation $\vec{v} \sim \vec{w}$ iff $\vec{v} = s\vec{w}$ for $s \in \bR -\{O\}$. 
We denote by $[x]$ the equivalence class of a nonzero vector $x$. 
The general linear group $\GLnp$ acts on $\bR^{n+1}$ and $\PGLnp$ acts faithfully on $\bR P^n$. 
Denote by $\bR_+ =\{ r \in \bR| r > 0\}$.
The {\em real projective sphere} $\SI^n$ is defined as the quotient of $\bR^{n+1} -\{O\}$ under the quotient relation 
$\vec{v} \sim \vec{w}$ iff $\vec{v} = s\vec{w}$ for $s \in \bR_+$. 
We will also use $\SI^n$ as the double cover of $\bR P^n$. 
Then $\Aut(\SI^n)$, isomorphic to the subgroup $\SLnp$ of $\GLnp$ of 
determinant $\pm 1$, double-covers $\PGLnp$. 
$\Aut(\SI^{n})$ acts as a group of projective automorphisms of $\SI^n$. 
A {\em projective map} of a real projective orbifold to another is a map that is projective by charts to $\bR P^n$. 
Let $\Pi: \bR^{n+1}-\{O\} \ra \bR P^n$ be a projection and let $\Pi':  \bR^{n+1}-\{O\} \ra \SI^n$ denote one for $\SI^n$. 
An infinite subgroup $\Gamma$ of $\PGLnp$ (resp. $\SLnp$) is {\em strongly irreducible} if every finite-index subgroup is irreducible. 
A {\em subspace} $S$ of $\bR P^n$ (resp. $\SI^n$) is the image of a subspace with the origin removed under the projection $\Pi$ (resp. $\Pi'$).

A cone $C$ in $\bR^{n+1} -\{O\}$ is a subset so that given a vector $x \in C$, 
$s x \in C$ for every $s \in \bR_+$. 
A {\em convex cone} is a cone that is a convex subset of $\bR^{n+1}$ in the usual sense. 
A {\em proper convex cone} is a convex cone not containing a complete affine line. 

A line in $\bR P^n$ or $\SI^n$ is an embedded arc in a $1$-dimensional subspace. 
A {\em projective geodesic} is an arc in a projective orbifold developing into a line in $\bR P^n$
or to a one-dimensional subspace of $\SI^n$. 
An affine subspace $A^n$ can be identified with the complement of a codimension-one subspace 
$\bR P^{n-1}$ so that the geodesic structures are same up to parameterizations. 
A {\em convex subset} of $\bR P^n$ is a convex subset of an affine subspace in this paper. 
A {\em properly convex subset} of  $\bR P^n$ is a precompact convex subset of an affine subspace. 
$\bR^n$ identifies with an open half-space in $\SI^n$ defined by a linear function on $\bR^{n+1}$. 
(In this paper an affine space is either embedded in $\bR P^n$ or $\SI^n$.)

An {\em $i$-dimensional complete affine subspace} is 
a subset of a projective orbifold projectively diffeomorphic to 
an $i$-dimensional affine subspace in some affine subspace $A^n$ of $\bR P^n$ or $\SI^n$. 

Again an affine subspace in $\SI^n$ is a lift of an affine space in $\bR P^n$, 
which is the interior of an $n$-hemisphere. 
Convexity and proper convexity in $\SI^n$ are defined in the same way as in $\bR P^n$.



The complement of a codimension-one subspace $W$ in $\bR P^n$ can be considered an affine 
space $A^n$ by correspondence 
\[[1, x_1, \dots, x_n] \ra (x_1, \dots, x_n)\] for a coordinate system where $W$ is given by $x_0=0$. 
The group $\Aff(A^n)$ of projective automorphisms acting on $A^n$ is identical with 
the group of affine transformations of form 
\[ \vec{x} \mapsto A \vec{x} + \vec{b} \] 
for a linear map $A: \bR^n \ra \bR^n$ and $\vec{b} \in \bR^n$. 
The projective geodesics and the affine geodesics agree up to parametrizations.

A subset $A$ of $\bR P^n$ or $\SI^n$ {\em spans} a subspace $S$ if $S$ is the smallest subspace containing $A$. 





We will consider an orbifold $\orb$ with a real projective structure: 
This can be expressed as 
\begin{itemize}
\item having a pair $(\dev, h)$ where 
$\dev:\torb \ra \bR P^n$ is an immersion equivariant with respect to 
\item the homomorphism $h: \pi_1(\orb) \ra \PGLnp$ where 
$\torb$ is the universal cover and $\pi_1(\orb)$ is the group of deck transformations acting on $\torb$. 
\end{itemize}
$(\dev, h)$ is only determined up to an action of $\PGLnp$ 
given by 
\[ g \circ (\dev, h(\cdot)) = (g \circ \dev, g h(\cdot) g^{-1}) \hbox{ for } g \in \PGLnp. \]
We will use only one pair where $\dev$ is an embedding for this paper and hence 
identify $\torb$ with its image. 
A {\em holonomy} is an image of an element under $h$. 
The {\em holonomy group} is the image group $h(\pi_1(\orb))$.  

The Klein model of the hyperbolic geometry is given as follows: 
Let $x_0, x_1, \dots, x_n$ denote the standard coordinates of $\bR^{n+1}$. 
\hypertarget{term-Klein-model}{
Let $B$ be the interior in $\bR P^n$  or $\SI^n$ of a standard ball} that is the image of the positive cone of 
$x_0^2 > x_1^2 + \dots + x_n^2$ in $\bR^{n+1}$.
Then $B$ can be identified with a hyperbolic $n$-space. The group of isometries of the hyperbolic space 
equals the group $\Aut(B)$ of projective automorphisms acting on $B$. 
Thus, a complete hyperbolic manifold carries a unique real projective structure and is denoted by $B/\Gamma$ for 
$\Gamma \subset \Aut(B)$.  
Actually, $g(B)$ for any $g \in \PGLnp$ will serve as a Klein model of the hyperbolic space, and 
$\Aut(gB) = g\Aut(B)g^{-1}$ is the isometry group.
(See \cite{Cbook} for details.) 

We also have a lift $\dev': \torb \ra \SI^n$ and $h': \pi_1(\orb) \ra \SLnp$, 
which are also called developing maps and holonomy homomorphisms. 
The discussions below apply to $\bR P^n$ and $\SI^n$ equally. 
This pair also completely determines the real projective structure on $\orb$. 
We will use this pair as $(\dev, h)$. 

Fixing $\dev$, we can identify $\torb$ with $\dev(\torb)$ in $\SI^n$ when $\dev$ is an embedding. 
This identifies $\pi_1(\orb)$ with a group of projective automorphisms $\Gamma$ in $\Aut(\SI^n)$.
The image of $h'$ is still called a {\em holonomy group}.

An orbifold $\orb$ is {\em convex} (resp. {\em properly convex} and {\em complete affine}) 
if $\torb$ is a convex domain (resp. a properly convex domain and an affine subspace.). 

A {\em totally geodesic hypersurface} $A$ in $\torb$ (resp. $\orb$) is 
a subspace of codimension-one where each point $p$ in $A$ has a neighborhood $U$ in $\torb$ (resp. $\orb$) 
with a chart $\phi$ so that $\phi|A$ has the image in a hyperspace. 


Let $\mathcal{A}$ denote the antipodal map $\SI^{n} \ra \SI^{n}$. 
Given a projective structure where $\dev: \torb \ra \bR P^n$ is an embedding to a properly convex 
open subset as in this paper, 
$\dev$ lifts to an embedding $\dev': \torb \ra \SI^n$ to an open domain $D$ without any pair of antipodal 
points. $D$ is determined up to $\mathcal{A}$. 

We will identify $\torb$ with $D$ or $\mathcal{A}(D)$ and 
$\pi_1(\orb)$ with $\bGamma$. 
Then $\bGamma$ lifts to a subgroup $\bGamma'$ of $\SLnp$ acting faithfully and discretely on $\torb$.
There is a unique way to lift so that {$D/\bGamma'$} is projectively diffeomorphic to {$\torb/\bGamma$.} 
Thus, we also define the p-end vertices of p-R-ends of $\torb$ as points in
the boundary of $\torb$ in $\SI^n$ from now on. (see  \cite{conv}.)  



\subsection{Convexity and convex domains}\label{subsec-conv}

\begin{proposition}\label{prop-projconv}
\begin{itemize}
\item A real projective $n$-orbifold is convex if and only if the developing map sends 
the universal cover to a convex domain in $\bR P^n$ {\rm (}resp. $\SI^n${\rm ).} 
\item A real projective $n$-orbifold is properly convex if and only if the developing map sends 
the universal cover to a precompact properly convex open domain in an affine patch of $\bR P^n$
{\rm (}resp. $\SI^{n}${\rm ).}
\item If a convex real projective $n$-orbifold is not properly convex and not complete affine, then
its holonomy is virtually reducible in $\PGL(n+1, \bR)$ {\rm (}resp. $\SLnp${\rm ).} In this case, $\torb$ is foliated 
by affine subspaces $l$ of dimension $i$ with the common boundary $\clo(l) - l$ equal to 
a fixed subspace  $\bR P^{i-1}_{\infty}$ {\rm (}resp. $\SI^{i-1}_{\infty}${\rm )} in $\Bd \torb$. 
\end{itemize}
\end{proposition}
\begin{proof} 
The first item is Proposition A.1 of \cite{psconv}. 
The second follows immediately. 
For the final item, 
a convex subset of $\bR P^n$ is a convex subset of an affine patch $A^n$, isomorphic to an affine space. 
Recall the double covering $\SI^{n} \ra \bR P^n$.
A convex open domain $D$ in $A^n$ lifts to a convex open domain $D' \subset H^{o}$ for 
a closed hemisphere $H$ in $\SI^{n}$.  If $D$ is not properly convex, the closure $\clo(D')$ must have 
a pair of antipodal points in $H$. They must be in $\Bd H$. 
A great open segment $l$ must connect this antipodal pair in $\Bd H$ and pass an interior point of $D'$. 
If a subsegment of $l$ is in $\Bd D'$, then $l$ is in a supporting hyperspace and $l$ does not pass 
an interior point of $D'$. Thus, $l \subset D'$. 
Hence, $D$ contains a complete affine line. 

Thus, $D$ contain a maximal complete affine subspace.  
Two such complete maximal affine subspaces do not intersect since otherwise a larger complete affine subspace of 
higher dimension is in $D$ by convexity. We showed in \cite{ChCh} that the maximal complete affine subspaces
foliated the domain.  (See also \cite{GV}.) The foliation is preserved under the group action 
since the leaves are lower-dimensional complete affine subspaces in $D$. 
This implies that the boundary of the affine subspaces is a lower dimensional subspace. 
These subspaces are preserved under the group action.
\end{proof} 


\begin{definition}\label{defn-join}
Given a convex set $D$ in $\bR P^n$, we obtain a connected cone $C_D$ in $\bR^{n+1}-\{O\}$ mapping to $D$,
determined up to the antipodal map. For a convex domain $D \subset \SI^n$, we have a unique domain $C_D \subset \bR^{n+1}-\{O\}$. 

A \hypertarget{term-join}{{\em join}} of two properly convex subsets $A$ and $B$ in a convex domain $D$ of $\bR P^n$ or $\SI^n$ is defined 
\[A \ast B := \{[ t x + (1-t) y]| x, y \in C_D,  [x] \in A, [y] \in B, t \in [0, 1] \} \]
where $C_D$ is a cone corresponding to $D$ in $\bR^{n+1}$. The definition is independent of the choice of $C_D$. 
\end{definition} 


\begin{definition}
Let $C_1, \dots, C_m$ be cone respectively in a set of independent vector subspaces $V_1, \dots, V_m$ of $\bR^{n+1}$. 
In general, a {\em sum} of convex sets $C_1, \dots, C_m$ in $\bR^{n+1}$ 
in independent subspaces $V_i$ is defined by
\[ C_1+ \dots + C_m := \{v | v = c_1+ \cdots + c_m, c_i \in C_i \}.\]
A \hypertarget{term-qjoin}{{\em strict join}} of convex sets $\Omega_i$ in $\SI^n$ (resp. in $\bR P^n$) is given as 
{
\[\Omega_1 \ast \cdots \ast \Omega_m := \Pi(C_1 + \cdots + C_m) \hbox{ (resp. } \Pi'(C_1 + \cdots + C_m)  ) \]}
where each $C_i-\{O\}$ is a convex cone with image $\Omega_i$ for each $i$. 
\end{definition}
(The join above does depend on the choice of cones.)

\subsubsection{The Benoist theory.} \label{sub-ben}

In late 1990s, Benoist more or less completed the theory 
of the divisible action as started by Benzecri, Vinberg, Koszul, Vey, and so on
in the series of papers \cite{Ben1}, \cite{Ben2}, \cite{Ben3}, \cite{Ben4}, \cite{Ben5}, \cite{Benasym}.
The comprehensive theory will aid us much in this paper. 

\begin{proposition}[Corollary 2.13 \cite{Ben3}]\label{prop-Benoist}  
Suppose that a discrete subgroup $\Gamma$ of $\SLn$ {\rm (}resp. $\PGL(n, \bR)${\rm ),} $n \geq 2$,
acts on a properly convex $(n-1)$-dimensional open domain $\Omega$ in $\SI^{n-1}$ {\rm (}resp, $\bR P^{n-1}${\rm )} 
so that $\Omega/\Gamma$ is compact. Then the following statements are equivalent. 
\begin{itemize} 
\item Every subgroup of finite index of $\Gamma$ has a finite center. 
 \item Every subgroup of finite index of $\Gamma$ has a trivial center. 
\item Every subgroup of finite index of $\Gamma$ is irreducible in $\SLn$ {\rm (}resp. in $\PGL(n, \bR)${\rm ).} 
That is, $\Gamma$ is strongly irreducible. 
\item The Zariski closure of $\Gamma$ is semisimple. 
\item $\Gamma$ does not contain an infinite nilpotent normal subgroup. 
\item $\Gamma$ does not contain an infinite abelian normal subgroup.
\end{itemize}
\end{proposition}
\begin{proof}
Corollary 2.13 of \cite{Ben3} considers $\PGL(n, \bR)$ and $\bR P^{n-1}$. 
However, the version for $\SI^{n-1}$ follows from this since we can always lift 
a properly convex domain in $\bR P^{n-1}$ to one $\Omega$ in $\SI^{n-1}$ and 
the group to one in $\SLn$ acting on $\Omega$. 
\end{proof}

The group with properties above is said to be the group with a {\em trivial virtual center}. 

\begin{theorem}[Theorem 1.1 of \cite{Ben3}] \label{thm-Benoist} 
Let $n-1 \geq 1$. 
Suppose that a virtual-center-free discrete subgroup $\Gamma$ of $\SLn$ {\rm (}resp. $\PGL(n, \bR)${\rm )}  acts on 
a strongly tame, properly convex $(n-1)$-dimensional open domain $\Omega \subset \SI^{n-1}$ so 
that $\Omega/\Gamma$ is compact.
Then every representation of a component of $\Hom(\Gamma, \SLn)$ {\rm (}resp. $\Hom(\Gamma, \PGL(n, \bR))${\rm ) }
containing the inclusion representation also acts on a properly convex $(n-1)$-dimensional open domain cocompactly. 
\end{theorem}
({When $\Gamma$ is a discrete group of hyperbolic isometries}  and $n=3$, Inkang Kim \cite{ink} proved this simultaneously.)

We call the group such as above theorem a {\em vcf-group}. By above Proposition \ref{prop-Benoist}, 
we see that every representation of the group acts irreducibly.




\begin{proposition}[Theorem 1.1. of Benoist \cite{Ben2}] \label{prop-Ben2} Assume $n \geq 2$. 
Let $\Sigma$ be a closed $(n-1)$-dimensional strongly tame, properly convex projective orbifold
and let $\Omega$ denote its universal cover in $\SI^{n-1}$ {\rm (}resp. $\bR P^{n-1}${\rm ).}  
Then 
\begin{itemize}
\item $\Omega$ is projectively diffeomorphic to the interior of 
a strict join $K:=K_1 * \cdots * K_{l_0}$ where $K_i$ is a properly convex open domain of dimension $n_i \geq 0$ in 
the subspace $\SI^{n_i}$ in $\SI^n$ {\rm (}resp. $\bR P^{n_i}$ in $\bR P^n${\rm )}.
$K_{i}$ corresponds to a convex cone $C_i \subset \bR^{n_i+1}$ for each $i$. 
\item $\Omega$ is the image of {$C_1 + \cdots + C_{l_{0}}$}. 
\item The fundamental group 
$\pi_1(\Sigma)$ is virtually isomorphic to a subgroup of 
$\bZ^{l_0-1} \times \Gamma_1 \times \cdots \times \Gamma_{l_0}$ for 
$(l_0 -1) + \sum_{i=1}^{l_{0}} n_i = n$, where $\Gamma_{i}$ is the image of 
the restriction map of $\pi_{1}(\Sigma)$ on $K_{i}$. 
\item $\pi_{1}(\Sigma)$ acts on $K$ cocompactly and discretely. 
\item 
Each $\Gamma_j$ acts on $K_j$ cocompactly, the Zariski closure $G_{j}$ is an irreducible reductive group, 
and $G_{j}$ acts trivially on $K_m$ for $m \ne j$. 
\item The subgroup corresponding to $\bZ^{l_0-1}$ acts trivially on each $K_j$.
\end{itemize} 
\end{proposition} 

\subsection{The flexibility of boundary} 

The following lemma gives us some flexibility of boundary.  
A smooth hypersurface embedded in a real projective manifold is called {\em strictly convex} if 
under a chart to an affine subspace, it maps to a hypersurface which is defined by a real function with positive Hessians 
at points of the hypersurface.

\begin{lemma} \label{lem-pushing} 
Let $M$ be a strongly tame or compact properly convex real projective orbifold with strictly convex $\partial M$.
We can modify $\partial M$ inward $M$ and the result bound 
a strongly tame or compact properly convex real projective orbifold $M'$ with strictly convex $\partial M'$
\end{lemma} 
\begin{proof}
Let $\Omega$ be a properly convex domain covering $M$. 
We may modify $M$ by pushing $\partial M$ inward. 
We take an arbitrary inward vector field defined on a tubular neighborhood of $\partial M$.
(See Section 4.4 of \cite{Cbook} for the definition of the tubular neighborhoods.)
We use the flow defined by them to modify $\partial M$. By the $C^{2}$-convexity condition, 
for sufficiently small change the image of $\partial M$ is still strictly convex and smooth. 
Let the resulting compact $n$-orbifold be denoted by $M'$.
$M'$ is covered by a subdomain $\Omega'$ in $\Omega$. 

Since $M'$ is a compact subset of $M$, 
$\Omega'$ is a properly imbedded domain in $\Omega$ 
and thus, $\Bd \Omega' \cap \Omega = \partial \Omega'$. 
$\partial \Omega'$ is a strictly convex hypersurface since so is $\partial M'$.
This means that $\Omega'$ is locally convex. 
{A locally convex subset in the Hilbert metric of $M^{o}$ is convex}  (see \cite{Thnote}).
Hence, $\Omega'$ is convex and hence is properly convex being a subset of a properly convex domain. 
So is $M'$. 
\end{proof}

Thus, by choosing one in the interior, we may assume without loss of generality that a strictly convex 
boundary component can be pushed out 
to a strictly convex boundary component. 

\subsection{A needed lemma} 
The following will be used in many of our papers. 

\begin{lemma} \label{lem-simplexbd}
Let $\orb$ be a strongly tame properly convex real projective orbifold. 
Let $\sigma$ be a convex domain in $\clo(\torb) \cap P$ for a subspace $P$. 
Then either $\sigma \subset  \Bd \torb$ or $\sigma^o$ is in $\torb$. 
\end{lemma} 
\begin{proof} 
Suppose that $\sigma^o$ meets 
$\Bd \tilde {\mathcal{O}}$ and is not contained in it entirely.  
Since the complement of $\sigma^{o}\cap \Bd \torb$ is a relatively open set in $\sigma^{o}$, 
we can find a segment $s \subset \sigma^o$ with a point $z$ so that a component $s_1$ of $s-\{z\}$ 
is in $ \Bd \tilde {\mathcal{O}}$ and the other component $s_2$ is disjoint from it. 
We may perturb $s$ in the subspace containing $s$ and 
$\bv_{\tilde E}$ so that the new segment $s' \subset \clo(\torb)$ 
meets $\Bd \tilde {\mathcal {O}}$ only in its interior point. 
This contradicts the fact that $\tilde {\mathcal{O}}$ is convex by Theorem A.2 of \cite{psconv}. 
\end{proof}

\section{The duality of real projective orbifolds}\label{sec-duality} 

The duality is a natural concept in real projective geometry and it will continue to play an important role 
in this theory as well. 

\subsection{The duality} 
We start from linear duality. Let $\Gamma$ be a group of linear transformations $\GL(n+1, \bR)$. 
Let $\Gamma^{\ast}$ be the {\em affine dual group} defined by $\{g^{\ast -1}| g \in \Gamma \}$. 

Suppose that $\Gamma$ acts on a properly convex cone $C$ in $\bR^{n+1}$ with the vertex $O$.
An open convex cone  $C^{\ast}$ in $\bR^{n+1, \ast}$  is {\em dual} to an open convex cone $C $ in $\bR^{n+1}$  if 
$C^{\ast} \subset \bR^{n+1 \ast}$ is the set of linear functionals taking positive values on $\clo(C)-\{O\}$.
{$C^{\ast}$ is a cone with the origin as the vertex.}
Note $(C^{\ast})^{\ast} = C$. 

Now $\Gamma^{\ast}$ will acts on $C^{\ast}$.
A {\em central dilatational extension} $\Gamma'$ of $\Gamma$ by $\bZ$ is given by adding a dilatation by a scalar 
$s \in \bR_+ -\{1\}$ for the set $\bR_+$ of positive real numbers. 
The dual $\Gamma^{\prime \ast}$ of $\Gamma'$ is a central dilatation extension of $\Gamma^{\ast}$. 
 Also, if $\Gamma'$ is cocompact on $C$ if and only if $\Gamma^{\prime \ast}$ is on $C^{\ast}$. 
 (See \cite{wmgnote} for details.)

 Given a subgroup $\Gamma$ in $\PGL(n+1, \bR)$, a {\em lift} in $\GL(n+1, \bR)$ is any subgroup that maps to $\Gamma$ bijectively.
 Given a subgroup $\Gamma$ in $\Pgl$, the dual group $\Gamma^{\ast}$ is the image in $\Pgl$ of the dual of 
 any linear lift of $\Gamma$. 

A properly convex open domain $\Omega$ in $P(\bR^{n+1})$ is {\em dual} to a properly convex open domain
$\Omega^{\ast}$ in $P(\bR^{n+1, \ast})$ if $\Omega$ corresponds to an open convex cone $C$ 
and $\Omega^{\ast}$ to its dual $C^{\ast}$. \hypertarget{term-dual}{We say that $\Omega^{\ast}$ is dual to $\Omega$}. 
We also have $(\Omega^{\ast})^{\ast} = \Omega$ and $\Omega$ is properly convex if and only if so is $\Omega^{\ast}$. 

We call $\Gamma$ a \hypertarget{term-divisible}{{\em divisible group}} if a central dilatational extension acts cocompactly on $C$.
$\Gamma$ is divisible if and only if so is $\Gamma^{\ast}$. 

We define $\SI^{n\ast} := \bR^{n+1 \ast} -\{O\}/\sim$ 
where $\vec{v} \sim \vec{w}$ iff $\vec{v} = s\vec{w}$ for $s \in \bR_+$. 

For an open properly convex subset $\Omega$ in $\SI^{n}$, the dual domain is defined as the quotient  in $\SI^{n\ast}$
of the dual cone of the cone $C_\Omega$ corresponding to $\Omega$. 
The dual set $\Omega^{\ast}$ is also open and properly convex
but the dimension may not change.  
Again, we have $(\Omega^{\ast})^{\ast} =\Omega$. 

Given a properly convex domain $\Omega$ in $\SI^n$ (resp. $\bR P^n$), 
we define the \hypertarget{term-augbd}{{\em augmented boundary}} of $\Omega$
\begin{multline} 
\Bd^{\Ag} \Omega  := \{ (x, h)| x \in \Bd \Omega, x \in h,  \\ 
h \hbox{ is an oriented supporting hyperplane of } \Omega \}  
\subset \SI^{n}\times \SI^{n\ast}.
\end{multline}
Define the projection \[\Pi_{\Ag}: \Bd^{\Ag} \Omega  \ra \Bd \Omega\] 
by $(x,h) \ra x$.
Each $x \in \Bd \Omega$ has at least one oriented supporting hyperspace.
An oriented hyperspace is an element of $\SI^{n \ast}$ since it is represented as a linear functional. 
Conversely, an element of $\SI^n$ represents an oriented hyperspace in $\SI^{n \ast}$. 
(Clearly, we can do this for $\bR P^{n}$ and the dual space $\bR P^{n \ast}$).

\begin{theorem} \label{thm-Serre} 
Let $A$ be a subset of $\Bd \Omega$. Let $A':= \Pi_{\Ag}^{-1}(A)$ be the subset of $\Bd^{\Ag}(A)$. 
Then $\Pi_\Ag| A': A' \ra A$ is a quasi-fibration. 
\end{theorem} 
\begin{proof}
We take a Euclidean metric on an affine space containing $\clo(\Omega)$. 
The supporting hyperplanes can be identified with unit vectors. 
Each fiber $\Pi_\Ag^{-1}(x)$ is 
a properly convex compact domain in a sphere of unit vectors through $x$. 
We find a continuous unit vector field to $\Bd \Omega$ by taking the center of mass of each fiber with respect 
to the Euclidean metric. This gives a local coordinate system on each fiber by giving the origin, and each fiber is a convex domain 
containing the origin. Then the quasi-fibration property is clear now.
\end{proof}

\begin{remark} 
We notice that for open properly convex domains $\Omega_1$ and $\Omega_2$ in $\SI^n$ 
(resp. in $\bR P^n$) we have 
\begin{equation}\label{eqn-dualinc}
\Omega_1 \subset \Omega_2 \hbox{ if and only if } \Omega_2^{\ast} \subset \Omega_1^{\ast}  
\end{equation}
\end{remark}

\begin{remark} 
Given a strict join $A\ast B$ for a properly convex compact $k$-dimensional domain $A$ and a properly convex compact $n-k-1$-dimensional domain $B$, 
\begin{equation}\label{eqn-dualjoin} 
(A \ast B)^{\ast} = A^{\ast} \ast B^{\ast}.
\end{equation} 
This follows from the definition and 
 realizing every linear functional as a sum of linear functionals in the direct-sum subspaces.
\end{remark}

An element $(x, h)$ is $\Bd^{\Ag} \Omega$ if and only if $x \in \Bd \Omega$ and $h$ is represented 
by a linear functional $\alpha_h$ so that $\alpha_h(y) > 0$ for all $y$ in the open cone $C$ corresponding to $\Omega$ and 
$\alpha_h(v_x) =0$ for a vector $v_x$ representing $x$. 

Let $(x, h) \in \Bd^{\Ag} \Omega$. 
Since the dual cone $C^{\ast}$ consists of all nonzero $1$-form $\alpha$ so that $\alpha(y) > 0$ for all $y \in \clo(C) - \{O\}$. 
Thus, $\alpha(v_x) > 0$ for all $\alpha \in C^{\ast}$ and $\alpha_y(v_x) = 0$. 
$\alpha_h \not\in C^{\ast}$ since $v_x \in \clo( C)-\{O\}$. 
But $h \in \Bd \Omega^{\ast}$ as we can perturb $\alpha_h$ so that it is in $C^{\ast}$. 
Thus, $x$ is a supporting hyperspace at $h \in \Bd \Omega^{\ast}$.
Hence we obtain a continuous map ${\mathcal{D}}_{\Omega}: \Bd^{\Ag} \Omega \ra \Bd^{\Ag} \Omega^{\ast}$. 
We define a \hypertarget{term-dualitymap}{{\em duality map}} 
\[ {\mathcal{D}}_{\Omega}: \Bd^{\Ag} \Omega \leftrightarrow \Bd^{\Ag} \Omega^{\ast} \] 
given by sending $(x, h)$ to $(h, x)$ for each $(x, h) \in \Bd^{\Ag} \Omega$. 



\begin{proposition} \label{prop-duality}
Let $\Omega$ and $\Omega^{\ast}$ be dual domains in $\SI^{n}$ and $\SI^{n \ast}$ {\rm (}resp. $\bR P^{n}$ 
and $\bR P^{n \ast}${\rm ).} 
\begin{itemize}  
\item[(i)] There is a proper quotient map $\Pi_{\Ag}: \Bd^{\Ag} \Omega \ra \Bd \Omega$
given by sending $(x, h)$ to $x$. 
\item[(ii)] Let a projective automorphism group 
$\Gamma$ acts on a properly convex open domain $\Omega$ if and only 
$\Gamma^{\ast}$ acts on $\Omega^{\ast}$.
\item[(iii)] There exists a duality map 
\[ {\mathcal{D}}_{\Omega}: \Bd^{\Ag} \Omega \leftrightarrow \Bd^{\Ag} \Omega^{\ast} \] 
which is a homeomorphism. 
\item[(iv)] Let $A \subset \Bd^{\Ag} \Omega$ be a subspace and $A^{\ast}\subset \Bd^{\Ag} \Omega^{\ast}$
be the corresponding dual subspace $\mathcal{D}_{\Omega}(A)$. A group $\Gamma$ acts on $A$ so that $A/\Gamma$ is compact 
if and only if $\Gamma^{\ast}$ acts on $A^{\ast}$ and $A^{\ast}/\Gamma^{\ast}$ is compact. 
\end{itemize} 
\end{proposition} 
\begin{proof} 
We will prove for $\bR P^n$ but the same proof works for $\SI^n$. 
(i) Each fiber is a closed set of hyperplanes at a point forming a compact set.
The set of supporting hyperplanes at a compact subset of $\Bd \Omega$ is closed. 
The closed set of hyperplanes having a point in a compact subset of {$\bR P^{n}$} is compact. 
Thus, $\Pi_{\Ag}$ is proper.  Clearly, $\Pi_{\Ag}$ is continuous, and it is an open map
since $\Bd^{\Ag} \Omega$ is given the subspace topology from $\bR P^n \times \bR P^{n \ast}$
with a box topology where $\Pi_{\Ag}$ extends to a projection.

(ii) Straightforward. (See Chapter 6 of \cite{wmgnote}.)

(iii) ${\mathcal{D}}_{\Omega}$ has the inverse map ${\mathcal{D}}_{\Omega^{\ast}}$. 

(iv) The item is clear from (ii) and (iii).  
\end{proof} 

\begin{definition}
The two subgroups $G_1$ of $\Gamma$ and $G_2$ of $\Gamma^{\ast}$ are {\em dual} if 
sending $g \ra g^{-1, T}$ gives us an isomorphism $G_1 \ra G_2$. 
\hypertarget{term-dualset}{A set in $A \subset \Bd \Omega$ is {\em dual}} to a set $B \subset \Bd \Omega^{\ast}$ if 
 $\mathcal{D}: \Pi_\Ag^{-1}(A) \ra \Pi_{\Ag}^{-1}(B)$ is a one-to-one and onto map. 
\end{definition}



\subsection{The duality of convex real projective orbifolds with strictly convex boundary} 

We have $\orb = \Omega/\Gamma$ for an open properly convex domain $\Omega$, 
the dual orbifold $\orb^{\ast} = \Omega^{\ast}/\Gamma^{\ast}$ is a properly convex real projective orbifold. 
The dual orbifold is well-defined up to projective diffeomorphisms. 




\begin{theorem}[Vinberg]  \label{thm-dualdiff} 
Let $\orb$ be a strongly tame properly convex real projective open orbifold. 
The dual orbifold $\orb^{\ast}$ is diffeomorphic to $\orb$.
\end{theorem}
We call the map the {\em Vinberg duality diffeomorphism}.
For an orbifold $\orb$ with boundary, 
the map is a diffeomorphism in the interiors $\orb^{o} \ra \orb^{\ast o}$.
Also,  $\mathcal{D}_{\orb}$ gives us the diffeomorphism $\partial \orb \ra \partial \orb^{\ast}$. 
(I conjecture that they form a diffeomorphism $\orb \ra \orb^{\ast}$ up to isotopies.) 




\section{Ends} \label{sec-ends}

In this section, we begin by explaining 
the R-ends. We classify them into three cases: complete affine ends, properly convex ends, and
nonproperly convex and not complete affine ends. 
We also introduce T-ends.

Suppose that $\mathcal{O}$ is a strongly tame
properly convex real projective orbifolds and a universal cover $\tilde{\mathcal{O}}$
with compact boundary $\partial \orb$ and some ends. 
(This will be the universal assumption for this paper.)

We recall the material in Section \ref{sub-ends}. 
An {\em  end neighborhood system} is  a sequence of open sets $U_1, U_2, \ldots$ in $\orb$ 
so that 
\begin{itemize} 
\item $U_i \supset U_{i+1}$ for each $i$ where 
\item each $U_i$ is a component of the complement of a compact subset in $\orb$ so that $\clo(U_i)$ is not compact, and 
\item given each compact set $K$ in $\orb$, $U_i \cap K \ne \emp$ for only finitely many $i$. 
\end{itemize} 
Two such sequences $\{U_1, U_2, \ldots \}$ 
and $\{U'_1, U'_2, \ldots \}$ are equivalent if 
\begin{itemize} 
\item for each $U_i$ we find $k$ so that $U'_j \subset U_i$ for $j > k$ 
and 
\item conversely for each $U'_i$ we find $k'$ such that $U_j \subset U'_i$ for $j > k'$. 
\end{itemize} 
An equivalence class of end neighborhoods is said to be an {\em end} of $\orb$.
A {\em neighborhood} of an end is one of the open set in the sequence in the equivalence class
of the end. 


A {\em end neighborhood system} of $\orb$ is the union of 
mutually disjoint collection of end neighborhoods for all ends 
where each end neighborhood is diffeomorphic to an orbifold times an interval. 


Given a component of such a system, the inverse image is a disjoint 
union of connected open sets. Let $\tilde U$ be a component. 
$\tilde U$ is  a \hyperlink{term-pendnh}{proper pseudo-end neighborhood}.
The subgroup $\bGamma_{\tilde U}$  acts on $\tilde U$  
so that $\tilde U/\bGamma_{\tilde U}$ is homeomorphic to the product end neighborhood. 
Note that any other component $\tilde U'$ is of form $\gamma(\tilde U)$ for $\gamma \in \bGamma - \bGamma_{\tilde U}$ 
and $\bGamma_{\tilde U'} = \gamma \bGamma_{\tilde U} \gamma^{-1}$ and $\bv_{\tilde U'} = \gamma(\bv_{\tilde U})$. 

Here, $\tilde U'$ may not be a neighborhood of an end of $\torb$ in the topological sense as in the cases of horospherical ends. 
$\bGamma_{\tilde U}$ is \hyperlink{term-pendfund}{the pseudo-end fundamental group}. 
Let $\tilde E$ denote the pseudo-end corresponding to a pseudo-end sequence containing $\tilde U$. 
Up to the $\bGamma$-action, there are only finitely many pseudo-end vertices and pseudo-end fundamental groups. 
For an end $E$, $\bGamma_{\tilde U}$ is well-defined up to conjugation by $\bGamma$, 
and we denote it by $\bGamma_{\tilde E}$. 
Its conjugacy class is more appropriately denoted $\bGamma_{\tilde E}$. 

A p-R-end may have more than one p-end vertex; however, the radial foliation structure will determine 
a p-R-end and vice versa.

We showed: 
\begin{lemma}\label{lem-inv}
Suppose that $\mathcal{O}$ is a strongly tame properly convex real projective orbifold 
with R- or T-ends and a universal cover $\tilde{\mathcal{O}}$.
Let $U'$ be  the inverse image of the union $U$ of mutually disjoint end neighborhoods. 
For a given component $\tilde U$ of $U'$, 
if $\gamma(\tilde U)\cap \tilde U \ne \emp$, then $\gamma(\tilde U)= \tilde U$ and 
$\gamma$ lies in the fundamental group $\bGamma_{E'}$ of the p-end $E'$ associated with $\tilde U$. 
\end{lemma}

We will assume that our real projective orbifold
$\orb$ is a strongly {tame convex orbifold with boundary and some  ends.}
We now define radial ends and totally geodesic ends. 

\begin{description} 
\item[Radial ends] 
A \hypertarget{term-rend}{radial end} is an end with an end neighborhood $U$ where each 
component $\tilde U$ of the inverse image $p_\orb^{-1}(U)$
has a foliation by properly embedded projective geodesics ending at a common point $\bv_{\tilde U} \in \bR P^n$. 
\begin{itemize} 
\item The {\em space of directions} of oriented projective geodesics through $\bv_{\tilde E}$ forms 
an $(n-1)$-dimensional real projective sphere. 
We denote it by \hypertarget{term-lsphere}{$\SI^{n-1}_{\bv_{\tilde E}}$} and call it a {\em linking sphere}. 
\item Let $\tilde \Sigma_{\tilde E}$ denote the space of equivalence classes of lines from $\bv_{\tilde E}$ in $\tilde U$
where two lines are regarded equivalent if they are identical near $\bv_{\tilde E}$. 
$\tilde \Sigma_{\tilde E}$ projects to a convex open domain in an affine space in  $\SI^{n-1}_{\bv_E}$
by the convexity of $\torb$. Then by Proposition \ref{prop-projconv} $\tilde \Sigma_{\tilde E}$ is projectively diffeomorphic to
\begin{itemize}
\item either {a complete affine space} $A^{n-1}$, 
\item a properly convex domain, 
\item or a convex but not properly convex 
and not complete affine domain in $A^{n-1}$. 
\end{itemize} 
\item The subgroup $\bGamma_{\tilde E}$, a p-end fundamental group, of $\bGamma$ 
fixes $\bv_{\tilde E}$ and  acts 
as a projective automorphism group on \hyperlink{term-lsphere}{$\SI^{n-1}_{\bv_E}$}. 
Thus, the quotient $\tilde \Sigma_{\tilde E}/\bGamma_{\tilde E}$ admits a real projective 
structure of one dimension lower. 
\item We denote by $\Sigma_{\tilde E} $ the real projective $(n-1)$-orbifold $\tilde \Sigma_{E}/\bGamma_{E}$. 
Since we can find a transversal orbifold $\Sigma_{\tilde E}$ to the radial foliation in 
a p-end neighborhood for each p-R-end $\tilde E$ of $\mathcal{O}$,
it lifts to a transversal surface $\tilde \Sigma_{\tilde E}$ in $\tilde U$. 
\item We say that a p-R-end $\tilde E$ is  {\em convex} (resp. {\em properly convex}, and {\em complete affine}) 
if $\tilde \Sigma_{\tilde E}$ is convex  (resp. properly convex, and complete affine). 
\end{itemize}

Thus, an R-end is either
\begin{itemize}
\item[(i)] complete affine (CA), 
\item[(ii)] properly convex (PC), or 
\item[(iii)] convex but not properly convex and not complete affine (NPNC). 
\end{itemize}

\item[Totally geodesic ends] 
An {\em end} $E$ is {\em totally geodesic} if a closed end neighborhood $U$ in $\orb$ compactifies
 by adding a totally geodesic suborbifold $\Sigma_E$
 in an ambient orbifold, homeomorphic to $\Sigma_E \times I$ for a closed interval $I$.
The choice of the compactification is said to be the {\em totally geodesic end structure}. 
Two compactifications are equivalent if some respective neighborhoods are projectively diffeomorphic. 
(One can see in \cite{cdcr1} two inequivalent ways to compactify for real projective elementary annulus.)
If $\Sigma_E$ is properly convex, then the end is said to be {\em properly convex}. 
\end{description}
We call such an end simply a {\em T-end}. 
Note that the end orbifolds are determined for R- or T-ends up to diffeomorphisms. 





Let us end with a few facts. 

\begin{lemma}\label{lem-endcover} 
Let $U$ be {a p-R-end neighborhood with smooth $\Bd U \cap \torb$ of}  a p-end vertex $\bv_{\tilde E}$. 
Suppose that $\Bd U \cap \torb \cap l$ is a unique point for 
each open great segment $l$ from endpoints $\bv_{\tilde E}$ {in a direction of} $\tilde \Sigma_{\tilde E}$. 
Then $\Bd U \cap \torb$ covers a compact orbifold
and $p_{\orb}(U)$ is diffeomorphic to $\Sigma_{\tilde E} \times \bR$. 
\end{lemma} 

\begin{proposition}\label{prop-endvertex} 
Suppose that $\mathcal{O}$ is a strongly tame convex real projective orbifold.
Let $U'$ be an R-end neighborhood in $\orb$. 
Let $\tilde U$ be $p_{\orb}^{-1}(U')$ as above with $E'$ the p-end in $\torb$ associated with a component $U$ of $\tilde U$.
Then
\begin{itemize}
\item the closure of each component of  $\tilde U$ 
contains the p-end vertex $\bv_{E'}$ of the leaf of radial foliation in $\tilde U$, lifted from $U$.
\item There exists a unique one for each component $U_1$ of $\tilde U$ associated with a p-R-end $E'$ of $\torb$. 
\item The subgroup of $h(\pi_1(\mathcal{O}))$ acting on $U_1$ 
or fixing the p-R-end vertex $\bv_{E'}$ is precisely the subgroup $\bGamma_{E'}$.
\end{itemize} 
\end{proposition}





\subsection{Horospherical domains, lens domains, lens-cones, and so on.} 




A {\em segment} is a convex arc in a $1$-dimensional subspace of $\bR P^{n}$ or $\SI^{n}$. 
We will denote it by $\ovl{xy}$ if $x $ and $y$ are endpoints. It is uniquely determined by $x$ and $y$ if $x$ and $y$ are not antipodal. 
In the following, all the sets are required to be inside an affine subspace $A^n$ and its closure either in $\bR P^n$ or $\SI^n$. 
\begin{itemize}
\item An $n$-dimensional submanifold $L$ of $A^n$ is said to be a {\em horoball} if it is strictly convex, 
and the boundary $\partial L$ is diffeomorphic to $\bR^{n-1}$ and $\Bd L - \partial L$ is a single point. 
The boundary $\partial L$ is said to be a {\em horosphere}. 
\item An $n$-dimensional subdomain $L$ of $A^{n}$ is a {\em lens} if $L$ is a convex domain 
and $\partial L$ is a disjoint union of two smoothly strictly convex embedded open
$(n-1)$-cells $\partial_+ L$ and $\partial_{-} L$. 
\item A {\em cone} is a bounded domain $D$ in an affine patch with a point in the boundary, called an {\em end vertex} $v$
so that every other point $x \in D$ has an open segment $\ovl{vx}^{o} \subset D$.
\item The cone $\{p\} \ast L$ over a lens-shaped domain $L$ in $A^n$, $p\not\in \clo(L)$ 
is a {\em lens-cone} if it is a convex domain and satisfies
\begin{itemize}
\item $\{p\}\ast L = \{p\} \ast \partial_+ L$ for one boundary component $\partial_+ L$ of $L$ and 
\item every segment from $p$ to $\partial_{+} L$ meets the other boundary component
$\partial_{-} L$ of $L$ at a unique point.
\end{itemize} 
\item As a consequence each line segment from $p$ to $\partial_{+} L$ is transversal to $\partial_{+} L$.
$L$ is called the {\em lens} of the lens-cone. 
(Here different lenses may give the identical lens-cone.)
Also, $\{p\}\ast L - \{p\}$ is a manifold with boundary $\partial_{+} L$.  
\item Each of  two boundary components of $L$ is called a {\em top} or {\em bottom} {hypersurface}
depending on whether it is further away from $p$ or not. The top component is denoted by $\partial_+ L$
and the bottom one by $\partial_{-}L$. 
\item An $n$-dimensional subdomain $L$ of $A^{n}$ {is a {\em generalized lens}} if $L$ is a convex domain 
and $\partial L$ is a disjoint union of a smoothly strictly convex embedded open
$(n-1)$-cell $\partial_- L$ and an imbedded open $(n-1)$-cell $\partial_{+} L$, which is not necessarily smooth. 
\item 
A cone  $\{p\} \ast L$ is said to be a {\em generalized lens-cone}  if 
\begin{itemize} 
\item $\{p\} \ast L  = \{p\} \ast \partial_+ L, p \not\in \clo(L)$ and 
\item every segment from $p$ to $\partial_{+} L$ meets $\partial_{-} L$ at a unique point.
\end{itemize} 
A lens-cone will of course be considered a generalized lens-cone. 
\item  We again define the top hypersurface  and the bottom one as above.
They are denoted by $\partial_{+} L$ and $\partial_{-} L$ respectively. 
$\partial_+ L$ can be nonsmooth. $\partial_-L$ is required to be smooth. 
\item A {\em totally-geodesic submanifold} is a convex domain in a hyperspace. 
A {\em cone-over} a totally-geodesic submanifold $D$ is a union of all segments 
with  one endpoint a point $x$ not in the hyperspace
and the other endpoint in $D$. We denote it by $\{x\} \ast D$. 
\end{itemize}


We apply these to ends: 

\begin{description}
\item [Horospherical R-end] A p-R-end of $\torb$ is {\em horospherical} if it 
has a horoball in $\torb$ as a p-end neighborhood, or {equivalently a convex} open p-end neighborhood $U$ in $\torb$ 
so that $\Bd U \cap \torb = \Bd U - \{v\}$ for a boundary fixed point $v$. 
\item[Lens-shaped R-end]  A p-R-end $\tilde E$ is {\em lens-shaped} (resp. {\em totally geodesic cone-shaped},
{\em generalized-lens-shaped}\,),  if it has a p-end neighborhood that is a lens-cone (resp. a cone over a totally-geodesic domain, 
 a generalized lens-cone p-end neighborhood). Here we require that 
the p-end fundamental group $\bGamma_{\tilde E}$ acts on a lens of the lens-cone if 
$\tilde E$ is lens-shaped and acts on a generalized lens of the generalized lens-cone
if $\tilde E$ is generalized-lens-shaped. 
\end{description}

An end of $\orb$ is of {\em lens-shaped}  (resp. {\em totally geodesic cone-shaped},
{\em generalized-lens-shaped}\,) if 
the corresponding p-end is lens-shaped  (resp. {\em totally geodesic cone-shaped},
{\em generalized-lens-shaped}\,).  

Let the p-R-end $\tilde E$ have a p-end neighborhood of form $\{p\} \ast L -\{p\}$ that is 
a generalized lens-cone $p \ast L$ over a generalized lens $L$
where $\partial (p\ast L -\{p\}) = \partial_+L$ and $\bGamma_{\tilde E}$ acts on $L$. 
Notice that $\{p\}\ast L -\{p\} - \partial_{-} L$ has two components since $\partial_{-} L$ meets each segment from $p$ 
to $\partial_{+} L$ at a unique point. 
A {\em concave p-end neighborhood} of $\tilde E$ is the open p-end neighborhood 
$\{p\}\ast L -\{p\} - L$ in $\torb$ contained in 
a p-R-end neighborhood in $\tilde{\mathcal{O}}$.
As it is defined, such a p-R-end neighborhood always exists for a generalized-lens-shaped p-R-end. 



Additionally, we define: 
\begin{itemize} 
\item A {\em quasi-lens cone} is a properly convex cone of form $p\ast S$ for a strictly convex open smooth 
hypersurface $S$ 
so that
\begin{itemize}
\item $\partial (\{p\} \ast S -\{p\})= S, p \in \clo(S) - S$ hold,
\item the space of directions from $p$ to $S$ is a properly convex domain in $\SI^{n-1}_p$, and 
\item each segment from $p$ in these directions meets $S$ transversally.
\end{itemize} 
\item  An R-end $\tilde E$ is {\em quasi-lens-shaped} if it has a p-end neighborhood that is 
a quasi-lens cone. 
\end{itemize} 

Finally, we define: 
\begin{description} 
\item[Lens-shaped T-end] A p-T-end $\tilde E$ of $\torb$ is {\em lens-shaped} if it has a lens neighborhood $L$ in 
an ambient orbifold of $\torb$ where 
\[\clo(\partial L)-\partial L= \Bd \partial L \subset P\] for the hyperspace $P$ containing 
$\tilde \Sigma_{\tilde E}$ and $L /\bGamma_{\tilde E}$ is is a compact orbifold.
\end{description}

\begin{remark} 
The reason we introduce a generalized-lens or horospherical condition is that we need this conditions to hold when we deform
the orbifolds keeping it properly convex. (See \cite{conv} and \cite{book}.) 
However, a more general condition can be used sometimes. 
\end{remark}

Horospherical ends will be shown to be cusp ends later.  Strictly speaking, we distinguish
between these two types.  See Section \ref{sub-horo}.


\subsection{Examples of ends} \label{sec-exend}

We will present some examples here, which we will fully justify later.

An \hypertarget{term-ellipsoid}{{\em ellipsoid}} is a subset in an affine space defined as a zero locus of a positive definite quadratic polynomial in term of 
the affine coordinates. 
Recall the Klein model of hyperbolic geometry: It is a pair $(B, \Aut(B))$ where 
$B$ is the interior of an ellipsoid in $\bR P^n$ or $\SI^{n}$ and 
$\Aut(B)$ is the group of projective automorphisms of $B$. 
$B$ has a Hilbert metric which in this case is the hyperbolic metric times a constant. 
Then $\Aut(B)$ is the group of isometries of $B$. (See Section \ref{sec-endth}.)

From hyperbolic manifolds, we obtain some examples of ends. 
Let $M$ be a complete hyperbolic manifold with cusps. 
$M$ is a quotient space of the interior $B$ of an ellipsoid in $\bR P^n$ or $\SI^n$
under the action of a discrete subgroup $\bGamma$ of $\Aut(B)$. 
Then horoballs are p-end neighborhoods of the horospherical R-ends. 

Suppose that $M$ has totally geodesic embedded surfaces $S_1,.., S_m$ homotopic to the ends.  
\begin{itemize}
\item We remove the outside of $S_j$s to obtain a properly convex 
real projective orbifold $M'$ with totally geodesic boundary.
\item Each $S_i$ corresponds to a disjoint union of totally geodesic domains $\bigcup_{j \in J} \tilde S_{i, j}$
for a collection $J$. For each $\tilde S_{i, j} \subset B$, a group 
$\Gamma_{i,j}$ acts on it where $\tilde S_{i, j}/\Gamma_{i, j}$ is 
a closed orbifold projectively diffeomorphic to $S_i$. 
\item Then $\Gamma_{i, j}$ fixes a point $p_{i,j}$ outside the ellipsoid by taking the dual 
point of $\tilde S_{i, j}$ outside the ellipsoid. 
\item Hence, we form the cone 
$M_{i, j} := \{p_{i, j}\} \ast \tilde S_{i, j}$. 
\item We obtain the quotient 
$M_{i, j}/\Gamma_{i, j} -\{p_{i, j}\}$ and identify $\tilde S_{i, j}/\Gamma_{i, j}$ to $S_{i,j}$ in 
$M'$ to obtain the examples of real projective manifolds with R-ends. 
\item $(p_{i, j}\ast \tilde S_{i, j})^{o}$ is a p-R-end neighborhood and the end is a totally geodesic R-end. 
\end{itemize}
This orbifold is called the {\em hyper-ideal extension} of the hyperbolic manifolds as real projective manifolds.
By Proposition \ref{prop-lensend}, these will be lens-shaped. 



Following proposition shows that many lens-shaped R-ends exist. 

\begin{proposition}\label{prop-lensend}
Suppose that $M$ is a strongly tame convex real projective orbifold.
Let $\tilde E$ be a p-R-end of $M$ with admissible $\pi_{1}(\tilde E)$. 
Suppose that 
\begin{itemize}
\item the holonomy group of the p-end fundamental group $\pi_{1}(\tilde E)$ 
\begin{itemize} 
\item is generated by the homotopy classes of finite orders or
\item  is generated by elements whose holonomy transformation fixes the p-end vertex with eigenvalues $1$ 
\end{itemize} 
and 
\item $\tilde E$ has a $\pi_{1}(\tilde E)$-invariant $n-1$-dimensional totally geodesic 
properly convex domain $D$ in a p-end neighborhood and not containing the p-end vertex in the closure of $D$.
\end{itemize} 
Then the p-R-end $\tilde E$ is lens-shaped.
\end{proposition}
\begin{proof} 
Let $\tilde M$ be the universal cover of $M$ in $\SI^n$. It will be sufficient to prove for this case. 
$\tilde E$ is a p-R-end of $M$, and $\tilde E$ has a $\pi_{1}(\tilde E)$-invariant $n-1$-dimensional totally geodesic 
properly convex domain $D$. 
{Since $D$ projects to $\tilde \Sigma_{\tilde E}$, 
the domain $D$ is transverse to radial rays from $\bv_{\tilde E}$. }

Under the first assumption, 
since the end fundamental group $\Gamma_{\tilde E}$
is generated by elements of finite order, 
the eigenvalues of the generators  corresponding to the p-end vertex  $\bv_{\tilde E}$ equal $1$
and hence every element of the end fundamental group has $1$ as the eigenvalue at
$\bv_{\tilde E}$.

Now assume that the holonomy of the elements of the end fundamental groups
fixe the p-end vertices with eigenvalues equal to $1$. 

Then the p-end neighborhood $U$ can be chosen to be 
the open cone over the totally geodesic domain with vertex $\bv_{\tilde E}$. 
$U$ is projectively diffeomorphic  to the interior of a properly convex cone in 
an affine subspace $A^n$. 
The end fundamental group acts on 
$U$ as a discrete linear group of determinant $1$. 
The theory of convex cones applies, and using the level sets of 
the Koszul-Vinberg function, we obtain 
a smooth convex one-sided neighborhood $N$ in $U$ 
(see Lemmas 6.5 and 6.6 of Goldman \cite{wmgnote}).

We obtain an one-sided neighborhood in the other side as follows: 
We take $R(N)$ for by a reflection $R$ fixing each point of the hyperspace containing $\tilde \Sigma$ and the p-end vertex.
{Then we choose a diagonalizable transformation $\mathcal{D}$ fixing the p-end vertex and 
every point of $D$ so that the image $\mathcal{D}R(F)$ is in $\torb$
for a fundamental domain $F$ of $R(N)$. We obtain $\mathcal{D}R(N) \subset \torb$. 
Thus, $N \cup \mathcal{D}R(N)$ is the lens we needed.}
The cone $ \bv_{\tilde E}  \ast (N\cup DR(N))$ is the lens-cone neighborhood for $\tilde E$. 
\end{proof}

A more specific example is below. 
Let $S_{3,3,3}$ denote the $2$-orbifold with base space homeomorphic to a $2$-sphere and 
three cone-points of order $3$. 
The orbifolds satisfying the following properties are the example of Tillman as described in \cite{conv} and 
the hyperbolic Coxeter $3$-orbifolds based on 
an ideal $3$-polytopes of dihedral angles $\pi/3$. (See Choi-Hodgson-Lee \cite{CHL}.)

\begin{proposition} \label{prop-lensauto}
Let $\mathcal{O}$ be a  strongly tame convex real projective $3$-orbifold with R-ends where each end orbifold 
is homeomorphic  to a sphere $S_{3,3,3}$ or a disk with three silvered edges and three corner-reflectors of orders $3, 3, 3$. 
Assume that holonomy group of $\pi_{1}(\orb)$ is strongly irreducible. 
Then the orbifold has only lens-shaped R-ends or horospherical R-ends. 
\end{proposition}
\begin{proof}
Again, it is sufficient to prove this for the case $\torb \subset \SI^3$. 
Let $\tilde E$ be a p-R-end corresponding to an R-end whose end orbifold is diffeomorphic to $S_{3, 3, 3}$. 
It is sufficient to consider only $S_{3,3,3}$ since it double-covers the disk orbifold. 
Since $\bGamma_{\tilde E}$ is generated by finite order elements fixing a p-end vertex $\bv_{\tilde E}$,
every holonomy element has eigenvalue equal to $1$ at $\bv_{\tilde E}$.
Take a finite-index free abelian group $A$ of rank two in $\bGamma_{\tilde E}$. 
Since $\Sigma_E$ is convex, 
a convex projective torus $T^2$ covers $\Sigma_E$ finitely.
Therefore, $\tilde \Sigma_{\tilde E}$ is projectively diffeomorphic either to 
\begin{itemize}
\item a complete affine space or 
\item the interior of a properly convex triangle or
\item a half-space 
\end{itemize} 
by the classification of 
convex tori found in many places including \cite{wmgnote} and \cite{BenNil} and 
Proposition \ref{prop-projconv}. 
Since there exists a holonomy automorphism of order $3$ fixing a point of 
$\tilde \Sigma_{\tilde E}$, 
it cannot be a quotient of a half-space with a distinguished foliation by lines.
Thus, the end orbifold admits a complete affine structure or is a quotient of a properly convex triangle. 

 Suppose that $\Sigma_{\tilde E}$ has a complete affine structure.
{ Since $\lambda_{\bv_{\tilde E}}(g) = 1$} for all $g \in \bGamma_{\tilde E}$, 
 the only possibility from Theorem \ref{thm-comphoro} is when $\bGamma_{\tilde E}$ is virtually nilpotent and 
 and we have a cusp p-end for $\tilde E$. 
 
Suppose that $\Sigma_{\tilde E}$ has a properly convex  open triangle $T'$ as its universal cover. 
$A$ acts with an element $g'$ with the largest eigenvalue $>1$ and the smallest eigenvalue $<1$ as a transformation 
in $\SL_\pm(3, \bR)$ the group of projective automorphisms at $\SI^2_{\bv_{\tilde E}}$.
As an element of $\SL_{\pm}(4, \bR)$, we have $\lambda_{\bv_{\tilde E}}(g') = 1$ and the product 
of the remaining eigenvalues is $1$, the corresponding the largest and smallest eigenvalues are $> 1$ and $< 1$. 
Thus, an element of $\SL_{\pm}(4, \bR)$, $g'$ fixes $v_1$ and $v_2$ other than $\bv_{\tilde E}$
in directions of vertices of $T'$.
Since $\bGamma_{\tilde E}$ has an order three elements exchanging the vertices of $T'$, 
there are three fixed points of an element of $A$ different from $\bv_{\tilde E}, \bv_{\tilde E-}$. 
By commutativity, there is a properly convex compact triangle $T\subset \SI^{3}$  with these three fixed points 
where $A$ acts on. Hence, $A$ is diagonalizable over the reals.

We can make any vertex of $T$ to be an attracting fixed point of an element of $A$. 
Each element $g \in \Gamma_{\tilde E}$ conjugates elements of $A$ to $A$. 
Therefore $g$ sends the attracting fixed points of elements of $A$ to those of elements of $A$. 
Hence $g(T) = T$ for all $g \in \Gamma_{\tilde E}$. 

Each point of the edge $E$ of $\clo(T)$ is an accumulation point of an orbit of $A$
by taking a sequence $g_{i}$ so that the associated eigenvalues $\lambda_{1}(g_{i})$ 
and $\lambda_{2}(g_{i})$ are going to $+\infty$ while $\log|\lambda_{1}(g)/\lambda_{2}(g)|$is bounded. 
Since $\lambda_{\bv_{\tilde E}} = 1$, writing every vector 
as a linear combinations of vectors in the direction of the four vectors, 
this follows. 
Hence $\Bd T \subset \Bd \torb$ and $T \subset \clo(\orb)$. 

If $T^{o} \cap \Bd \orb \ne \emp$, then $T \subset \Bd \orb$ by Lemma \ref{lem-simplexbd}. 
Then each segment from $\bv_{\tilde E}$ ending in $\Bd \orb$ is in 
the direction of $\clo(\Sigma_{\tilde E})= T'$. It must end at $T$. 
Hence, $\torb = (T \ast \bv_{\tilde E})^{o}$, an open tetrahedron $\sigma$. 
Since the holonomy group acts on it, we can take a finite index group fixing each vertices of $\sigma$.
Thus, the holonomy group is virtually reducible.  This is a contradiction. 

Therefore, $T \subset \orb$ as $T \cap \Bd \orb = \emp$. 
We have a totally geodesic R-end and by Proposition \ref{prop-lensend}, the end is lens-shaped. 
(See also \cite{End2}.) 
\end{proof}


\begin{example}[Lee's example] \label{exmp-Lee}
Consider the Coxeter orbifold $\hat P$ with the underlying space on a polyhedron $P$ 
with the combinatorics of a cube with all sides mirrored and
all edges given order $3$ but vertices removed. 
By the Mostow-Prasad rigidity and the Andreev theorem, 
the orbifold has a unique complete hyperbolic structure. 
There exists a six-dimensional space of real projective structures on it 
as found in \cite{CHL} where one has a projectively fixed fundamental domain 
in the universal cover.

There are eight ideal vertices of $P$ corresponding to eight ends of $\hat P$. 
Each end orbifold is a $2$-orbifold based on a triangle with edges mirrored 
and vertex orders are all $3$. Thus, each end has a neighborhood homeomorphic 
to the $2$-orbifold multiplied by $(0, 1)$.
We can characterize them by a real-valued invariant.  
Their invariants are related when we are working on the restricted deformation space. 
(They might be independent  in the full deformation space as M. Davis and R. Green observed. )

This applies to S. Tillman's example. See our other paper on the mathematics archive \cite{conv} for details.

\end{example}


The following construction is called ``bending'' and was investigated by Johnson and Millson
\cite{JM}. 

\begin{example}[Bending] \label{exmp-bending} 
Let $\orb$ have the usual assumptions. 
We will concentrate on an end and not take into consideration of the rest of the orbifold. 
Certainly, the deformation given here may not extend to the rest.
(If the totally geodesic hypersurface exists on the orbifold, the bending does extend to the rest.
{S. Ballas and L. Marquis recently} found such examples for a link complement.)

Suppose that $\orb$ is {an oriented hyperbolic} manifold with a hyper-ideal end $E$. 
Then $E$ is a totally geodesic R-end with a p-R-end $\tilde E$.
Let  the associated orbifold 
$\Sigma_{E}$ for $E$ of $\orb$ be a closed $2$-orbifold and 
let $c$ be a  two-sided simple closed geodesic in $\Sigma_{E}$. 
Suppose that $E$ has an open end neighborhood $U$ in $\orb$ diffeomorphic to $\Sigma_{E} \times (0,1)$
with totally geodesic $\Bd U$ diffeomorphic to $\Sigma_E$.
Let $\tilde U$ be a p-end neighborhood in $\torb$ corresponding to $\tilde E$
bounded by $\tilde \Sigma_{\tilde E}$ covering $\Sigma_E$.
Then $U$ has a radial foliation whose leaves lifts to radial lines in $\tilde U$ from $\bv_{\tilde E}$. 

Let $A$ be an annulus in $U$ diffeomorphic to $c \times (0, 1)$, foliated by leaves of the radial foliation of $U$.
Now a lift $\tilde c$ of $c$ is in an embedded disk $A'$, covering $A$.
Let $g_c$ be the deck transformation corresponding to $\tilde c$ and $c$. 
Suppose that $g_{c}$ is orientation-preserving. 
Since $g_{c}$ is a hyperbolic isometry of the Klein model, 
the holonomy $g_c$ is conjugate 
to a diagonal matrix with entries $\lambda,  \lambda^{-1}, 1, 1, $ where $\lambda > 1$
and the last $1$ corresponds to the vertex $\bv_{\tilde E}$.  
We take an element $k_b$ of $\SLf$ of form in this system of coordinates
\begin{equation}\label{eqn-bendingm} 
\left(
\begin{array}{cccc}
1           &       0              & 0   & 0  \\
 0          &       1              & 0  & 0 \\ 
 0           &      0              & 1  & 0 \\ 
 0           &      0               & b & 1   
\end{array}
\right)
\end{equation}
where $b \in \bR$. 
$k_b$ commutes with $g_c$. 
Let us just work on the end $E$. 
We can ``bend'' $E$ by $k_b$: 

Now, 
$k_{b}$ induces a diffeomorphism $\hat k_{b}$ of an open neighborhood of $A$ in $U$ to another one 
of $A$ since $k_{b}$ commutes with $g_{c}$. 
We can find tubular neighborhoods $N_1$ of $A$ in $U$
and $N_{2}$ of $A$. 
We choose $N_{1}$ and $N_{2}$ so that they are
diffeomorphic by a projective map $\hat k_b$. 
Then we obtain two copies $A_1$ and $A_2$ of $A$ 
by completing $U - A$. 

Give orientations on $A$ and $U$. 
Let $N_{1,-}$ denote the left component of $N_{1} - A$ and 
let $N_{2, +}$ denote the right component of $N_{2} -A$. 


We take a disjoint union $(U - A) \coprod N_1 \coprod N_2$ and 
quotient it by identifying the copy of $N_{1,-}$ in $N_{1}$ with $N_{1, -}$ in $U- A $ by the identity map
and identify the copy of $N_{2,+}$ in $N_2$ with $N_{2,+}$ in $U - A$ by the identity also.
We glue back $N_1$ and $N_2$ by 
the real projective diffeomorphism $\hat k_b$ of a neighborhood of $N_1$ to that of $N_2$. 
Then $N_{1} - (N_{1,-} \cup A)$ is identified with $N_{2,+}$
and $N_{2} - (N_{2, +} \cup A)$ is identified with $N_{1, -}$. 
We obtain a new manifold. 

For sufficiently small $b$, we see that the end is still lens-shaped.
and it is not a totally geodesic R-end. (This follows since the condition of being 
a lens-shaped R-end is an open condition. 
See  \cite{End2}.)




For the same $c$, 
let $k_s$ be given by 
\begin{equation}\label{eqn-bendingm2} 
\left(
\begin{array}{cccc}
s           &       0              & 0   & 0  \\
 0          &       s              & 0  & 0 \\ 
 0           &      0              & s  & 0 \\ 
 0           &       0            & 0 & 1/s^3   
\end{array}
\right)
\end{equation}
where $s \in \bR_+$. 
These give us bendings of the second type. (We talked about this in \cite{conv}.) 
For $s$ sufficiently close to $1$, the property of being lens-shaped is preserved 
and being a totally geodesic R-end. 
(However, these will be understood by cohomology.)

If $s \lambda < 1$ for the maximal eigenvalue $\lambda$ of a closed curve $c_1$
meeting $c$ odd number of times, we have that the holonomy along $c_1$ has the attracting 
fixed point at $\bv_{\tilde E}$. This implies that we no longer have lens-shaped R-ends if we have 
started with a lens-shaped R-end. 

\end{example} 










\subsection{Characterization of complete R-ends} \label{sub-horo}


The results here  overlap with the results of Crampon-Marquis \cite{CM2} and 
Cooper-Long-Tillman \cite{CLT}. However, the results are of different direction than theirs 
and were originally conceived before their papers appeared.  We also make use of 
Crampon-Marquis \cite{CM2}. 

Let $\tilde E$ be a p-R-end.
A \hypertarget{term-mec}{{\em middle eigenvalue condition}} for a p-end fundamental group $\pi_{1}(\tilde E)$  holds if 
for each $g\in \pi_{1}(\tilde E) -\{\Idd\}$
the largest norm $\lambda_{1}(g)$ of eigenvalues of $g$ is strictly larger than 
the eigenvalue $\lambda_{\bv_{\tilde E}}(g)$ associated with p-end vertex $\bv_{\tilde E}$. 

Given an element $g \in \Gamma_{\tilde E}$, 
let $\left(\tilde \lambda_{1}(g), \dots, \tilde \lambda_{n+1}(g)\right)$ be the $(n+1)$-tuple of the eigenvalues
where we repeat each eigenvalue with the multiplicity given by the characteristic polynomial.
The {\em multiplicity} of a norm of an eigenvalues of $g$ is 
the number of times the norm occurs among the $(n+1)$-tuples of norms 
\[\left(|\tilde \lambda_{1}(g)|, \dots, |\tilde \lambda_{n+1}(g)|\right).\]

A  \hypertarget{term-wmec}{{\em weak middle eigenvalue condition}} for a p-end fundamental group $\pi_{1}(\tilde E)$  holds if 
for each $g\in \pi_{1}(\tilde E)$
if the eigenvalue $\lambda_{\bv_{\tilde E}}$ associated with the p-end vertex $\bv_{\tilde E}(g)$
has the largest norm of all eigenvalues of elements of $\pi_{1}(\tilde E)$, 
then the norm of the eigenvalue must have multiplicity $\geq 2$. 

Recall the parabolic subgroup of the isometry group $\Aut(B)$ of 
the hyperbolic space $B$ for an $(i_{0}+1)$-dimensional Klein model $B \subset \SI^{i_{0}+1}$
fixing a point $p$ in the boundary of $B$. 
Such a group is isomorphic to $\bR^{i_0}$ and is Zariski closed. 

Let $E$ be an $i_0$-dimensional \hyperlink{term-ellipsoid}{ellipsoid} 
containing the point $\bv$ in a subspace $P$ of dimension $i_{0}+1$ in 
$\SI^{n}$.  Let $\Aut(P)$ denote the group of projective automorphisms of $P$, 
and let $\SL_{\pm}(n+1, \bR)_{P}$ the subgroup of $\SL_{\pm}(n+1, \bR)$ acting on $P$. 
Let $r_{P}:\SL_\pm(n+1, \bR)_{P} \ra \Aut(P)$ denote the restriction homomorphism $g \ra g| P$. 
An \hypertarget{term-ppgroup}{{\em $i_{0}$-dimensional partial parabolic subgroup}} is 
one mapping  under $R_{P}$ isomorphically
to a parabolic subgroup of $\Aut(P)$ acting freely on $E -\{\bv\}$, fixing $\bv$.

Suppose now that $\torb \subset \bR P^{n}$. 
Let $P'$ denote a subspace of dimension $i_{0}+1$ containing an $i_{0}$-dimensional ellipsoid $E'$ 
containing $\bv$. 
Let $\PGL(n+1,\bR)_{P'}$ denote the subgroup of $\PGL(n+1, \bR)$ acting on $P'$. 
Let $R_{P'}: \PGL(n+1, \bR)_{P'} \ra \Aut(P')$ denote the restriction $g\mapsto g|P'$.
An {\em $i_{0}$-dimensional partial parabolic subgroup} is one mapping  under $R_{P'}$ isomorphically
to a parabolic subgroup of $\Aut(P')$ acting  freely on $E'-\{\bv\}$, fixing $\bv$. 
When $i_{0} = n-1$, we will drop the ``partial'' from the term ``partial parabolic group''.

 An {\em $i_0$-dimensional cusp group} is a finite extension of a projective conjugate  
of a discrete cocompact subgroup of a group of an $i_{0}$-dimensional partial parabolic subgroup. 
If the horospherical neighborhood with the p-end vertex $\bv$ 
has the p-end fundamental group that is a discrete $(n-1)$-dimensional cusp group, 
then we call the p-end to be {\em cusp-shaped}. 

Our first main result classifies CA p-R-ends. We need the notion of NPNC-ends that 
will be explained in Section \ref{sec-notprop}.

Given a cusp p-R-end, the p-end fundamental group $\Gamma_{\bv}$ acts on a p-end neighborhood $U$
and $\Gamma_{\bv}$ is a subgroup of an $(n-1)$-dimensional parabolic group $\mathcal{H}_{\bv}$. 
Since $\mathcal{H}_{\bv} \cap \Gamma_{\bv}$ is cocompact in $\mathcal{H}_{\bv}$, 
we let $F$ be a compact fundamental domain  of $\mathcal{H}_{\bv}$
with respect to $\mathcal{H}_{\bv} \cap \Gamma_{\bv}$. 
Hence, we can form 
\[V:=\bigcap_{g\in \mathcal{H}_{\bv}} g(U) = \bigcap_{g\in F} g(U). \]
By definition $V$ is $\mathcal{H}_{\bv}$-invariant. Since $F$ is compact and bounded 
in $\SL_\pm(n+1, \bR) $ (resp. in $\PGL(n+1, \bR)$), the set $V$ is not empty by the second part of the equality. 
$\partial V$ is an orbit of an $(n-1)$-dimensional parabolic group $\mathcal{H}_{\bv}$. 
Hence, $V$ is a cusp p-end neighborhood of $\tilde E$. 
Thus, a cusp R-end is horospherical.
Conversely, a horospherical R-end is a cusp R-end by Theorem \ref{thm-affinehoro}
under some assumption on $\orb$ itself. 

\begin{corollary}\label{cor-cusphor} 
Let $\mathcal O$ be a strongly tame properly convex real projective $n$-orbifold.
Let $E$  be an R-end of its universal cover $\torb$. 
Then $E$ is horospherical R-end if and only if $\tilde E$ is a cusp R-end. \qed
\end{corollary} 

\begin{theorem}\label{thm-mainaffine} 
Let $\mathcal O$ be a strongly tame properly convex real projective $n$-orbifold.
Let $\tilde E$  be a p-R-end of its universal cover $\torb$. 
Let $\bGamma_{\tilde E}$ denote the end fundamental group. 
Then $\tilde E$ is a complete affine p-R-end if and only if $\tilde E$ is a cusp p-R-end
or an NPNC-end with fibers of dimension $n-2$ by altering the p-end vertex.
Furthermore,  if $\tilde E$ is a complete affine p-R-end and 
$\Gamma_{\tilde E}$ satisfies the weak middle eigenvalue condition, 
then $\tilde E$ is a cusp p-R-end. 
\end{theorem}
\begin{proof}
Theorem \ref{thm-comphoro} is the forward direction.
Since a cusp p-R-end is horospherical, 
Theorem \ref{thm-affinehoro} 
{implies the converse in this case. 
Also, an NPNC-end with fibers of dimension $n-2$ becomes a complete affine end 
when we change the end vertex.}
The last statement follows by Corollary \ref{cor-cusp}. 
\end{proof}






\begin{theorem}[Horosphere] \label{thm-affinehoro} 
Let $\mathcal O$ be a strongly tame  properly convex real projective $n$-orbifold. 
Let $\tilde E$  be a horospherical R-end of its universal cover $\torb$, $\torb \subset \SI^n$ {\rm (}resp. $\subset \bR P^n${\rm )} 
and $\bGamma_{\tilde E}$ denote the p-end fundamental group. 
\begin{itemize}
\item[(i)] The space \hyperlink{term-lsphere}{$\tilde \Sigma_{\tilde E} = R_{\bv_{\tilde E}}(\torb) \subset \SI^{n-1}_{\bv_{\tilde E}}$}
equivalence classes of 
lines segments from the endpoint $\bv_{\tilde E}$ in $\torb$ forms a complete affine space of dimension $n-1$.
\item[(ii)] The norms of eigenvalues of $g \in \bGamma_{\tilde E}$
are all $1$.
\item[(iii)] $\bGamma_{\tilde E}$ is virtually abelian and a finite index subgroup is in 
a conjugate of an $(n-1)$-dimensional parabolic subgroup of $\SO(n, 1)$ of rank $n-1$ in $\SL_\pm(n+1, \bR)$ or $\PGL(n+1, \bR)$.
And hence $\tilde E$ is cusp-shaped.
\item[(iv)] For any compact set $K' \subset \orb$ inside a horospherical end neighborhood, 
$\orb$ contains a {horospherical end neighborhood} disjoint from $K'$. 
\item[(v)] A p-end vertex of a horospherical p-end cannot be an endpoint of a segment in $\Bd \tilde{\mathcal{O}}$. 
\end{itemize}
\end{theorem}
\begin{proof} 
We will prove for the case $\torb \subset \SI^n$. The $\bR P^n$-version follows from this. 
Let $U$ be a horospherical p-end neighborhood with the p-end vertex $\bv_{\tilde E}$.
The space of great segments from the p-end vertex passing $U$ forms a convex subset $\tilde \Sigma_{\tilde E}$ 
of a complete affine space $\bR^{n-1} \subset \SI^{n-1}_{\tilde E}$ by Proposition \ref{prop-projconv}.
The space 
covers an end orbifold $\Sigma_{\tilde E}$ 
with the discrete group $\pi_1(\tilde E)$ acting as a discrete subgroup $\Gamma'_{\tilde E}$ of
the projective automorphisms so that $\tilde \Sigma_{\tilde E}/\Gamma'_{\tilde E}$
 is projectively isomorphic to $\Sigma_{\tilde E}$.

(i) By Proposition \ref{prop-projconv}, one of the following three happens: 
\begin{itemize}
\item $\tilde \Sigma_{\tilde E}$ is properly convex.
\item $\tilde \Sigma_{\tilde E}$ 
is foliated by complete affine spaces of dimension $i_0$, $1 \leq i_{0} < n-1$, with the common boundary sphere of dimension $i_0-1$ and 
the space of the leaves forms a properly open convex subset $K^o$ of $\SI^{n-i_0-1}$. 
$\bGamma_{\tilde E}$ acts on $K^o$ cocompactly but perhaps not discretely.
\item $\tilde \Sigma_{\tilde E}$ is a complete affine space. 
\end{itemize}
We aim to show that the first two cases do not occur. 

Suppose that we are in the second case and 
 $1 \leq i_0 \leq n-2$. This implies that $\tilde \Sigma_{\tilde E}$ is foliated by complete affine spaces of dimension $i_0 \leq n-2$. 

For each element $g$ of $\bGamma_{\tilde E}$, a complex or negative eigenvalue of $g$ in $\bC - \bR_+$ 
cannot have a maximal or minimal absolute value different from $1$ and $\lambda_{\bv_{\tilde E}}(g)$: Otherwise 
$\{g^m(x)| m \in \bZ\}$ for a generic point $x$ of $U$ has accumulation points on 
a great circle or a pair of antipodal points disjoint from $\{\bv_{\tilde E}\}$. 
We take the convex hull of the orbits  in $\torb$ of $\{g^m(x)| m \in \bZ\}$. 
This is not properly convex, a contradiction. 
Thus, the largest and the smallest absolute value
eigenvalues of $g$ are positive. 


Since $\bGamma_{\tilde E}$ acts on a properly convex subset $K$ of dimension $\geq 1$,  an element 
$g$ has a norm of an eigenvalue $>1$ and a norm of eigenvalue $< 1$ by Proposition 1.1 of \cite{Ben5}
as a projective automorphism on the great sphere spanned by $K$.
Hence, we obtain the largest norm of eigenvalues and the smallest one of $g$ in $\Aut(\SI^n)$ both different from $1$. 
Therefore, let $\lambda_1(g) >1$ be the greatest norm of the eigenvalues of $g$  and 
$\lambda_2(g)< 1$ be the smallest norm  of the eigenvalues of $g$ as an element of $\SL_\pm(n+1, \bR)$. 
Let $\lambda_{\bv_{\tilde E}}(g) >0$ 
be the eigenvalue of $g$ associated with $\bv_{\tilde E}$. 
These are all positive. 
The possibilities for $g$ are as follows 
\begin{alignat*}{3}
 \lambda_1(g)  && \, = \lambda_{\bv_{\tilde E}}(g)  && \,> \lambda_2(g), \\ 
  \lambda_1(g)  && \,  > \lambda_{\bv_{\tilde E}}(g)  &&  \, > \lambda_2(g), \\
\lambda_1(g) && \, > \lambda_2(g)  && \,  = \lambda_{\bv_{\tilde E}}(g). 
\end{alignat*}
In all cases, at least one of the largest norm or the smallest norm is different from $\lambda_{\bv_{\tilde E}}(g)$. 
By the paragraph immediately above, this norm is realized by a positive eigenvalue. 
We take $g^{n}(x)$ for a generic point $x \in U$. As $n \ra \infty$ or $n \ra -\infty$, the sequence limits to a point $x_{\infty}$ 
in $\clo(U)$ distinct from $\bv_{\tilde E}$.  
Also, $g$ fixes a point $x_\infty$, and $x_{\infty}$ has a different positive eigenvalue from $\lambda_{\bv_{\tilde E}}(g)$.
{As $x_\infty \not \in \torb$,} it should be $x_\infty =\bv_{\tilde E}$ by the definition of the horoballs.
This is a contradiction. 

The first possibility is also shown not to occur similarly. Thus, 
 $\tilde \Sigma_{\tilde E}$ is a complete affine space. 

(ii) If $g \in \bGamma_{\tilde E}$ has a norm of eigenvalue different from $1$, then 
we can apply the second and the third paragraphs above to obtain a contradiction. 
We obtain $\lambda_j (g)= 1$ for each norm $\lambda_j(g)$ of eigenvalues of $g$ for every $g \in \bGamma_{\tilde E}$.


(iii) Since $\tilde \Sigma_{\tilde E}$ is a complete affine space, 
$\tilde \Sigma_{\tilde E}/\bGamma_{\tilde E}$ is a complete affine manifold with the norms of eigenvalues of holonomy matrices all equal to $1$
where $\bGamma'_{\tilde E}$ denotes the affine transformation group corresponding to $\bGamma_{\tilde E}$. 
(By D. Fried \cite{Fried86}, this implies that $\pi_1(\tilde E)$ is virtually nilpotent.) 
The conclusion follows by Proposition 7.21 of \cite{CM2} (related to Theorem 1.6 of \cite{CM2}): 
By the proposition, we see that $\bGamma_{\tilde E}$ is in a conjugate of $\SO(n, 1)$
and hence acts on an $(n-1)$-dimensional ellipsoid fixing a unique point. 
Since a horosphere has a Euclidean metric invariant under the group action, 
the image group is in a Euclidean isometry group. 
Hence, the group is virtually abelian by the Bieberbach theorem. 

(iv) We can choose an exiting sequence of p-end horoball neighborhoods $U_i$
where a cusp group acts. We can consider the hyperbolic spaces to understand this. 

(v) Suppose that $\Bd \torb$ contains a segment $s$ ending at the p-end vertex $\bv_{\tilde E}$. 
Then $s$ is on an invariant hyperspace of $\bGamma_{\tilde E}$.
Now conjugating $\bGamma_{\tilde E}$ into an $(n-1)$-dimensional parabolic subgroup $P$ of $\SO(n,1)$ fixing 
$(1,-1,0,\dots, 0) \in \bR^{n+1}$ by say an element $h$ of $\SL_{\pm}(n+1,\bR)$. 
By simple computations, 
we can find a sequence $g_i \in h\bGamma_{\tilde E}h^{-1} \subset P$ so that $\{g_i(h(s))\}$ geometrically converges to 
a great segment. Thus, for a sequence $h^{-1} g_{i} h \in \bGamma_{\tilde E}$, $h^{-1}g_{i}h(s)$ geometrically 
converges to a great segment in $\clo(\torb)$. 
This contradicts the proper convexity of $\torb$. 
\end{proof}





We will now show the converse of Theorem \ref{thm-affinehoro}.



\begin{theorem}[Complete affine]\label{thm-comphoro} 
Let $\orb$ be a strongly tame  properly convex $n$-orbifold. 
Suppose that $\tilde E$ is a complete-affine p-R-end of its universal cover $\torb$ in $\SI^n$ or 
in $\bR P^n$. Let $\bv_{\tilde E} \in \SI^n$  be the p-end vertex with the p-end fundamental group 
$\bGamma_{\tilde E}$. Then
\begin{itemize} 
\item[(i)]  
\begin{itemize}
\item $\bGamma_{\tilde E}$ is virtually unipotent where all norms of eigenvalues of elements equal $1$, or 
\item $\bGamma_{\tilde E}$ is virtually nilpotent where 
\begin{itemize}
\item each $g \in \bGamma_{\tilde E}$ has at most two norms of the eigenvalues, 
\item at least one $g \in \bGamma_{\tilde E}$ has two norms,  and 
\item if $g \in \bGamma_{\tilde E}$ has two distinct norms of the eigenvalues, 
the norm of $\lambda_{\bv_{\tilde E}}(g)$ has a multiplicity one.
\end{itemize}
\end{itemize}
\item[(ii)] In the first case, $\bGamma_{\tilde E}$ is horospherical, i.e., cuspidal. 
\end{itemize} 
\end{theorem} 

\begin{proof} 
{The proof here is for $\SI^n$ but it implies the $\bR P^n$-version. }
Using Selberg's lemma, we may choose a torsion-free finite-index subgroup. 
We may assume without loss of generality that $\Gamma$ is torsion-free 
since we only need to prove the theorem for a finite index subgroup. 
Hence, $\Gamma$ does not fix a point in $\tilde \Sigma_{\tilde E}$. 


(i) Since $\tilde E$ is complete affine, $\tilde \Sigma_{\tilde E} \subset \SI^{n-1}_{\bv_{\tilde E}}$ is identifiable with $\bR^{n-1}$. 
$\bGamma_{\tilde E}$ induces $\Gamma'_{\tilde E}$ in $\Aff(\bR^{n-1})$
that are of form $x \mapsto Mx + b$ where $M$ is a linear map $\bR^{n-1} \ra \bR^{n-1}$
and $b$ is a vector in $\bR^{n-1}$. 
For each $\gamma \in \bGamma_{\tilde E}$, 
\begin{itemize}
\item let $\gamma_{\bR^{n-1}}$ denote this affine transformation, and
\item we denote by $\hat L(\gamma_{\bR^{n-1}})$ the linear part of the affine transformation $\gamma_{\bR^{n-1}}$. 
\item Let $\vec{v}(\gamma_{\bR^{n-1}})$ denote  the translation vector. 
\end{itemize}

At least one eigenvalue of ${\hat L}(\gamma_{\bR^{n-1}})$ is $1$ 
since $\gamma$ acts without fixed point on $\bR^{n-1}$.
(See \cite{KS}.)
Now, ${\hat L}(\gamma_{\bR^{n-1}})$ has a maximal 
invariant vector  subspace $A$ of $\bR^{n-1}$ where all norms of the eigenvalues are $1$.

Suppose that $A$ is a proper $\gamma$-invariant vector subspace of $\bR^{n-1}$. 
Then $\gamma_{\bR^{n-1}}$ acts on the affine space $\bR^{n-1}/A$ 
as an affine transformation with the linear parts without a norm of eigenvalue equal to $1$.
Hence, $\gamma_{\bR^{n-1}}$ has a fixed point in $\bR^{n-1}/A$, and 
$\gamma_{\bR^{n-1}}$ acts on an affine subspace $A'$ parallel to $A$.

A subspace $H$ containing ${\bv_{\tilde E}}$ corresponds to the direction of $A'$ from $\bv_{\tilde E}$.
The union of segments with endpoints $\bv_{\tilde E}, \bv_{\tilde E-}$ in the directions in $A' \subset \SI^{n-1}_{\bv_{\tilde E}}$
is an open hemisphere of dimension $n$. 
Let $H^{+}$ denote this space where $\Bd H^{+} \ni \bv_{\tilde E}$ holds.
Since $\bGamma_{\tilde E}$ acts on $A'$, it follows that
$\bGamma_{\tilde E}$ acts on $H^{+}$.
Then $\gamma$ has at most two eigenvalues associated with $H^{+}$ one of which is {$\lambda_{\bv_{\tilde E}}(\gamma)$}
and the other is to be denoted $\lambda_{+}(\gamma)$.
We obtain $\lambda_{+}(\gamma)$ by writing 
\[ \gamma = 
\left(
\begin{array}{ccc}
 \lambda_{+}(\gamma) {\hat L}(\gamma_{\bR^{n-1}})  &  \lambda_{+}(\gamma) \vec{v}(\gamma_{\bR^{n-1}}) & 0\\
 0 &  \lambda_{+}(\gamma) &  0\\
 \ast &  \ast &  \lambda_{\bv_{\tilde E}}(\gamma) 
\end{array}
\right) 
\]
where we have
\[ \lambda_{+}(\gamma)^{n}\det({\hat L}(\gamma_{\bR^{n-1}})) \lambda_{\bv_{\tilde E}}(\gamma)  = \pm 1.\]
{(Note $\lambda_{\bv_{\tilde E}}(\gamma^{-1}) = \lambda_{\bv_{\tilde E}}(\gamma)^{-1}$
and $\lambda_{+}(\gamma^{-1}) = \lambda_{+}(\gamma)^{-1}$.)}

We will show that $\gamma$ is unimodular as an affine transformation
of $\bR^{n-1}$, i.e., all norms of eigenvalues are $1$. 
There are following possibilities for each $\gamma \in \bGamma_{\tilde E}$:
\begin{itemize}
\item[(a)] $\lambda_{1}(\gamma) > \lambda_{+}(\gamma), \lambda_{\bv_{\tilde E}}(\gamma)$.
\item[(b)] $\lambda_{1}(\gamma) = \lambda_{+}(\gamma)= \lambda_{\bv_{\tilde E}}(\gamma)$.
\item[(c)] $\lambda_{1}(\gamma) = \lambda_{+}(\gamma), \lambda_{1}(\gamma) > \lambda_{\bv_{\tilde E}}(\gamma)$.
\item[(d)]  $\lambda_{1}(\gamma) > \lambda_{+}(\gamma), \lambda_{1}(\gamma) =\lambda_{\bv_{\tilde E}}(\gamma)$.
\end{itemize}

Suppose that $\gamma$ satisfies (b). The relative eigenvalues of $\gamma$ on $\bR^{n-1}$ are all $\leq 1$. 
Either $\gamma$ is unimodular or we can take $\gamma^{-1}$ and we are in case (a).  

Suppose that $\gamma$ satisfies (a).
There exists a projective subspace $S$ of dimension $\geq 0$ where 
the points are associated with eigenvalues with the norm $\lambda_{1}(\gamma)$
where $\lambda_{1}(\gamma) >  \lambda_{+}(\gamma), \lambda_{\bv_{\tilde E}}(\gamma)$. 

Let $S'$ be the smallest subspace containing $H$ and $S$. Let $U$ be a p-end neighborhood of $\tilde E$. 
Let $y_1$ and $y_2$ be generic points of $U \cap S' - H$
so that $\ovl{y_1 y_2}$ meets $H$ in its interior.

Then we can choose a subsequence $m_i$, $m_i \ra \infty$, so that 
$\gamma^{m_i}(y_1) \ra f$ and $\gamma^{m_i}(y_2) \ra f_-$ as $i \ra +\infty$
unto relabeling $y_1$ and $y_2$ for a pair of antipodal points $f, f_- \in S$.
This implies $f, f_- \in \clo(\torb)$, 
and $\torb$ is not properly convex, which is a contradiction. 


If $\gamma$ satisfies (c), then 
\begin{equation}\label{eqn-c}
\lambda_{1}(\gamma) = \lambda_{+}(\gamma) \geq \lambda_{i}(\gamma) \geq\lambda_{\bv_{\tilde E}}(\gamma)
\end{equation} 
for all other norms of eigenvalues $\lambda_{i}(\gamma)$: 
Otherwise, we can use the argument similar to above to obtain contradiction
where we also have to consider $\gamma^{-1}$.
Similarly if $\gamma$ satisfies (d), then 
we have 
\begin{equation}\label{eqn-d}
 \lambda_{1}(\gamma) =\lambda_{\bv_{\tilde E}}(\gamma)\geq \lambda_{i}(\gamma) \geq \lambda_{+}(\gamma)
 \end{equation}
for all other norms of eigenvalues $\lambda_{i}(\gamma)$.




There is a homomorphism 
\[\lambda_{\bv_{\tilde E}}: \bGamma_{\tilde E} \ra \bR_{+} \hbox{ given by } g \mapsto \lambda_{\bv_{\tilde E}}(g).\]
This gives us an exact sequence 
\[ 1 \ra N \ra \bGamma_{\tilde E}  \ra R \ra 1 \]
where $R$ is a finitely generated subgroup of $\bR_{+}$, an abelian group.
For an element $g \in N$, 
$\lambda_{\bv_{\tilde E}}(g) = 1$. 
Since the relative eigenvalue corresponding to $\hat L(g_{\bR^{n-1}})|A$ is $1$, 
a matrix form shows that $\lambda_{+}(g) = 1$. 
Equations \eqref{eqn-c} and \eqref{eqn-d} show that $g$ is unimodular. 
%
Thus, $N$ is therefore virtually unipotent by Fried \cite{Fried86} again. 
Taking a finite cover again, we may assume that $N$ is unipotent. 

Since $R$ is a finitely generated abelian group, 
$\bGamma_{\tilde E}$ is solvable. 
Since $\tilde \Sigma_{\tilde E}= \bR^{n-1}$ is complete affine, 
Proposition S of Goldman and Hirsch  \cite{GH} implies 
\[\det(g_{\,\bR^{n-1}}) = 1 \hbox{ for all } g \in \bGamma_{\tilde E}.\]

If $\gamma$ satisfies (c), then all norms of eigenvalues of $\gamma$ 
except for $\lambda_{\bv_{\tilde E}}(\gamma)$ equal $\lambda_{+}(\gamma)$ 
since otherwise by equation \eqref{eqn-c}, the above determinant is less than $1$. 
Similarly, if $\gamma$  satisfies (d),  then all norms of eigenvalues of $\gamma$ 
except for $\lambda_{\bv_{\tilde E}}(\gamma)$ equals $\lambda_{+}(\gamma)$.

Therefore, only (b), (c), and (d) hold and $g_{\, \bR^{n-1}}$ is unimodular for every $g \in \bGamma_{\tilde E}$. 
Hence, $\bGamma_{\tilde E}|\bR^{n-1}$ is virtually unipotent group and hence is virtually nilpotent
by Fried \cite{Fried86} again. 

Suppose that every $\gamma$ is unimodular. Then we have the first case of (i). If not, then the second case of (i) holds. 

(ii) This follows by Lemma \ref{lem-unithoro}.

\end{proof} 

The second case will be studied later. See Corollary \ref{cor-caseiii}.
We will show the end $\tilde E$ to be a NPNC-end with fiber dimension $n-2$ when we choose another point as the new p-end vertex for $\tilde E$.
Clearly, this case is not horospherical. 


\begin{lemma}\label{lem-unithoro}
 Suppose that eigenvalues of elements of  $\bGamma_{\tilde E}$ have unit norms only. 
Then a nilpotent Lie group fixing $\bv_{\tilde E}$ contains a finite index subgroup of $\bGamma_{\tilde E}$
and $\tilde E$ is horospherical, i.e., cuspidal. 
\end{lemma}
 \begin{proof}
 Since $\tilde \Sigma_{\tilde E}/\bGamma_{\tilde E}$ is a compact complete-affine manifold, 
a finite index subgroup $F$ of $\bGamma_{\tilde E}$ is contained in a nilpotent Lie subgroup
acting on $\tilde \Sigma_{\tilde E}$ by Theorem 3 in Fried \cite{Fried86}.
Now, by Malcev, it follows that the same group is contained in 
a nilpotent group $N$ acting on $\SI^n$ since $F$ is unipotent. 
The dimension of $N$ is $n-1 = \dim \tilde \Sigma_{\tilde E}$ by Theorem 3 of \cite{Fried86} again. 

Let $U$ be a component of the inverse image of a p-end neighborhood 
{so that ${\bv_{\tilde E}} \in \Bd U$.} Assume that $U$ is a radial p-end neighborhood of $\bv_{\tilde E}$. 
A finite index subgroup $F$ of $\bGamma_{\tilde E}$ is in $N$ so that $N/F$ is compact by Malcev's results.  
$N$ acts on a smaller open set covering a p-end neighborhood
by taking intersections under images of it under $N$ if necessary. 
We let $U$ be this open set from now on. Consequently, $\Bd U \cap \torb$ is smooth. 
We will now show that $U$ is a horospherical p-end neighborhood:
We identify ${\bv_{\tilde E}}$ with $[1, 0, \dots, 0]$.
Let $W$ denote the subspace in $\SI^n$ containing ${\bv_{\tilde E}}$ supporting $U$. 
$W$ corresponds to the boundary of the direction of $\tilde \Sigma_{\tilde E}$ and hence is unique 
and, thus, $N$-invariant. Also, 
$W \cap \clo(\torb)$ is a properly convex subset of $W$. 

Let $y$ be a point of $U$. 
Suppose that $N$ contains a sequence $\{g_i\}$ so that  
\[g_i(y) \ra x_0 \in W \cap \clo(\torb) \hbox{ and } x_0 \ne {\bv_{\tilde E}};\] 
that is, $x_0$ in the boundary direction of $A$ from $\bv_{\tilde E}$. 
Let $U_{1} = \clo(U) \cap W$. Let $V$ be the smallest subspace containing $\bv_{\tilde E}$ and $U_{1}$. 
The dimension of $V$ is $\geq 1$ as it contains $x_0$.

Again $N$ acts on $V$. 
Now, $V$ is divided into disjoint open hemispheres of various dimensions where $N$ acts on:
By Theorem 3.5.3 of \cite{Var}, $N$ preserves a flag structure 
$V_0 \subset V_1 \subset \dots \subset V_k = V$.
We take components of complement $V_i - V_{i-1}$. 
Let $H_V:=V - V_{k-1}$. 

Suppose that $\dim V = n-1$ for contradiction. 
Then $H_V \cap U_1$ is not empty since otherwise, we would have a smaller dimensional $V$. 
Let $h_V$ be the component of $H_V$ meeting $U_1$.
Since $N$ is unipotent, $h_V$ has an $N$-invariant metric by Theorem 3 of Fried \cite{Fried86}.

We claim that the orbit of the action of $N$ is of dimension $n-1$ and hence locally transitive on $H_V$: 
If not, then a one-parameter subgroup $N'$ fixes a point of $h_V$.  
This group acts trivially on $h_V$ since the unipotent group contains a trivial orthogonal subgroup. 
Since $N'$ is not trivial, it acts as a group of nontrivial translations on the affine space $H^o$.
We obtain that $N'(U)$ is not properly convex, and 
an orbit of $N$ is open in $h_{V}$. 
Hence, $N$ acts locally simply-transitively without fixed points.

The orbit of $N$ in $h_V$ is closed  since $h_{V}$ has an $N$-invariant metric. 
Thus, $N$ acts transitively on $h_V$. 

Hence, the orbit $N(y)$ of $N$ for $y \in H_V \cap U_1$ contains a component of $H_V$. 
Since $\bGamma_{\tilde E}(y) \subset \clo(\torb)$
and a convex hull in $\clo(\torb)$ is $N(y)$ where $N(y) \subset H_V$. 
Since $F \bGamma_{\tilde E}  = N$ for a compact subset $F$ of $N$, 
the orbit $\bGamma_{\tilde E}(y)$ is within a bounded distance from every point of $N(y)$. 
Thus, a convex hull in $\clo(H_V)$ is $N(y)$, and 
this contradicts the assumption that $\clo(\torb)$ is properly convex
(compare with arguments in \cite{CM2}.)

$N$ acts transitively on $A$ by Propositions S and T of Goldman and Hirsch \cite{GH}
 since $A/\bGamma_{\tilde E}$ is compact. 

Suppose that the dimension of $V$ is $\leq n-2$. 
Let $J$ be a subspace of dimension $1$ bigger than $\dim V$ and containing $V$ and meeting $U$. 
Then $J$ is sent to disjoint subspaces or to itself under $N$. Since $N$ acts transitively on $A$, 
a nilpotent subgroup $N_J$ of $N$ acts on $J$.
Now we are reduced to $\dim V$ by one or more. 
The orbit $N_J(y)$ for a limit point $y \in H_V$ contains a component of $V - V_{k-1}$
as above. Thus, $N_J(y)$ contains the same component, an affine subspace. 
As above, we have a contradiction to the proper convexity. 

Therefore, points such as $x_0 \in W \cap \Bd(\torb)  - \{\bv_{\tilde E}\}$ do not exist. 
Hence for any sequence of elements $g_i \in \bGamma_{\tilde E}$, we have
$g_i(y) \ra \bv_{\tilde E}$. 

Hence, $\Bd U = (\Bd U \cap \torb) \cup \{ \bv_{\tilde E}\}$.  
Clearly, $\Bd U$ is homeomorphic to an $(n-1)$-sphere.
Since $U$ is radial, this means that $U$ is a horospherical p-end neighborhood. 

\end{proof}


\section{The uniform middle eigenvalue condition and properly convex radial ends} \label{sec-endth}

In this section, we define various types of middle eigenvalue conditions for properly convex radial ends. 
We will just state results from \cite{End2} without proofs. 
We state the equivalence {of the lens condition} and the uniform middle eigenvalue condition for R-ends. 
Finally, we discuss the quasi-lens-shaped R-ends. 

Let $\orb$ be real projective orbifold with R-ends or T-ends.
Let $\tilde E$ be a p-R-end  or a p-T-end and $\bGamma_{\tilde E}$ the associated p-end fundamental group.
Let $\tilde \Sigma_{\tilde E}$ denote the universal cover of the end orbifold $\Sigma_{\tilde E}$ associated with $\tilde E$. 
If every subgroup of finite index of a group $\bGamma_{\tilde E} \subset \bGamma$ has a finite center, 
we say that $\bGamma_{\tilde E}$ 
is a {\em virtual center-free group} or  a {\em vcf-group}.

We say that 
$\bGamma_{\tilde E}$ is an {\em pseudo-admissible group} if the following holds:
\begin{itemize} 
\item $\clo(\tilde \Sigma_{\tilde E}) = K_{1}\ast \cdots \ast K_{k}$ 
where each $K_{i}$ is properly convex or is $0$-dimensional and 
\item $\bGamma_{\tilde E}$ is virtually a direct product 
$\Gamma_{1} \times \cdots \times \Gamma_{k}$ 
where 
\begin{itemize}
\item $\Gamma_{i}$ acts on $K_{j}$ trivially for $j\ne i$, and  
\item $K^{o}_{i}/\Gamma_{i}$ is compact. 
\end{itemize}
\end{itemize}
Let $r_{K_{i}}: \bGamma_{\tilde E} \ra \Aut(K_{i})$ denote the restriction homomorphism $g \mapsto g|K_{i}$ for $g \in \bGamma_{\tilde E}$. 
It follows that $\Gamma_{i}$ is isomorphic to $r_{K_{i}}(\Gamma_{i})$. 

The group $\bGamma_{\tilde E}$ is {\em admissible} 
if each $\Gamma_{i}$ is a hyperbolic group or equivalently $K_{i}$ is strictly convex
by \cite{Ben1}. We call $\Gamma_{i}$ the {\em hyperbolic factor} of $\bGamma_{\tilde E}$. 

We note {that the strict convexity} of $K_{i}$ is needed for technical reason and we hope to eliminate it in the future. 
{(See \cite{book}.)}

If two groups $G_{1}$ and $G_{2}$ have finite-index subgroups isomorphic to each other, we write $G_{1} \cong G_{2}$. 

(See Section \ref{sub-ben} for details. In this paper, we will simply use $\bZ^{k-1}$ and $\Gamma_i$ to denote 
the subgroup in $\bGamma_{\tilde E}$ corresponding to it.)
We say that 
$\tilde E$ is {\em virtually non-factorable} if the center is trivial for every finite index subgroup of $\bGamma_{\tilde E}$; 
otherwise, $\tilde E$ is virtually factorable. 


Let $\Gamma$ be generated by finitely many elements $g_1, \ldots, g_m$. 
The {\em conjugate word length} $\cwl(g)$ of $g \in \pi_1(\tilde E)$ is the minimum of 
the word length of the conjugates of $g$ in $\pi_1(\tilde E)$. 


Let $\Omega$ be a convex domain in an affine space $A$ in $\bR P^n$ or $\SI^n$. 
Let $o, s, q, p$ denote four points on a one dimensional subspace $l$, and let 
$\bar o, \bar p, \bar q, \bar s$ denote respectively the first coordinates of the homogeneous coordinates  of $l$
so that the second coordinates are normalized to be $1$. 
Then $[o, s, q, p]$ denotes the cross ratio of four points on a one-dimensional subspace as defined by 
\[ \frac{\bar o - \bar q}{\bar s - \bar q} \frac{\bar s - \bar p}{\bar o - \bar p}. \] 
Define a metric by defining for every $p, q \in \Omega$, 
\[d_\Omega(p, q)= \log|[o,s,q,p]|\] where $o$ and $s$ are 
endpoints of the maximal segment $l$ in $\Omega$ containing $p, q$
where $o, q$ separates $p, s$ in $l$. 
The metric is one given by a Finsler metric provided $\Omega$ is properly convex. (See \cite{Kobpaper}.)
Given a properly convex real projective structure on ${\mathcal{O}}$, it carries a Hilbert metric which we denote by $d_{\torb}$ 
on $\tilde{\mathcal{O}}$. 
This induces a metric on ${\mathcal{O}}$ denoted by $d_{\orb}$.

Let $d_K$ denote the Hilbert metric on the interior $K^o$ of a properly convex domain $K$ in $\bR P^n$ or $\SI^n$. 
Suppose that a projective automorphism $g$ acts on $K$. 
Let $\leng_K(g)$ denote the infinum of $\{ d_K(x, g(x))| x \in K^o\}$. 

Suppose that a group $\Gamma$ of projective automorphisms of $K$ 
acts on $K^{o}$. 
Then we note that there exists some constant $C> 1$ so that 
\[ \leng_{K}(g) \leq C \cwl(g) \hbox{ for } g \in \Gamma.\]
If $\Gamma$ acts on $K^{o}$ properly discontinuously and cocompactly, 
then we have
\[C^{-1}\cwl(g) \leq  \leng_{K}(g) \leq C \cwl(g) \hbox{ for } g \in \Gamma\]
for a constant $C> 1$. 

\subsection{The uniform middle eigenvalue condition for properly convex R-ends and the lens-shaped condition. }

The following definition applies to properly convex R-ends. 
A {\em middle eigenvalue condition} for $\bGamma_{\tilde E}$ is the condition 
$\lambda_{1}(g) > \lambda_{\bv_{\tilde E}}(g)$ for the largest norm $\lambda_{1}(g)$ of the eigenvalues of 
each $g \in \bGamma_{\tilde E}$. 
The condition is a {\em weak middle eigenvalue condition} if we only require that 
if $\lambda_{\bv_{\tilde E}}(g)$ is the largest norm of the eigenvalues, then its multiplicity is $\geq 2$. 

\begin{definition}\label{defn-umec}
Let $\bv_{\tilde E}$ be a p-end vertex of a properly convex p-R-end $\tilde E$. 
Let $\tilde \Sigma_{\tilde E} \subset \SI^{n-1}_{\tilde E}$ denote the universal cover of the end orbifold 
corresponding to $\tilde E$. 
Suppose that the p-end fundamental group $\Gamma_{\tilde E}$ is admissible.
The p-end fundamental group $\bGamma_{\tilde E}$ satisfies the {\em uniform {middle eigenvalue} condition} 
\begin{itemize}
\item if each $g\in \bGamma_{\tilde E}$ satisfies for a uniform  $C> 1$ independent of $g$
\begin{equation}\label{eqn-umec}
C^{-1} \leng_{\tilde \Sigma_{\tilde E}}(g) \leq \log\left(\frac{\bar\lambda(g)}{\lambda_{\bv_{\tilde E}}(g)}\right) \leq 
C \leng_{\tilde \Sigma_{\tilde E}}(g) , 
\end{equation}
for $\bar \lambda(g)$ equal to 
the largest norm of the eigenvalues of $g$
and the eigenvalue $\lambda_{\bv_{\tilde E}}(g)$ of $g$ at $\bv_{\tilde E}$.
\end{itemize}



We say that $\bGamma_{\tilde E}$ satisfies the {\em weakly uniform middle-eigenvalue conditions} if 
we replace the above condition by the following: 
\begin{itemize} 
\item If $\lambda_{\bv_{\tilde E}}(g)$, $g \in \bGamma_{\tilde E}$ has the largest norm among 
eigenvalues, then the norm has to be of multiplicity $\geq 2$, 
\item the uniform middle eigenvalue condition for each hyperbolic factor $\bGamma_i$, i.e., the condition
{\eqref{eqn-umec}}. 
\end{itemize}

\end{definition}
The definition of course applies to the case when $\bGamma_{\tilde E}$ has the finite index subgroup with the above properties.

We can summarize the middle eigenvalue conditions (MEC) as: 
\begin{equation}
\hbox{weak MEC} \preceq  
\left\{
\begin{array}{c}
   \hbox{MEC}   \\
  \hbox{weak uniform MEC} 
\end{array}
\right\}
\preceq \hbox{uniform MEC}.
\end{equation}
Here $\preceq$ indicates the strength of the conditions. 

We give a dual definition: 
\begin{definition} 
Suppose that $\tilde E$ is a properly convex p-T-end. 
{Let $g^{\ast}:\bR^{n+1 \ast} \ra \bR^{n+1 \ast}$ be the dual transformation of $g: \bR^{n+1} \ra \bR^{n+1}$. }
Then each element {$g^{\ast}$} of the dual group $\bGamma_{\tilde E}^{\ast}$ fixes a point $\bv^{\ast}_{\tilde E} \in \bR P^{n \ast}$
corresponding to hyperspace containing 
$\tilde \Sigma_{\tilde E}$ with the eigenvalue to be denoted {$\lambda_{\bv_{\tilde E}^{\ast}}(g^{\ast})$}.
The p-end fundamental group $\bGamma_{\tilde E}$ satisfies the {\em uniform middle-eigenvalue condition}
if it satisfies {
\begin{equation}\label{eqn-umecD}
C^{-1} \leng_{\tilde \Sigma_{\tilde E}}(g) \leq \log\left(\frac{\bar\lambda(g)}{\lambda_{\bv_{\tilde E}^{\ast}}(g^{\ast})}\right) \leq 
C \leng_{\tilde \Sigma_{\tilde E}}(g)  
\end{equation}  }
for $\bar \lambda(g)$ equal to 
the largest norm of the eigenvalues of $g$.
\end{definition} 
Again the middle eigenvalue condition and the associated conditions 
is as follows: 
for each $g\in \pi_{1}(\tilde E) -\{\Idd\}$, 
the largest norm $\lambda_{1}(g)$ of eigenvalues of $g$ is strictly larger than 
the eigenvalue $\lambda_{\bv_{\tilde E}^{\ast}}(g)$. 
The end fundamental group $\bGamma_{\tilde E}$ will act on a properly convex domain $K^o$ of lower-dimension
and we will apply the definition here. 
This condition is similar to ones studied by Guichard and Wienhard \cite{GW}, and the results also 
seem similar. Our main tools to understand these questions are in Appendix A in \cite{End2}, and 
the author does not really know the precise relationship here.) 



The condition is an open condition; and hence a ``structurally stable one."
(See \cite{End2}.)
Our main result to be proved in \cite{End2} (see also \cite{endclass}) is the following: 
\begin{theorem}\label{thm-secondmain} 
Let $\mathcal{O}$ be a strongly tame properly convex real projective orbifold. 
Assume that the holonomy group of $\mathcal{O}$ is strongly irreducible.
The p-end fundamental group $\pi_{1}(\tilde E)$ of a p-end $\tilde E$ 
acts admissibly on on {$\clo(\tilde \Sigma_{\tilde E}) = K_{1}\ast \cdots \ast K_{s}$}
so that $K_{i}$ are strictly convex and $K_{i}^{o}/\Gamma_{i}$ is compact Hausdorff for each $i$ 
where $\Gamma_{i}$ is the restriction to $K_{i}$ of $\pi_{1}(\tilde E)$
and $\pi_{1}(\tilde E)$ is virtually isomorphic to \[\bZ^{s-1}\times \Gamma_{1}\times \cdots \times \Gamma_{s}.\] 
\begin{itemize} 
\item Let $\tilde E$ be a properly convex  p-R-end. 
\begin{itemize} 
\item Suppose that the p-end holonomy group of $\tilde E$ satisfies the uniform middle-eigenvalue condition.
Then $\tilde E$ is generalized-lens-shaped.
\item Suppose that the p-end holonomy group of $\tilde E$ satisfies the weakly uniform middle-eigenvalue condition.
Then $\tilde E$ is generalized-lens-shaped  or quasi-lens-shaped .
\end{itemize} 
\item 
If $\tilde E$ is virtually factorable or is a totally geodesic R-end, 
then we can replace the word ``generalized-lens-shaped''
to ``lens-shaped'' in each of the above statements. 
\end{itemize} 
\end{theorem}


\begin{remark}[Duality of ends] \label{rem-duality} 
Above orbifold $\orb=\torb/\Gamma$ has a diffeomorphic dual orbifold $\orb^{\ast}$, 
where $\orb^{\ast}$ is defined as the quotient of the dual domain 
$\torb^{\ast}$ by the dual group $\Gamma^{\ast}$ of $\Gamma$ by Theorem \ref{thm-dualdiff}.
By Theorem \ref{thm-dualdiff}, each end neighborhood of a properly convex 
strongly tame open orbifold $\orb$ goes to an end neighborhood of the dual orbifold $\orb^{\ast}$. 
The ends of $\orb$ and $\orb^{\ast}$ are in a one-to-one correspondence. 
Horospherical ends are dual to themselves, i.e., ``self-dual'', 
and properly convex R-ends and T-ends are dual to one another. (See \cite{End2}.)
We will see that properly convex generalized-lens-shaped R-ends
are always dual to lens-shaped T-ends by \cite{End2}. 
\end{remark}


\begin{theorem}\label{thm-equ2}
Let $\orb$ be a strongly tame properly convex real projective orbifold.
Assume that the holonomy group is strongly irreducible.
Assume that the universal cover $\torb$ is a subset of $\SI^n$ {\rm (}resp. $\bR P^n${\rm ).}
Let $\tilde S_{\tilde E}$ be a totally geodesic ideal boundary of a p-T-end $\tilde E$ of $\torb$. 
Let $\bGamma_{\tilde E}$ act admissibly on $\tilde S_{\tilde E}$. 
Then the following conditions are equivalent\,{\rm :} 
\begin{itemize} 
\item[(i)] $\tilde E$ satisfies the uniform middle-eigenvalue condition.
\item[(ii)] $\tilde S_{\tilde E}$ has a lens-neighborhood in an ambient open manifold containing $\torb$.
\end{itemize} 
\end{theorem}


\begin{remark}[Virtually factorable ends are self-dual in a generalized sense] \label{rem-red} 
Let $\orb$ be as above, and $\orb^{\ast}$ be the dual orbifold. 
A generalized-lens-shaped virtually factorable properly convex R-end of $\orb$ is always 
totally geodesic by a result in \cite{End2}. 
By Theorem \ref{thm-secondmain}, the R-end is lens-shaped always. 
The dual end of $\orb^{\ast}$ is totally geodesic and lens-shaped since it satisfies the uniform middle eigenvalue condition. 
Since the holonomy group of the corresponding p-end $\tilde E^{\ast}$ fixes a unique point dual 
to the totally geodesic convex domain in a p-end neighborhood of $\tilde E$ of $\torb$. 
The p-end can be made into a totally geodesic p-R-end  by taking a cone over that point. 
Thus, the virtually factorable properly convex ends are ``self-dual''. {We} consider these the best types of cases.
(See Section \ref{sec-duality}.) 

\end{remark}






 \subsection{The characterization of quasi-lens p-R-end \\ neighborhoods} \label{sub-quai-lens} 

This is the last remaining case for the properly convex ends with weakly uniform middle eigenvalue conditions. 
We will only prove for $\SI^n$. 

\begin{definition}\label{defn-quasilens}
\begin{itemize}
\item Let $D$ be the properly convex 
totally geodesic $(n-2)$-dimensional domain so that 
\[U = D \ast \bv \subset \SI^{n-1} \subset \SI^{n}.\] 
Let $G$ be a projective automorphism group of $\SI^{n-1}$. 
$G$ acts on $D$ and $\bv$ as  the p-end fundamental group for $D\ast \bv$ with p-end vertex $\bv$
satisfying the weakly uniform middle eigenvalue condition. 
\item Let $\SI^1$ be a great circle in $\SI^{n}$ meeting $\SI^{n-1}$ at $\bv$. 
\item Extend $G$ to act on $\SI^1$ as a nondiagonalizable transformation fixing $\bv$. 
\item Let $\zeta$ be a projective automorphism of $\SI^{n}$
acting on $U$ and $\SI^1$ so that $\zeta$ commutes with $G$ and restrict to a diagonalizable transformation on $\clo(D)$
and act as a nondiagonalizable transformation on $\SI^1$ fixing $\bv$ and with largest norm eigenvalue at $\bv$.
\end{itemize}
Every element $g$ of $\langle G, \zeta \rangle$ can be written as a matrix
\begin{equation}
\left( \begin{array}{c|c}
S(g) & 
\begin{array}{cc} 
0 \quad & \quad 0 \end{array} 
 \\
\hline
\begin{array}{c} 
0 \\ 0 
\end{array} 
 &  \begin{array}{cc}
\lambda_{\bv}(g) & \lambda_{\bv}(g)v(g)\\
0 & \lambda_{\bv}(g) \end{array} 
\end{array}\right) \label{eqn:qj}
\end{equation} 
where $\bv =[0, \dots, 1]$. 
Note that $g \mapsto v(g) \in \bR$ is a well-defined map inducing a homomorphism 
\[ \langle G, \zeta \rangle \ra H_1(\bGamma_{\tilde E}) \ra \bR\] 
and hence 
\[ |v(g)| \leq C \cwl(g)\] for a positive constant $C$. 


For $g \in \langle G, \zeta \rangle $, 
let $\lambda_{2}(g)$ denote the largest norm of the eigenvalue associated with $\clo(D)$.
\begin{description}
\item[Positive translation condition] We choose an affine coordinate on a component $I$ of $\SI^1 -\{\bv, \bv_-\}$.
We assume that for each $g \in \langle G, \zeta \rangle  $,
if $\lambda_{\bv}(g) > \lambda_2(g)$ 
then $v(g) > 0$ in equation \eqref{eqn:qj}, and 
\[ \frac{v(g)}{\log \frac{\lambda_{\bv}(g)}{\lambda_2(g)}} > c_1 > 0 \]
for a constant $c_1$. 
\end{description}
\end{definition}


\begin{proposition}\label{prop-quasilens1} 
Suppose that $\langle G, \zeta \rangle$ satisfies the positive translation condition
and is admissible.
Then  the above $U$ is in the boundary of a properly convex p-end open neighborhood $V$ of $\bv$ 
and $\langle G, \zeta \rangle$ acts on $V$ properly. 
\end{proposition}

We call the end of type constructed as above a {\em quasi-lens-shaped}  end. 
This generalizes the quasi-hyperbolic annulus discussed in \cite{cdcr2}. 



From this, we obtain:


\begin{theorem} \label{thm-quasilens2} 
Let $\orb$ be a strongly tame properly convex real projective orbifold.
Suppose that $\pi_1(\orb)$ is strongly irreducible. 
Let $\tilde E$ be a properly convex p-R-end with the admissible end fundamental group 
satisfying the weakly uniform middle eigenvalue conditions
but not the uniform middle eigenvalue condition.
Then $\tilde E$ has a quasi-lens-shaped p-R-end neighborhood. 
\end{theorem}


 \section{The NPNC-ends} \label{sec-notprop}

We will now study the ends where the transverse real projective structures are 
\begin{itemize} 
\item not properly convex but 
\item not projectively diffeomorphic to a complete affine subspace.
\end{itemize} 
First, these ends have transverse end orbifolds that are foliated by complete affine leaves,
and the associated exact sequence of the p-end fundamental groups. 
We will basically show how to split this group. 
We explain the eigenvalue result following from our {transverse weak middle eigenvalue} condition for NPNC-ends.
Then we introduce a hypothesis and derive the splitting. 
We explain joins and quasi-joins and conditions for them. 
We state the classification of NPNC-ends in \cite{End3}. 
Finally, we do the remaining case of the complete affine R-ends that were left from above. 

\subsection{The structure of the NPNC-ends} 
Let $\tilde E$ be a p-R-end of $\orb$ and let $U$ be the corresponding p-end neighborhood in $\torb$
with the p-end vertex $\bv_{\tilde E}$. 
The closure $\clo(\tilde \Sigma_{\tilde E}) \subset \SI^{n-1}_{\bv_{\tilde E}}$ contains 
a great $(i_0-1)$-dimensional sphere and $\tilde \Sigma_{\tilde E}$ is foliated by $i_0$-dimensional 
open hemispheres, i.e, complete affine spaces,  
with this boundary by Proposition \ref{prop-projconv}.
Let $\SI^{i_0-1}_\infty$ denote the great $(i_0-1)$-dimensional sphere 
$\SI^{n-1}_{\bv_{\tilde E}}$ in $\Bd \tilde \Sigma_{\tilde E}$
which is a boundary of complete-affine leaves. 
The space of $i_0$-dimensional hemispheres in $\SI^{n-1}_{\bv_{\tilde E}}$ with boundary $\SI^{i_0-1}_\infty$ form 
a projective sphere $\SI^{n-i_0-1}$. 
The projection \[\SI^{n-1}_{\bv_{\tilde E}} - \SI^{i_0-1}_\infty \ra \SI^{n-i_0-1} \] 
gives us an image of $\tilde \Sigma_{\tilde E}$ 
that is the interior of a  properly convex compact set $K$. (See \cite{ChCh} for details. See also \cite{GV}.)
Here, we will call $i_{0}$ the {\em fiber-dimension} of the NPNC-end $E$. 

Let $\SI^{i_0}_\infty$ be a great $i_0$-dimensional sphere containing $\bv_{\tilde E}$ corresponding to the directions
of $\SI^{i_0-1}_\infty$ from $\bv_{\tilde E}$. 
The space of  $(i_0+1)$-dimensional hemispheres 
with boundary $\SI^{i_0}_\infty$ again has the structure of the projective sphere $\SI^{n-i_0-1}$, 
identifiable with the above one. 
We also have the projection \[\Pi_K: \SI^{n} - \SI^{i_0}_\infty \ra \SI^{n-i_0-1} \] 
giving us the image $K^o$ of $\tilde \Sigma_{\tilde E}$. 


Each $i_0$-dimensional 
hemisphere $H^{i_0}$ in \hyperlink{term-lsphere}{$\SI^{n-1}_{\bv_{\tilde E}}$} with 
$\Bd H^{i_0} = \SI^{i_0-1}_\infty$ corresponds to an $(i_0+1)$-dimensional hemisphere 
$H^{i_0+1}$ in $\SI^n$ with common boundary $\SI^{i_0}_\infty$ that contains $\bv_{\tilde E}$. 
Let $\SL_\pm(n+1, \bR)_{\SI^{i_0}_\infty, \bv_{\tilde E}}$ 
denote the subgroup of $\Aut(\SI^n)$ acting on 
$\SI^{i_0}_\infty$ and $\bv_{\infty}$.  
The projection $\Pi_K$ induces a homomorphism 
\[\Pi_K^{\ast}: \SL_\pm(n+1, \bR)_{\SI^{i_0}_\infty, \bv_{\tilde E}} 
\ra \SL_\pm( n-i_{0}, \bR).\]

Suppose that $\SI^{i_0}_\infty$ is $h(\pi_1(\tilde E))$-invariant. 
We let $N$ be the subgroup of $h(\pi_1(\tilde E))$ of elements inducing trivial actions on $\SI^{n-i_0-1}$. 
The above exact sequence 
\[ 1 \ra N \ra h(\pi_1(\tilde E)) \stackrel{\Pi^*_K}{\longrightarrow} N_K \ra 1\] 
is so that the kernel normal subgroup $N$ acts trivially on $\SI^{n-i_0-1}$ but acts on each hemisphere with 
boundary equal to $\SI^{i_0}_\infty$
and $N_K$ acts faithfully by the action induced from $\Pi^*_K$.
Here $N_K$ is a subgroup of the group $\Aut(K)$ of the group of projective automorphisms of $K$.
We say that $N_{K}$ is the {\em semisimple quotient } of $h(\pi_1(\tilde E))$ or $\bGamma_{\tilde E}$. 

\begin{theorem}\label{thm-folaff}
Let $\Sigma_{\tilde E}$ be the end orbifold of an NPNC p-R-end $\tilde E$ of 
properly convex $n$-orbifold $\orb$.  
Let $\torb$ be the universal cover in $\SI^n$. 
Suppose that the holonomy group of $\pi_{1}(\orb)$ is strongly irreducible.
We consider the induced action of $h(\pi_1(\tilde E))$ 
on $\Aut(\SI^{n-1}_{\bv_{\tilde E}})$ for the corresponding p-end vertex $\bv_{\tilde E}$. 
Then 
\begin{itemize} 
\item $\Sigma_{\tilde E}$ is foliated by complete-affine subspaces of dimension $i_0$, $i_0 > 0$.
\item $h(\pi_1(\tilde E))$ fixes the great sphere $\SI^{i_0-1}_\infty$ of dimension $i_0-1$ in \hyperlink{term-lsphere}{$\SI^{n-1}_{\bv_{\tilde E}}$}. 
\item There exists an exact sequence 
\[ 1 \ra N \ra \pi_1(\tilde E) \stackrel{\Pi^*_K}{\longrightarrow} N_K \ra 1 \] 
where $N$ acts trivially on quotient great sphere $\SI^{n-i_0-1}$ and 
$N_K$ acts faithfully on a properly convex domain $K^o$ in $\SI^{n-i_0-1}$ isometrically 
with respect to the Hilbert metric $d_K$. 
\end{itemize} 
\end{theorem}
\begin{proof} 
These follow from Section 1.4 of \cite{ChCh}. (See also \cite{GV}.) 
\end{proof} 


\subsection{The main eigenvalue estimations}

We denote by $\bGamma_{\tilde E}$ the p-end fundamental group acting on $U$ fixing $\bv_{\tilde E}$. 
Denote the induced foliations on $\Sigma_{\tilde E}$ and $\tilde \Sigma_{\tilde E}$ by ${\mathcal F}_{\tilde E}$.
For each element $g \in \bGamma_{\tilde E} $, we define $\leng_K(g)$ to 
be $\inf \{ d_K(x, g(x))| x \in K^o \}$.

\begin{definition}\label{defn:jordan}
Given an eigenvalue $\lambda$ of an element $g \in \SLnp$, 
a {\em $\bC$-eigenvector} $\vec v$ is a nonzero vector in {
\[\bR E_{\lambda}(g) := \bR^{n+1} \cap \big(\ker (g - \lambda I) 
+ \ker (g - \bar \lambda I)\big), \lambda \ne 0, {\mathrm{Im}} \lambda \geq 0.\]}
A $\bC$-fixed point is the direction of a $\bC$-eigenvector. 

Any element of $g$ has a primary decomposition. (See Section 6.8 of \cite{HK}.)
Write the minimal polynomial of $g$ as $\prod_{i=1}^{m} (x- \lambda_{i})^{r_{i}}$ for $r_{i} \geq 1$ and 
mutually distinct complex numbers $\lambda_{1}, \dots, \lambda_{m}$. 
Define
\[C_{\lambda_{i}}(g) := \ker (g- \lambda_{i}\Idd)^{r_{i}} \subset \bC^{n+1}\]
where $r_{i} = r_{j}$ if $\lambda_{i} = \bar \lambda_{j}$. 
Then the primary decomposition theorem states 
\[\bC^{n+1} = \bigoplus_{i=1}^{m} C_{\lambda_{i}}(g).\]

A point $[\vec v], \vec v \in \bR^{n+1},$ is {\em affiliated} with a norm $\mu$ of an eigenvalue if 
\[\vec v \in {\mathcal{R}}_{\mu}(g):= \bigoplus_{i \in \{j| |\lambda_{j}| = \mu\}} C_{\lambda_{i}}(g) \cap \bR^{n+1}.\] 
Let $\mu_{1}, \dots, \mu_{l}$ denote the set of distinct norms of eigenvalues of $g$. 
We also have $\bR^{n+1} = \bigoplus_{i=1}^{l} {\mathcal{R}}_{\mu_{i}}(g)$. 
Here, ${\mathcal{R}}_{\mu}(g) \ne \{0\}$ if $\mu$ equals $|\lambda_{i}|$ for at least one $i$. 
\end{definition}

Let $V^{i_{0}+1}_\infty$ denote the subspace of $\bR^{n+1}$ corresponding 
to $\SI^{i_{0}}_\infty$. 
By { the $g$-invariance of} $\SI^{i_{0}}_\infty$, 
if \[{\mathcal{R}}_{\mu}(g)\cap V^{i_{0}+1}_\infty \ne \emp\] for some finite collection $J$, 
then ${\mathcal{R}}_{\mu}(g) \cap V^{i_{0}+1}_\infty$ always contains a $\bC$-eigenvector of $g$. 

\begin{definition}\label{defn:eig} 
Let $\Sigma_{\tilde E}$ be the end orbifold of a nonproperly convex and p-R-end $\tilde E$ of a 
strongly tame properly convex $n$-orbifold $\orb$.
Let $\bGamma_{\tilde E}$ 
be the p-end fundamental group. 
We fix a choice of a Jordan decomposition of $g$ for each $g \in \bGamma_{\tilde E}$. 
\begin{itemize}
\item Let $\lambda_1(g)$ denote the largest norm of the eigenvalue of $g \in \bGamma_{\tilde E}$ affiliated
with $\vec{v} \ne 0$, $[\vec{v}] \in \SI^n - \SI^{i_0}_\infty$,
i.e., \[\lambda_{1}(g):= \max\{\mu\,| \exists\, \vec v \in  {\mathcal{R}}_{\mu}(g) -  V^{i_0+1}_\infty\}.\]
\item Also, let $\lambda_{n+1}(g)$ denote the smallest one affiliated with a nonzero vector $\vec v$, $[\vec{v}] \in \SI^n - \SI^{i_0}_\infty$,
i.e., \[\lambda_{n+1}(g):= \min \{\mu\,| \exists \vec v \in {\mathcal{R}}_{\mu(g)} - V^{i_0+1}_\infty\}. \]
\item Let $\lambda(g)$ be the largest of the norms of the eigenvalues of $g$ with $\bC$-eigenvectors 
of form $\vec v$, $[\vec v] \in \SI^{i_0}_\infty$
and $\lambda'(g)$ the smallest such one. 
\end{itemize} 
\end{definition}



Suppose that $K$ has a decomposition into $K_1 * \cdots * K_{l_0}$ for properly convex domains 
$K_i$, $i =1, \dots, l_0$. 
Let $K_i, i=1, \dots, s$, be the ones with dimension $\geq 2$.
By Theorem 1.1 of \cite{Ben2}, 
$N_K$ is virtually isomorphic to a cocompact subgroup of 
the product \[ \bZ^{l_0-1} \times \Gamma_1 \times \dots \times \Gamma_{s}\]
where $\Gamma_i$ is obtained from $N_K$ by restricting to $K_i$ and 
$\bZ^{l_{0}-1}$ is a free abelian group of finite rank.

We will assume that the p-end fundamental group $\pi_{1}(\tilde E)$ satisfies 
the {transverse weak middle eigenvalue} condition for NPNC-ends: 
\begin{definition}\label{defn:weakmec}
Let $\tilde E$ be an NPNC-p-R-end. 
Let $\bar \lambda(g)$ denote the largest norm of the eigenvalues 
of $g \in \Gamma_{\tilde E}$. 
Let $\lambda_{\bv_{\tilde E}}(g)$ denote the eigenvalue of $g$ at $\bv_{\tilde E}$. 
The p-end $\tilde E$ satisfies 
the {{\em transverse weak middle eigenvalue condition}} for $\tilde E$ if for each element $g$ of $\pi_{1}(\tilde E)$, 
\begin{equation} \label{eqn:mec}
\bar \lambda(g) \geq \lambda_{1}(g) \geq \lambda_{\bv_{\tilde E}}(g) \hbox{ holds. } 
\end{equation} 
Also, the end $E$ of $\orb$ satisfies the same condition if a corresponding p-end does. 
\end{definition}










The following proposition is needed to understand the NPNC-ends.  
The norms of eigenvalues associated with $\SI^{i_{0}}_{\infty}$ are bounded above by the maximal norm associated with outside 
$\SI^{i_{0}}_{\infty}$. 
\begin{proposition}\label{prop-eigSI}  
Let $\Sigma_{\tilde E}$ be the end orbifold of a NPNC p-R-end $\tilde E$ of a
strongly tame properly convex $n$-orbifold $\orb$. 
Suppose that $\torb$ in $\SI^n$ {\rm (}resp. $\bR P^n${\rm )} 
covers $\orb$ as a universal cover. 
Let $\bGamma_{\tilde E}$ be the p-end fundamental group satisfying the {transverse weak middle eigenvalue} condition
for the NPNC-p-R-end $\tilde E$.
Let $g \in \bGamma_{\tilde E}$. 
Then
\[\lambda_1(g) \geq \lambda(g) \geq  \lambda'(g) \geq \lambda_{n+1}(g)\]
holds. 
\end{proposition}

\subsection{Basic hypotheses to derive splitting}

We will assume some hypotheses and derive consequences. 
Afterwards, we will show that our ends satisfy these conditions. 

Let $\CN$ denote an $i_{0}$-dimensional partial parabolic group with 
elements of form $\CN(\vec{v})$ given as follows:
\renewcommand{\arraystretch}{1.2}
\begin{equation} \label{eqn-nilmatstd}
\newcommand*{\temp}{\multicolumn{1}{r|}{}}
\CN(\vec{v}):= \left( \begin{array}{ccccccc}         \Idd_{n-i_0-1} & 0 &\temp &  0 & 0& \dots & 0 \\ 
                                                       \vec{0}           & 1  &\temp &  0  & 0&  \dots & 0\\
                                                       \cline{1-7}
                                                       \vec{c}_1(\vec{v})    & {v}_1 &\temp & 1   & 0 &  \dots & 0 \\
                                                       \vec{c}_2(\vec{v})   & {v}_2 &\temp & 0   & 1 & \dots & 0\\
                                                       \vdots   & \vdots &\temp & \vdots & \vdots & \ddots & \vdots \\
                                                       \vec{c}_{i_0+1}(\vec{v})  & \frac{1}{2}||\vec{v}||^2& \temp & {v}_1 & v_2 & \dots  & 1
                                                        
                                    \end{array} \right)
 \end{equation}
where $||v||$ is the norm of $\vec{v} = (v_1, \cdots, v_{i_{0}}) \in \bR^{i_0}$. 
We require $\CN$ to be a group and 
\[\bR^{i_{0}} \ra \CN, \vec{v} \ra \CN(\vec{v}) \]
be an isomomorphism. 
We note $\CN(\vec{v}) = \prod_{j=1}^{i_0} \CN(e_j)^{v_j}$.  
By the way we defined this, 
\[\vec{c}_k:\bR^{i_0} \ra \bR^{n-i_0-1} \hbox{ for } k=1, \dots, i_0, \] 
are  linear functions of $\vec{v}$ 
defined as $\vec{c}_k(\vec{v}) = \sum_{j=1}^{i_0} \vec c_{kj} v_j$ for $\vec{v} = (v_1, v_2, \dots, v_{i_0})$
so that we form a group. (We do not need the property of $\vec{c}_{i_0+1}$ at the moment.)

We denote by $C_1(\vec{v})$ the $(n-i_0-1) \times i_0$-matrix given by 
the matrix with rows $\vec{c}_j(\vec{v})$ for $j= 1, \dots, i_0$
and by $c_2(\vec{v})$ the row $(n-i_0-1)$-vector $\vec{c}_{i_0+1}(\vec{v})$. 
The lower-right $(i_0+1)\times (i_0+1)$-matrix in the
matrix is in the {\em standard matrix form}. 

Since $\SI^{i_0}_\infty$ is invariant, the element
$g$, $g\in \bGamma_{\tilde E}$, can be put into a {\em standard} form 
\renewcommand{\arraystretch}{1.2}
\begin{equation}\label{eqn-matstd}
\newcommand*{\temp}{\multicolumn{1}{r|}{}}
\left( \begin{array}{ccccccc} 
S(g) & \temp & s_1(g) & \temp & 0 & \temp & 0 \\ 
 \cline{1-7}
s_2(g) &\temp & a_1(g) &\temp & 0 &\temp & 0 \\ 
 \cline{1-7}
C_1(g) &\temp & a_4(g) &\temp & A_5(g) &\temp & 0 \\ 
 \cline{1-7}
c_2(g) &\temp & a_7(g) &\temp & a_8(g) &\temp & a_9(g) 
\end{array} 
\right)
\end{equation}
where $S(g)$ is an $(n-i_0-1)\times (n-i_0-1)$-matrix
and $s_1(g)$ is an $(n-i_0-1)$-column vector, 
$s_2(g)$ and $c_2(g)$ are $(n-i_0-1)$-row vectors, 
$C_1(g)$ is an $i_0\times (n-i_0-1)$-matrix, 
$a_4(g)$  is an  $i_0$-column vector, 
$A_5(g)$ is an $i_0\times i_0$-matrix, 
$a_8(g)$ is an $i_0$-row vector, 
and $a_1(g), a_7(g)$, and $ a_9(g)$ are scalars. 
(We can do this for every $g\in \bGamma_{\tilde E}$ simultaneously.) 

Denote 
\[ \hat S(g) =
\left( \begin{array}{cc} 
S(g) & s_1(g)\\ 
s_2(g) & a_1(g)
\end{array} \right),\]
and is called a semisimple part of $g$.  

\begin{hypothesis} \label{h-norm}
\begin{itemize} 
\item  Let $K$ be defined as above for a p-R-end $\tilde E$.
Assume that $K^{o}/N_K$ is a compact set.
\item $\bGamma_{\tilde E}$ satisfies the {transverse weak middle eigenvalue} condition.
And elements are in the matrix form { \eqref{eqn-matstd}} under a common coordinate system. 
\item A group $\CN$ of form \eqref{eqn-nilmatstd} 
acts on each hemisphere with boundary $\SI^i_{\infty}$, and fixes
$\bv_{\tilde E} \in \SI^i_\infty$. 
\item The p-end fundamental group $\bGamma_{\tilde E}$ normalizes 
$\CN$ {both} in the above coordinate system. 
\item $\CN$ acts on a p-end neighborhood $U$ of $\tilde E$.
\item $\CN$ acts on the space of $i_{0}$-dimensional leaves of $\tilde \Sigma_{\tilde E}$
by an induced action. 
\end{itemize}
\end{hypothesis}

We show the following: 

Let $a_5(g)$ denote $\left| \det(A^5_g)\right|^{\frac{1}{i_0}}$.
Define 
\[\mu_g:= \frac{a_5(g)}{a_1(g)} = \frac{a_9(g)}{a_5(g)} \hbox{ for } g \in \bGamma_{\tilde E}\]
by following Lemma \ref{lem-similarity}. 

\begin{lemma}[Similarity]  \label{lem-similarity}
Assume Hypothesis \ref{h-norm}. 
Then 
any element $g \in \bGamma_{\tilde E}$ 
induces an $(i_0\times i_0)$-matrix $M_g$ given by
\[g \CN(\vec{v}) g^{-1} = \CN(\vec{v}M_g) \hbox{ where } \] 
\[M_g = \frac{1}{a_1(g)} (A_5(g))^{-1} = \mu_g O_5(g)^{-1} \]
for $O_5(g)$ in a compact Lie group $G_{\tilde E}$, and 
the following hold. 
\begin{itemize} 
\item $(a_5(g))^2 = a_1(g) a_9(g)$ or equivalently $\frac{a_5(g)}{a_1(g)}= \frac{a_9(g)}{a_5(g)}$.
\item Finally, $a_1(g), a_5(g),$ and $a_9(g)$ are all nonzero. 
\end{itemize}  
\end{lemma}

\begin{lemma}[$K$ is a cone]  \label{lem-conedecomp1}
Assume Hypothesis \ref{h-norm}. 
Then the following hold\,{\rm :} 
\begin{itemize} 
\item $K$ is a cone over an 
$(n-i_0-2)$-dimensional properly convex domain $K''$. 
\item The rows of $(C_1({\vec{v}}), \vec{v}^T)$ are proportional to a single vector and 
we can find a coordinate system where $C_1({\vec v}) = 0$
not changing any entries of the lower-right $(i_0+2)\times (i_0+2)$-submatrices for all $\vec v \in \bR^{i_0}$. 
\item We can find a common coordinate system where 
\begin{equation}\label{eqn-O5coor}
O_5(g)^{-1} = O_5(g)^T, O_5(g) \in O(i_0), s_1(g) = s_2(g) = 0 \hbox{ for all } g \in \bGamma_{\tilde E}.
\end{equation} 
\item In this coordinate system, we have
\begin{equation} \label{eqn-conedecomp1}
 a_9(g) c_2({\vec{v}})  
 = c_2({\mu_g\vec{v}O_5(g)^{-1}}) S(g)  + \mu_g \vec{v} O_5(g)^{-1} C_1(g).
\end{equation} 


\end{itemize}
\end{lemma}

\begin{lemma} \label{lem-matrix}
Assume Hypothesis \ref{h-norm}. 
Then we can find coordinates so that the following holds for all $g \in \Gamma_{\tilde E}$\, {\rm :} 
\begin{align} 
\, \frac{a_{9}(g)}{a_{5}(g)} O_5(g)^{-1} a_4(g) &= a_8(g)^{T} \hbox{ or }\,  \frac{a_{9}(g)}{a_{4}(g)}   a_4(g)^T O_5(g) = a_8(g), \, \\
\, \hbox{If } \mu_{g} =1, \hbox{ then } &
a_1(g) = a_9(g)  = \lambda_{\bv_{\tilde E}}(g) \hbox{ and } 
A_5(g) = \lambda_{\bv_{\tilde E}}(g) O_5(g). 
\end{align}
\end{lemma} 

Finally, we show the ``splitting''.
\begin{proposition}[Splitting] \label{prop-decomposition}
Assume Hypothesis \ref{h-norm}. 
Suppose additionally the following\,{\rm :} 
\begin{itemize} 
\item Suppose that $a_{1}(g) \geq a_{5}(g), {a_{9}(g)}$ whenever $a_{1}(g)$ is the largest norm of the eigenvalues 
of the semisimple part  $\hat S(g)$ of $g$ for each $g \in \bGamma_{\tilde E}$. 
\item $K = \{k\} \ast K''$ a strict join, and $K^{o}/N_{K}$ is compact. 
\item A center of $\bGamma_{\tilde E}$ maps to $N_{K}$
going to a Zariski dense group of the virtual center of $\Aut(K)$.  
\end{itemize} 
Then $K''$ embeds projectively in the closure of $\Bd \torb$
invariant under $\bGamma_{\tilde E}$, and 
one can find a coordinate system so that for every $\CN(\vec v)$ and each element $g$ of 
$\bGamma_{\tilde E}$ is written so that
\begin{itemize}
\item $C_1(\vec v)=0, c_2({\vec v})=0$, and 
\item $C_1(g)=0$ and $c_2(g) = 0$.
\end{itemize}
\end{proposition}

As a consequence, we obtain 
\begin{equation} \label{eqn-form1}
g = \left( \begin{array}{cccc} 
 S(g) & 0 & 0 & 0 \\ 
 0 & a_1(g) & 0 & 0 \\ 
 0& a_4(g) & a_5(g) O_5(g) & 0  \\
0 & a_7(g) & a_8(g) & a_9(g) 
\end{array} \right), 
\end{equation} 
\begin{equation} \label{eqn-form2}
 \CN(\vec{v}) = \left( \begin{array}{cccc} 
 \Idd & 0 & 0 & 0 \\ 
 0 &    1 & 0 & 0 \\ 
0 & \vec{v}^T & \Idd & 0 \\ 
0 & \frac{1}{2} ||\vec{v}||^2 & \vec{v} & 1 
 \end{array} \right). 
 \end{equation}

Thus,  when $\mu_{g}=1$ for all $g \in \bGamma_{\tilde E}$, by 
taking a finite index subgroup of $\bGamma_{\tilde E}$, 
we conclude that each $g \in \bGamma_{\tilde E} $ has the form 
\begin{equation} \label{eqn-formgii}
\newcommand*{\temp}{\multicolumn{1}{r|}{}}
\left( \begin{array}{ccccccc} 
S(g) & \temp & 0 & \temp & 0 & \temp & 0 \\ 
 \cline{1-7}
0 &\temp & \lambda_{\bv_{\tilde E}}(g) &\temp & 0 &\temp & 0 \\ 
 \cline{1-7}
0 &\temp & \lambda_{\bv_{\tilde E}}(g)\vec{v}^T_g &\temp & \lambda_{\bv_{\tilde E}}(g) O_5(g) &\temp & 0 \\ 
 \cline{1-7}
0 &\temp & a_7(g) &\temp & \lambda_{\bv_{\tilde E}}(g) \vec{v}_g O_5(g) &\temp & \lambda_{\bv_{\tilde E}}(g)
\end{array} 
\right) 
\end{equation}
for $\vec{v}_{g} \in \bR^{i_{0}}$,
an $(n-i_{0}-1)\times (n-i_{0}-1)$-matrix $S(g)$, and
an orthogonal $i_{0}\times i_{0}$-matrix $O_{5}(g)$. 

\subsection{Joined R-ends and quasi-joined R-ends} \label{subsec-join}

We will now discuss about joins and their generalizations in depth in this subsection. 


\begin{hypothesis}[$\mu_{g}\equiv1$]\label{h-qjoin} 
Let $G$ be a p-end fundamental group. 
We continue to assume Hypothesis \ref{h-norm} for $G$.  
\begin{itemize} 
\item Every $g \in \Gamma \ra M_g$ is 
so that $M_g$ is in a fixed compact group $O(i_0)$. Thus, $\mu_g = 1$ identically. 
\item $G$ acts on the subspace $\SI^{i_0}_\infty$ containing 
$\bv_{\tilde E}$ and the properly convex 
domain $K_{m_0}''$ in the subspace $\SI^{n-i_0-2}$ 
disjoint from $\SI^{i_0}_\infty$. 
\item $\CN$ acts on these two subspaces fixing every points 
of $\SI^{n-i_0-2}$.
\end{itemize} 
\end{hypothesis} 




We assumed $\bv_{\tilde E}$ to have coordinates $[0, \dots, 0, 1]$.
$\SI^{n-i_0-2}$ contains the standard points $[e_i]$ for $i=1, \dots, n-i_0 -1$ 
and $\SI^{i_0+1}$ contains $[e_i]$ for $i=n-i_0, \dots, n+1$. 
Let $H$ be the open $n$-hemisphere defined by $x_{n-i_0} > 0$. Then by convexity of $U$, 
we can choose $H$ so that $K'' \subset H$ and $\SI^{i_0}_\infty \subset \clo(H)$. 

Assume Hypothesis \ref{h-qjoin}. 
Then elements of $\CN$ have the forms of equation \eqref{eqn-nilmatstd} with 
\[C_1(\vec{v})=0, c_2(\vec{v})=0 \hbox{ for all } \vec{v} \in \bR^{i_0}\] 
and the group $G$ of elements of forms of equation \eqref{eqn-formgii} with 
\[s_1(g) =0, s_2(g) = 0, C_1(g) = 0, \hbox{ and } c_2(g) = 0.\] 
We assume further that $O_5(g) = \Idd_{i_0}$. 


Again we recall the projection $\Pi_K: \SI^n - \SI^{i_0}_\infty \ra \SI^{n-i_0-1}$. 
$G$ has an induced action on $\SI^{n-i_0 -1}$ and
acts on a properly convex set $K''$ in $\SI^{n-i_0-1}$ so that 
$K$ equals a strict join $k* K''$ for $k$ corresponding to a great sphere $\SI^{i_0+1}_{k}$. 


\begin{figure}
\centering
\includegraphics[height=6cm]{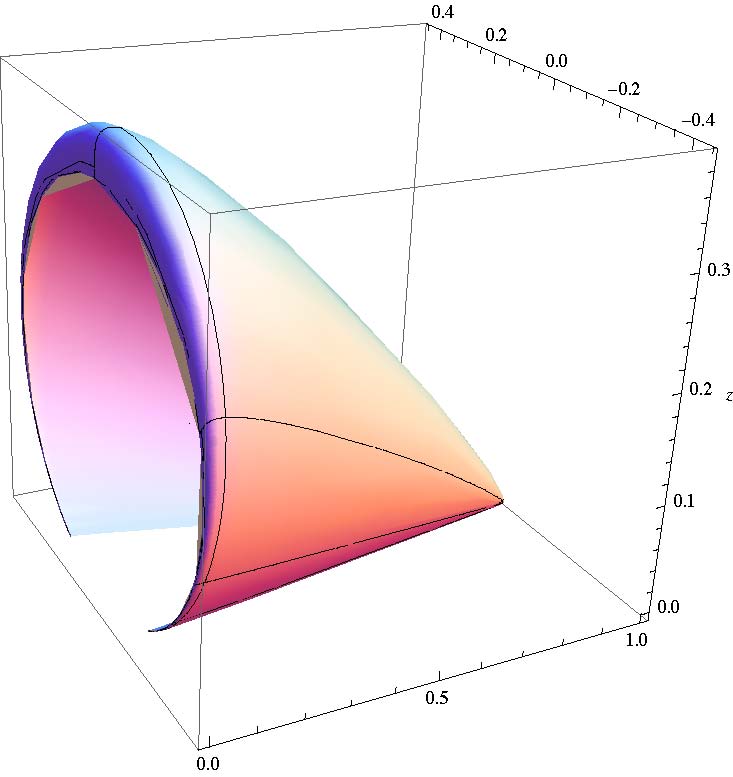}
\caption{A figure of a quasi-joined p-R-end neighborhood}
\label{fig:quasi-j}
\end{figure}

We define invariants from the form of equation \eqref{eqn-formgii}
\[\alpha_7(g):= \frac{a_7(g)}{\lambda_{\bv_{\tilde E}}(g)} - \frac{||\vec{v}_g||^2}{2} \]
for every $g \in G$. 
\[\alpha_7(g^n) = n \alpha_7(g) \hbox{ and }
\alpha_7(gh) = \alpha_7(g) + \alpha_7(h), \hbox{ whenever }g, h, gh \in G.\] 

Here $\alpha_7(g)$ is determined by factoring  
the matrix of $g$ into commuting matrices of form

\begin{multline}\label{eqn:kernel}
\newcommand*{\temp}{\multicolumn{1}{r|}{}}
\left( \begin{array}{ccccccc}
\Idd_{n-i_0-1} & \temp & 0 & \temp & 0 & \temp & 0 \\   
 \cline{1-7}
0                    &\temp & 1 & \temp & 0 & \temp & 0  \\ 
 \cline{1-7}
0                    & \temp &   0 &\temp & \Idd_{i_0} &\temp & 0 \\ 
 \cline{1-7}
0                    &\temp &  \alpha_7(g) &\temp & \vec{0} &\temp & 1 \\ 
\end{array} 
\right) \times \\   \newcommand*{\temp}{\multicolumn{1}{r|}{}}
\left( \begin{array}{ccccccc}
S_g & \temp & 0 & \temp & 0 &\temp & 0 \\ 
\cline{1-7}
0 & \temp & \lambda_{\bv_{\tilde E}}(g) & \temp & 0 & \temp & 0  \\ 
 \cline{1-7}
0& \temp & \lambda_{\bv_{\tilde E}}(g) \vec{v}_g &\temp & \lambda_{\bv_{\tilde E}}(g) O_5(g) &\temp & 0 \\ 
 \cline{1-7}
0& \temp & \lambda_{\bv_{\tilde E}}(g)  \frac{||\vec{v}||^2}{2}  
&\temp & \lambda_{\bv_{\tilde E}}(g) \vec{v}_g O_5(g) &\temp & \lambda_{\bv_{\tilde E}}(g) \\ 
\end{array} 
\right).
\end{multline}

\begin{remark} \label{rem:alpha7}
We give more explanation. 
Recall that the space of segments in a hemisphere $H^{i_0+1}$ with the vertices $\bv_{\tilde E}, \bv_{\tilde E-}$ 
forms an affine space $A^{i_{0}}$ one-dimension lower, and the group $\Aut(H^{i_0+1})_{\bv_{\tilde E}}$ of projective {automorphisms }
of the hemisphere fixing $\bv_{\tilde E}$ maps to $\Aff(A^{i_0})$ with {the kernel equal to 
the group of transformations of $(i_0+2)\times (i_0+2)$-matrix forms}
\renewcommand{\arraystretch}{1.2}
\begin{equation}\label{eqn:kernel2}
\newcommand*{\temp}{\multicolumn{1}{r|}{}}
\left( \begin{array}{ccccc} 
1 & \temp & 0 & \temp & 0  \\ 
 \cline{1-5}
0 &\temp & \Idd_{i_0} &\temp & 0 \\ 
 \cline{1-5}
b &\temp & \vec{0} &\temp & 1 \\ 
\end{array} 
\right)
\end{equation}
where $\bv_{\tilde E}$ is given coordinates $[0, 0, \dots, 1]$ and 
a center point of $H^{i_0+1}_l$ the coordinates $[1, 0, \dots, 0]$. 
In other words the transformations are of form 
\begin{align}\label{eqn:temp}
\left[
\begin{array}{c}
 1\\
 x_1  \\
 \vdots \\ 
 x_{i_0} \\ 
 x_{i_0+1}   
\end{array}
\right]
\mapsto 
\left[
\begin{array}{c}
 1 \\
 x_1\\
 \vdots \\ 
 x_{i_0} \\ 
 x_{i_0+1}+b
\end{array}
\right]
\end{align}
and hence $b$ determines the kernel element. 
Hence $\alpha_7(g)$ indicates the translation towards $\bv_{\tilde E}=[0,\dots, 1]$. 
\end{remark}

We define $G_+$ to be a subset of $G$ consisting of elements $g$ so that 
the largest norm $\lambda_1(g)$ of the eigenvalue occurs at the vertex $k$. 
We also assume that $\lambda_1(g) = \lambda_{\bv_{\tilde E}}(g)$
with all other norms of the eigenvalues occurring at $K''$ is strictly less than $\lambda_{\bv_{\tilde E}}(g)$.  
The second largest norm $\lambda_2(g)$ of the eigenvalue occurs at the complementary subspace $K''$ of $k$ in $K$.
Thus, $G_+$ is a semigroup.      
The condition that $\alpha_7(g) \geq 0$ for $g \in G_+$ is said to be the 
{\em positive translation condition}.


Again, we define \[\mu_7(g) : = \frac{\alpha_7(g)}{\log\frac{\lambda_{\bv_{\tilde E}}(g)}{\lambda_2(g)}}\] where 
$\lambda_2(g)$ denote the second largest norm of the  eigenvalues of $g$ and 
we restrict $g \in G_+$. 
The condition $\mu_7(g) > C_0, g \in  G_+$ for a uniform constant $C_0$
is called the {\em uniform positive translation condition}. 







 
 The assumptions below are just Hypotheses \ref{h-norm} and \ref{h-qjoin}. 
We fully state for a change. 
 
\begin{proposition}[Quasi-joins] \label{prop-qjoin}
Let $\Sigma_{\tilde E}$ be the end orbifold of an NPNC R-end $\tilde E$ of a
strongly tame properly convex $n$-orbifold $\orb$. 
Let $G$ be the p-end fundamental group. 
Let $\tilde E$ be an NPNC p-R-end
and $G$ and $\mathcal N$ acts on a p-end-neighborhood 
$U$ fixing $\bv_{\tilde E}$. Let $K, K'',$ and $\SI^{i_0}_\infty$ be as above.
We assume that $K^o/G$ is compact, $K= K''* k$ in $\SI^{n-i_0}$ with $k$ 
corresponding to a great sphere $\SI^{i_0+1}_{k}$ under the projection $\Pi_K$. 
Assume that 
\begin{itemize} 
\item $G$ satisfies the {transverse weak middle eigenvalue} condition. 
\item $\mu_{g} = 1$ for all $g \in G$. 
\item Elements of $G$ and $\CN$ are of form of equations \eqref{eqn-form1} and \eqref{eqn-form2}.
with \[C_1(\vec{v}) = 0, c_2(\vec{v})=0, C_1(g) = 0, c_2(g) =0\]
for every $\vec{v} \in \bR^{i_0}$ and $g \in G$.
\item $G$ normalizes $\CN$, {$\CN$ acts on $U$, }
and each leaf of $\mathcal{F}_{\tilde E}$ of $\tilde \Sigma_{\tilde E}$. 
\end{itemize} 
Then
\begin{itemize} 
\item[(i)] The condition $\alpha_7 \geq 0$ is a necessary condition 
that $G$ acts on a properly convex domain in $H$.
\item[(ii)] The uniform positive translation condition is equivalent to the existence of
a properly convex p-end-neighborhood $U'$ 
satisfying \[\clo(U')\cap \SI^{i_0+1}_k =\{\bv_{\tilde E}\}.\]
\item[(iii)] $\alpha_7$ is identically zero if and only if $U$ is a join and $U$ is properly convex. 
\end{itemize}
\end{proposition}

\begin{definition}\label{def-qjoin}
If $\tilde E$ satisfies the case (ii) of Proposition \ref{prop-qjoin}, then
$\tilde E$ is said to be a {\em quasi-joined p-R-end} and $G$ now is called a {\em quasi-joined p-end group}.
An end with an end neighborhood that is covered by a p-end neighborhood of 
such a p-R-ends is also called a {\em quasi-joined R-end}. 
Similarly, a joined R-end is given by the case (iii). 
\end{definition}



\subsection{The classification of the NPNC-ends} \label{sub:NPNC}

We will show using Proposition \ref{prop-qjoin} that NPNC R-ends satisfying the { transverse weak middle 
eigenvalue} conditions are quasi-joins. That is, Theorem \ref{thm-thirdmain} is shown. 
These quasi-joined R-ends do not satisfy the uniform middle eigenvalue condition essentially because of 
the existence of the unipotent group. 

The following theorem requires a Zariski density condition for the center.
We question whether this assumption can be dropped.
If the fundamental group is virtually abelian, this will be true. 


\begin{theorem}[Theorem 1.1 of \cite{End3}]\label{thm-thirdmain} 
Let $\mathcal{O}$ be a strongly tame  properly convex real projective orbifold.
Assume that the holonomy group of $\mathcal{O}$ is strongly irreducible.
\begin{itemize}
\item Let $\tilde E$ be an NPNC p-R-end. 
\item Let $K$ be the convex $n-i_{0}-1$-dimensional domain that is the space of 
$i_{0}$-dimensional  affine spaces foliating 
the universal cover $\tilde \Sigma_{\tilde E}$ of the end orbifold $\Sigma_{\tilde E}$. 
\end{itemize} 
We assume that 
\begin{itemize}
\item a virtual center of $\bGamma_{\tilde E}$ goes to a Zariski dense subgroup of  
the virtual center of the group $\Aut(K)$ of projective automorphisms of $K$ and 
\item the p-end fundamental group $\pi_{1}(\tilde E)$ satisfies the {transverse weak middle eigenvalue} condition
for NPNC-ends. 
\end{itemize}
Then $\tilde E$ is a {quasi-joined} p-R-end. 
\end{theorem}


We will discuss the idea of the proof: 
First we discuss the case when $N_K$ is discrete. Here, $N$ is virtually abelian 
and is conjugate to a discrete cocompact subgroup of a subgroup of an $i_{0}$-dimensional  cusp group. 
Here, each fiber $\Pi^{-1}_{K}(x), x \in K^{o}$ covers a compact $(i_{0}+1)$-dimensional orbifold.
The fibers have complete-affine structures of an affine space of dimension $i_{0}+1$. 
By {the proper-convexity} of $\torb$,  the inequality of Proposition \ref{prop-eigSI} and  
Theorem \ref{thm-mainaffine} imply that the fibers are $i_{0}+1$-dimensional horoballs. 
Hence, by taking a Zariski closure and Malcev's results, 
there exists an $i_{0}$-dimensional partial parabolic group as described by $\CN$ in equation \eqref{eqn-nilmatstd}.
Then $\bGamma_{\tilde E}$ then virtually normalizes $\CN$. 
By computations involving the normalization conditions, 
we show that the above exact sequence is in a block form virtually
by Proposition \ref{prop-decomposition},
and we show that 
the p-R-ends are {joined p-R-ends or quasi-joined p-R-ends.} 

We discuss the case when $N_K$ is not discrete. Here, there is a foliation 
by $i_{0}$-dimensional complete-affine spaces as above. 
The space of leaves has a transversal Hilbert metric induced from $K^{o}$. 
We can modify the Hilbert metric to a transversal Riemannian metric 
to make the foliation into a Riemannian foliation. (See \cite{End3}). 
The leaf closures are compact submanifolds $V_l$ 
by the theory of Molino \cite{Mol} on Riemannian foliations. These are fibers
in some singular fibrations. 
We use some estimate of Proposition \ref{prop-eigSI} to show that each leaf 
is of polynomial growth. 
This shows that the identity component of the closure of
$N_K$ is abelian and $\pi_1(V_l)$ for above fiber $V_{l}$ is solvable using the work of Carri\`ere \cite{Car}. 
One can then take the syndetic closure to 
obtain a bigger group that act transitively on each leaf. 
We find a normal $i_{0}$-dimensional cusp group acting on each leaf transitively. Then we show that 
the p-R-end also splits virtually by Proposition \ref{prop-decomposition}.

Finally, suppose that $\tilde E$ is a joined end. 
For both of these cases of $N_{K}$, we show that the orbifold has to be reducible
by considering the limit actions of some elements in the joined R-ends. 
This proves that the joined R-end does not exist, proving Theorem \ref{thm-thirdmain}. 



\begin{remark}\label{rem-otherv} 
We could replace the virtual center condition and the {transverse weak middle eigenvalue} condition 
of Theorem \ref{thm-thirdmain} 
with the weak uniform middle eigenvalue condition for NPNC-ends in \cite{End2}. 
This is to be expored also in \cite{End2}. 
\end{remark}



\subsection{Complete affine ends again}

\begin{corollary}[non-cusp complete-affine p-ends] \label{cor-caseiii}
Let $\orb$ be a strongly tame properly convex $n$-orbifold. 
Suppose that $\tilde E$ is a complete-affine p-R-end of its universal cover $\torb$ in $\SI^n$ or 
in $\bR P^n$. Let $\bv_{\tilde E} \in \SI^n$  be the p-end vertex with the p-end fundamental group 
$\bGamma_{\tilde E}$.  
Suppose that $\tilde E$ is not a cusp p-end. Then we can choose 
a different point as the p-end vertex for $\tilde E$ so that $\tilde E$ is a quasi-joined p-R-end with fiber homeomorphic 
to cells of dimension $n-2$.
\end{corollary} 
\begin{proof}
We will use the terminology of the proof of Theorem \ref{thm-comphoro}. 
Theorem \ref{thm-comphoro} 
shows that $\bGamma_{\tilde E}$ is virtually nilpotent and with at most two norms of eigenvalues
for each element. 
By taking a finite-index subgroup, we assume that $\bGamma_{\tilde E}$ is nilpotent. 
Let $Z$ be the Zariski closure, a nilpotent Lie group. 
Since $\bGamma_{\tilde E}\cap Z$ is a cocompact lattice in $Z$, 
and $\bGamma_{\tilde E}$ has the virtual cohomological dimension $n-1$, 
it follows that $Z$ is $(n-1)$-dimensional.  
Then, $Z$ fixes $\bv_{\tilde E}$ and 
we have a homomorphism 
\[\lambda_{\bv_{\tilde E}}: Z \ni g \ra \lambda_{\bv_{\tilde E}}(g) \in \bR.\]
Let $N$ denote the kernel of the homomorphism. 
Then $N$ is virtually a unipotent Lie group of dimension $n-2$ by  the proof of Theorem \ref{thm-comphoro}. 

Also, $Z$ acts transitively on the complete affine space $\tilde \Sigma_{\tilde E}$
since $\bGamma_{\tilde E}$ acts cocompactly on it. 
We modify $U$  to 
\[ \bigcap_{g\in Z} g(U) = \bigcap_{g\in F}g(U).\] 
This is a properly convex open set.
Since $F$ is the compact fundamental domain in $Z$ under the action of $\bGamma_{\tilde E}$,
the modified $U$ is not empty. 
We may assume that $Z$ acts on a properly convex p-end neighborhood $U$ of $\tilde E$ 

If we have $n=2$, then $\bGamma_{\tilde E}$ is cyclic acting on a properly convex domain $U$
and we see that it is a quasi-hyperbolic group in the sense of \cite{cdcr2}. 
The end is a quasi-joined one with fiber dimension $0$.
Suppose $n > 2$. 

Let $A$ be a hyperspace containing $\bv_{\tilde E}$ in direction of $\Bd \Sigma_{\tilde E} = \SI^{n-2}\subset \SI^{n-1}_{\bv_{\tilde E}}$. 
Then $U_{A}:= A \cap \clo(U)$ is a properly convex compact set where $Z$ acts on.
By Lemma \ref{lem-Ucpt}, $U_{A}^{o}/Z$ is compact. 
By Lemma \ref{lem-Ben2}, $U_{A}$ is a properly convex segment. 

\begin{lemma}\label{lem-Ben2} 
Let a connected nilpotent Lie group $S$ act cocompactly on a properly convex open domain $J$
where each element has at most two eigenvalues. 
Then the dimension of the domain is $0$ or $1$.
\end{lemma}
\begin{proof} 
By Theorem 1.1 of Benoist \cite{Ben2},  $\clo(J)$ is a join of  the closures of the symmetric spaces
and points. Since each element of $Z$ has no more than two eigenvalues,
$\clo(J)$ cannot contain symmetric spaces of dimension $> 1$ as factors. 
Also, $\clo(J)$ is not a simplex of dimension $> 1$ similarly. 
\end{proof}

\begin{lemma} \label{lem-Ucpt} 
$U_{A}^{o}/Z$ is compact. 
\end{lemma} 
\begin{proof} 
Suppose that $\dim U_{A}^{o} = n-1$. 
If the stabilizer of $Z$ of a point of $U_{A}^{o}$ is trivial, 
then $Z$ acts transitively on $U_{A}^{o}$ and hence $U_{A}^{o}/Z$ is compact. 

Suppose that the stabilizer is not trivial. 
Then the stabilizer of $Z$ of a point of $U_{A}^{o}$ is a noncompact group
since $Z$ is a nilpotent Lie group.  Let $\langle g^{t} \rangle$ be a noncompact one-parameter stabilizer group. 
Then $\langle g^{t}\rangle $ acts trivially on the hyperspace $A$. Hence $\langle g^{t}\rangle $ is a one-parameter translation group. 
Thus, $\bigcup_{t\in \bR}g^{t}(U) \subset U$ is not properly convex. 
This is a contradiction. 

Suppose now that $\dim U_{A}^{o} = i < n-1$. 
Let $H$ denote the complete affine space $\tilde \Sigma_{\tilde E}$.
Let $L$ be an $(i+1)$-dimensional subspace containing $U_{A}$ meeting $A$ transversally. 
Let $l$ be the $i$-dimensional affine subspace of $H$ corresponding to $L$. 
Since $g(U_{A})=U_{A}, U_{A}\subset L, g(L)$ and $\dim L = i+1$, it follows that 
\[ g(L) \cap L = \langle U_{A} \rangle \hbox{ or } g(L) = L.\]
Since $\langle U_{A} \rangle \cap A = \emp$,  it follows that
\[g(l)= l \hbox{ or } g(l) \cap l = \emp.\] 
Since $Z$ acts transitively and freely on $H$, and $\dim l = i$, it follows that 
the subgroup $\hat Z := \{g \in Z| g(l) = l\}$
has the dimension $i$. 

Now $\hat Z$ acts on on $U_{A}^{o}$. As above, the stabilizer of $\hat Z$ of a point of $U_{A}^{o}$ is trivial
since $U \cap L $ is properly convex. 
Hence, $\hat Z$ acts transitively on $U_{A}^{o}$ since $\dim \hat Z = \dim U_{A}^{o}$,
and $U_{A}^{o}/\hat Z = U_{A}^{o}/Z$ is compact. 
\end{proof}

If $\dim U_{A} = 0$, then $U$ is a horospherical p-end neighborhood 
where $\bGamma_{\tilde E}$ is unimodular and cuspidal by Theorem \ref{thm-affinehoro}. 
Hence, $\dim U_{A} = 1$. 

Let $q$ denote the other end point of $U_{A}$ than $\bv_{\tilde E}$. 
Let $U_{q}:= R_{q}(U) \subset \SI^{n-1}_{q}$ denote the set of directions of segments from $q$ in $U$.
Since $U$ is convex, $U_{q}$ is a convex open domain. 
The space $U_{q}$ is diffeomorphic to $\Bd U \cap \torb$ by projecting $\Bd U \cap \torb \ra \SI^{n-1}_{q}$ by radial rays. 
$Z$ acts on $U_{q}$ cocompactly since $Z$ acts {so on $\Bd U \cap \torb$}. 

By Theorem \ref{thm-comphoro}, $U_{q}$ is not complete-affine since 
the norm $\lambda_{q}(g)$ for some $g \in \bGamma_{\tilde E}$ has multiplicity $n-1 \ne n$. 
By Lemma \ref{lem-Ben2}, $U_{q}$ is not properly convex since $\dim U_{q} = n-1 \geq 2$. 
Thus, $q$ is the p-end vertex of a NPNC-end. Since the associated semisimple part has only two eigenvalues, 
the properly convex leaf space $K$ is $1$-dimensional by Lemma \ref{lem-Ben2} and fibers have the dimension 
$n-2=n-1-1$. 
Therefore, $U_{q}$ is foliated by $n-2$-dimensional complete-affine spaces. 
The leaf space $K$ is a properly convex segment. 

Each leaf $l$ equals $L \cap \Bd U\cap \torb$ intersected with a subspace $L$ of dimension $n-1$ containing $q$ and $l$.
These subspaces meet at a codimension-$2$ great-sphere $S_{1}$ containing $q$.
Since the leaves are disjoint on $\Bd U \cap \torb$, we have $S_{1}\subset A$.
Since $Z$ acts on $U_{q}$ and $\lambda_{\bv_{\tilde E}}(g) = \lambda_{q}(g), g \in N$, 
the Lie group $N$ acts on each complete affine leaf transitively. 
$N$ is a nilpotent Lie group since the elements are unimodular. 
We can apply Theorem \ref{thm-affinehoro} to the hyperspace $P$ containing the leaves
with a cocompact subgroup of $N$ acting on it.  
As $U \cap P$ is properly convex, $r_{P}(N)$ {is an $n-2$-dimensional cusp} group. 

Let $x_{1}, \dots, x_{n+1}$ be the coordinates of $\bR^{n+1}$.
Now give coordinates so that $q = [0, 0, \dots, 1]$ and $\bv_{\tilde E} = [1, 0, \dots, 0]$. 
Since these are fixed points, 
we obtain that elements of $N$ can be put into forms: 
\begin{equation} 
 N(\vec{v}):= \left( \begin{array}{cccc} 
 1 & 0 & 0 & 0 \\ 
 0 &    1 & 0 & 0 \\ 
0 & \vec{v}^T & \Idd & 0 \\ 
0 & \frac{1}{2} ||\vec{v}||^2 & \vec{v} & 1 
 \end{array} \right) \hbox{ for } \vec{v} \in \bR^{n-2}.
 \end{equation} 

Now, $\bGamma_{\tilde E}$ satisfies the {transverse weak middle eigenvalue} condition with respect to $q$
since $\bGamma_{\tilde E}$ has just two 
eigenvalues and $Z$ is generated by $N$ and $g^{t}$ for a nonunipotent element $g$ of $\bGamma_{\tilde E}$. 
$N$ admits an invariant Euclidean structure being a cusp group. 

Now, we come back to the affine $(n-1)$-space \hyperlink{term-lsphere}{$H = \tilde \Sigma_{\tilde E} \subset \SI^{n-1}_{\bv_{\tilde E}}$}. 
$A$ is given by $x_{2}=0$. 
A segment from $\bv_{\tilde E}$ in direction of $H \subset \SI^{n-1}_{\bv_{\tilde E}}$ is of 
form \[\big\{[t(1, 0, \dots, 0) + (1-t)(0, 1, \vec{w}, w_{n-1})] \big| t\in [0, 1], \vec{w} \in \bR^{n-2}, w_{n-1} \in \bR\big\}\]
corresponds to $(\vec{w}, w_{n-1}) \in \bR^{n-1}$ in the affine space $H$ with coordinates. 
From the above form of matrices of $N$, we obtain that
$g\in N$ given by $N(\vec{v})$ gives an affine transformation 
\[(\vec{w}, w_{n-1}) \mapsto \left(\vec{w} + \vec{v}, w_{n-1} + \frac{1}{2}||\vec{v}||^{2}\right).\]
The group $N$ acts on $H$ so that the orbits are paraboloids of dimension $n-2$ parallel to one another.
$g_{t}$ sends orbits of $N$ to orbits of $N$; i.e., paraboloids to parallel ones in a coordinate system of $H$.  
Since $g_{t}$ has different eigenvalues at $q$ and $\bv_{\tilde E}$ for $t \ne 0$, 
$g_{t}$ sends leaves to leaves without invariant leaves. 
Also, $Z$ is unimodular on $H$. 
{Each leaf has} the affine metric given as the paraboloid in $H$ using the affine differential geometry
for $H$ with an invariant parallel volume form.
The affine metric is Euclidean, and 
$g_{t}$ is an isometry between two leaves. 
We obtain a $g_{t}$-invariant Euclidean metric on $U_{q}$. 
Thus, $U_{q}$ has a $Z$-invariant Euclidean metric. 
Thus, $\bGamma_{\tilde E}$ is virtually abelian by the Bieberbach theorem 
since $\bGamma_{\tilde E}$ acts properly discontinuously on
$\partial U_{q}$ and hence on $U_{q}$. 

Thus, a virtual center of $\bGamma_{\tilde E}$ goes to a Zariski dense subgroup of the center of 
$\Aut(K)$. 
Thus, Theorem \ref{thm-thirdmain}  implies the result.
\end{proof}

\begin{corollary}[cusp and complete affine]\label{cor-cusp} 
Let $\orb$ be a strongly tame properly convex $n$-orbifold. 
Suppose that $\tilde E$ is a complete affine p-R-end of its universal cover $\torb$ in $\SI^n$ or 
in $\bR P^n$. Let $\bv_{\tilde E} \in \SI^n$  be the p-end vertex with the p-end fundamental group 
$\bGamma_{\tilde E}$. Suppose $\bGamma_{\tilde E}$ satisfies the weak middle eigenvalue condition.
Then $\tilde E$ is a complete affine R-end if and only if $\tilde E$ is a cusp R-end. 
\end{corollary} 
\begin{proof} 
Since a cusp end is horospherical end (see above Theorem \ref{thm-mainaffine}), 
we need to show the forward direction only by Theorem \ref{thm-affinehoro}. 
In the second possibility of Theorem \ref{thm-comphoro}, the norm of 
$\lambda_{\bv_{\tilde E}}$ has a multiplicity one for a nonunipotent element 
$\gamma$ with $\lambda_{\bv_{\tilde E}}(\gamma)$ equal to the maximal norm.  
Thus, the first possibility of Theorem \ref{thm-comphoro} holds. 
\end{proof}

\bibliographystyle{plain}

\begin{thebibliography}{99}

\bibitem{Ballas} S. Ballas,  
\newblock{Deformations of non-compact, projective manifolds, }
\newblock{{{\em Alg. \& Geom. Top.} 14-5 (2014), 2595--2625.}}


\bibitem{Ballas2} S. Ballas,  
\newblock{Finite volume properly convex deformations of the figure eight knot, }
\newblock{{{\em Geom. Dedicata} 178 (2015), 49--73}. }


\bibitem{BDL} S. Ballas, J. Danciger, and G. S. Lee, 
\newblock{Convex projective structures on nonhyperbolic $3$-manifolds,}
\newblock{arXiv:1508.04794 math.GT.} 

\bibitem{Baues} { O. Baues}, 
\newblock{Deformation spaces for affine crystallographic groups}, In: {\em Cohomology of groups and algebraic K-theory}, 
\newblock{ 55--129, Adv. Lect. Math. (ALM), 12, Int. Press, Somerville, MA, 2010.} 

\bibitem{Ben1} { Y. Benoist},
 \newblock{Convexes divisibles. I},
 \newblock{In {\em Algebraic groups and arithmetic}},
  {339--374}, Tata Inst. Fund. Res., Mumbai, 2004.
 
\bibitem{Ben2} { Y. Benoist},
\newblock{Convexes divisibles. II},
 \newblock{\em Duke Math. J.}, {120} (2003), 97--120.

\bibitem{Ben3} { Y. Benoist},
\newblock{Convexes divisibles. III},
\newblock{{\em Ann. Sci. Ecole Norm. Sup.} (4) 38 (2005), no. 5, 793--832. }

\bibitem{Ben4} { Y. Benoist},
\newblock{ Convexes divisibles IV : Structure du bord en dimension 3}, 
\newblock{{\em Invent. math.} 164 (2006), 249--278.}

\bibitem{Ben5} { Y. Benoist},
 \newblock{Automorphismes des c\^ones convexes},
\newblock{\em Invent. Math.}, {141} (2000), 149--193.

\bibitem{Benasym} { Y. Benoist}, 
\newblock{Propri\'et\'es asymptotiques des groupes lin\'eaires}, 
\newblock{ {\em Geom. Funct. Anal.} 7 (1997), no. 1, 1--47.}

\bibitem{BenNil} { Y. Benoist}, 
\newblock{Nilvari\'et\'es projectives},
\newblock{ {\em Comm. Math. Helv.} 69 (1994), 447--473.}

\bibitem{Benz} { J.-P., Benz\'ecri}, 
\newblock{Sur les vari\'et\'es localement affines et localement projectives},
\newblock{{\em Bull. Soc. Math. France} 88 (1960) 229--332.} 

\bibitem{Borel} { A. Borel}, 
\newblock{{\em Linear algebraic group},} 
\newblock{Springer Verlag, 2nd edition p.288 + xi, 1991.}

\bibitem{BS} {{A. Borel and J.-P. Serre}}, 
\newblock{Corners and arithmetic groups}, 
\newblock{{\em Comment. Math. Helv.} 48 (1973), 436--491.} 

 \bibitem{BH} 
  M. Bridson and A. Haef{l}iger, 
 \newblock{ {\em Metric spaces of non-positive curvature},} 
\newblock{ {Grad. Texts in Math.} Vol. 319, Springer-Verlag, New York 1999.}



\bibitem{Car} { Y. Carri\`ere}, 
\newblock{Feuilletages riemanniens \`a croissance polyn\^omiale}, 
\newblock{{\em Comment. Math. Helv.} 63 (1988), 1--20.}

\bibitem{Canary}
R. Canary, D.B.A. Epstein, and P. L. Green, 
\newblock{Notes on notes of Thurston}, 
\newblock{In: 
{\em Fundamentals of hyperbolic geometry: selected expositions}, pp. 1--115, 
London Math. Soc. Lecture Note Ser., 328, 
Cambridge Univ. Press, Cambridge, 2006.} 



\bibitem{ChCh}  { Y. Chae,  S. Choi,  and C. Park},  
\newblock{Real projective manifolds developing into an affine space}, 
\newblock{{\em Internat. J. Math.} 4 (1993), no. 2, 179--191.}

  \bibitem{Choi2004}
   { S. Choi},
    \newblock{Geometric structures on orbifolds and holonomy representations},
    \newblock{{\em Geom. Dedicata} 104 (2004), 161--199.}

 \bibitem{psconv} S.~Choi, 
\newblock{The convex and concave decomposition of manifolds with real projective structures}, 
\newblock{{\em M\'emoires SMF}, No. 78, 1999, 102 pp.}

 
  \bibitem{Choi2006}
   { S. Choi},
    \newblock{The deformation spaces of projective structures on 3-dimensional Coxeter orbifolds,}
    \newblock{{\em Geom. Dedicata} 119 (2006), 69--90.}




\bibitem{rdsv} { S. Choi}, 
\newblock{The decomposition and classification of radiant affine 3-manifolds},
\newblock{{\em Mem. Amer. Math. Soc.} 154 (2001), no. 730, viii+122 pp.}

\bibitem{cdcr1}
{ S.~Choi}, 
\newblock{Convex decompositions of real projective surfaces {{\rm {I}:}}
  $\pi$-annuli and convexity},
\newblock {\em J. Differential Geom.} 40 (1994), 165--208.

\bibitem{cdcr2}
{ S.~Choi},
\newblock{Convex decompositions of real projective surfaces {{\rm {II}:}}
  {A}dmissible decompositions},
\newblock {\em J. Differential Geom.}, 40 (1994), 239--283.

\bibitem{cdcr3}
{ S.~Choi}, 
\newblock{Convex decompositions of real projective surfaces {{\rm {III}:}}
  {F}or closed and nonorientable surfaces}, 
\newblock {\em J. Korean Math. Soc.}, 33 (1996), 1138--1171.

\bibitem{Cbook} 
{ S.~Choi}, 
\newblock{{\em Geometric structures on 2-orbifolds\,{\rm :} exploration of discrete symmetry}}, 
\newblock{MSJ Memoirs, Vol. 27. 171pp + xii, 2012}

\bibitem{endclass} 
S.~Choi,
\newblock{ The classification of radial ends of convex real projective orbifolds,}
\newblock{ arXiv:1304.1605.} 


\bibitem{End2} 
{S. ~Choi}, 
\newblock{{ {The classification of radial or totally geodesic ends of real projective orbifolds II{\rm :} properly convex ends}},} 
\newblock{ arXiv:1501.00352. } 

\bibitem{End3} 
{S. ~Choi}, 
\newblock{{ {The classification of radial or totally geodesic ends of real projective orbifolds III{\rm :} nonproperly convex convex ends}},} 
\newblock{ arXiv:1507.00809. } 

\bibitem{conv} { S. Choi}, 
\newblock{The convex real projective manifolds and orbifolds with radial or totally geodesic ends: the closedness and openness of deformations}, 
\newblock{arXiv:1011.1060}





\bibitem{book} {S. Choi,} 
\newblock{ { $\bR P^{n}$-orbifolds with ends and their deformation spaces}, }
\newblock{  {a book in preparation} }



 \bibitem{CG}
   { S. Choi and W.M. Goldman}, 
    \newblock{The deformation spaces of convex $\mathbb{RP}^2$-structures on 2-orbifolds},
    \newblock{{\em Amer. J. Math.} 127 (2005), 1019--1102.}


\bibitem{afftame} { S.~Choi and W. M. Goldman}, 
\newblock{Topological tameness of Margulis spacetimes}, 
\newblock{{{\em Amer. J. Math.} 139 (2017), 297--345. }}

\bibitem{CHL} 
{ S.~Choi, C.D.~Hodgson, and G.S.~Lee}, 
\newblock{Projective deformations of hyperbolic Coxeter 3-orbifolds}, 
\newblock{{\em Geom. Dedicata} 159 (2012), 125--167.}

\bibitem{CL} S. Choi and G. Lee, 
\newblock Projective deformations of weakly orderable hyperbolic Coxeter orbifolds, 
\newblock Geometry \& Topology 19 (2015) 1777Ð-1828.

\bibitem{CL2} D. Cooper and D. Long, 
\newblock{A generalization of the Epstein-Penner construction to projective manifolds},
\newblock{arXiv:1307.5016}{.}



 \bibitem{Cooper2006}
{    D. Cooper, D. Long, and M. Thistlethwaite},
    \newblock{Computing varieties of representations of hyperbolic 3-manifolds into ${\rm SL}(4,\mathbb R)$},
    \newblock{{\em Experiment. Math.} 15 (2006), 291--305.}

  \bibitem{CLT}
{   D. Cooper, D. Long, and M. Thistlethwaite},
    \newblock{ Flexing closed hyperbolic manifolds},
    \newblock{{\em Geom. Topol.} 11 (2007), 2413--2440.}
    
    \bibitem{CLT2} 
{    D. Cooper, D. Long, and S. Tillmann,} 
    \newblock{On convex projective manifolds and cusps}, 
    \newblock {\em Adv. Math.} 277 (2015), 181--251{.} 

    \bibitem{CLT3} 
{    D. Cooper, D. Long, and S. Tillmann,} 
\newblock{Deforming convex projective manifolds, }
\newblock{arXiv:1511.06206.} 

\bibitem{CG2} 
{ J. P. Conze and Y. Guivarch}, 
\newblock{Remarques sur la distalit\'e dans les espaces vectoriels}, 
\newblock{{\em C. R. Acad. Sci. Paris} 278 (1974), 1083--1086.}


\bibitem{CM2} { M. Crampon and L. Marquis}, 
\newblock{Finitude g\'eom\'etrique en g\'eom\'etrie de Hilbert}, 
\newblock{Ann. Inst. Fourier (Grenoble) 64 (2014), no. 6, 2299--2377. }

\bibitem{GV} 
J. de Groot and H. de Vries, 
\newblock{Convex sets in projective space}, 
\newblock{{\em Compositio Math.}, 13 (1958), 113--118.}

\bibitem{Fried86} { D. Fried}, 
\newblock{Distality, completeness, and affine structure},
\newblock{ {\em J. Differential Geometry} 24 (1986), 265--273.}

\bibitem{FG} { D. Fried and W. Goldman}, 
\newblock{Three-dimensional affine crystallographic groups}, 
\newblock{{\em Adv. Math.} 47 (1983), 1--49.}

\bibitem{FGH} { D. Fried, W. Goldman, and M. Hirsch}, 
\newblock{Affine manifolds with nilpotent holonomy}, 
\newblock{{\em Comment. Math. Helv.} 56 (1981), 487--523.}

\bibitem{Gconv} { W. Goldman}, 
\newblock {Convex real projective structures
    on compact surfaces}, 
\newblock{\em J. Differential Geometry}, 31 (1990), 791--845.

\bibitem{wmgnote} { W. Goldman}, 
\newblock{ Projective geometry on manifolds}, 
\newblock{Lecture notes available from the author 1988.}


\bibitem{GH} {W. Goldman and M. Hirsch}, 
\newblock{ Affine manifolds and orbits of algebraic groups}, 
\newblock{{\em Trans. Amer. Math. Soc.} 295 (1986), no. 1, 175--198.}

\bibitem{GL} { W. Goldman and F. Labourie},  
\newblock{Geodesics in Margulis space times},
\newblock{{\em Ergod. Th. \& Dynamic. Sys.} 32 (2012), 643--651.}

\bibitem{GLM} { W, Goldman, F. Labourie, and G. Margulis},  
\newblock{Proper affine actions and geodesic flows of
hyperbolic surfaces},
\newblock{{\em Annals of Mathematics} 170 (2009), 1051--1083. }

\bibitem{Gr} { M. Gromov},
\newblock{ Groups of polynomial growth and expanding maps},
\newblock{ {\em Inst. Hautes \'Etudes Sci. Publ. Math.} No. 53 (1981), 53--73. }



\bibitem{Guichard} { O. Guichard}, 
\newblock{Sur la r\'egularit\'e H\"older des convexes divisibles}, 
\newblock{{\em Erg. Th. \& Dynam. Sys.} 25 (2005), 1857--1880.}


\bibitem{GW} { O. Guichard and A. Wienhard}, 
\newblock{Anosov representations: domains of discontinuity and applications}, 
\newblock{{\em Invent. Math.}} 190 (2012), no. 2, 357--438. 

\bibitem{GKW} {O. Guichard, F. Kassel, and A. Wienhard},
\newblock{Tameness of Riemannian locally symmetric spaces arising from Anosov representations,}
\newblock{arXiv:1508.04759 }





\bibitem{heard}
D. Heard, C. Hodgson, B. Martelli, and C. Petronio, 
\newblock Hyperbolic graphs of small complexity, 
\newblock {\em Exp. Math.} 19 no. 2 (2010), 211--236.

\bibitem{HK}
{K. Hoffman and R. Kunze},
\newblock{\em Linear algebra, }
Second edition, Prentice-Hall, Inc., Englewood Cliffs, N.J. 1971 viii+407 pp. 

\bibitem{HP} 
M. Heusener and J. Porti, 
\newblock Infinitesimal projective rigidity under Dehn filling, 
\newblock {\em Geom. Topol.} 15 no.4 (2011), 2017--2071. 



\bibitem{JM}
{ D. Johnson and J. Millson}, 
\newblock{Deformation spaces associated to compact hyperbolic manifolds},
\newblock In {\em Discrete groups in geometry and analysis} (New Haven, Conn., 1984), pp. 48--106,
\newblock Progr. Math., 67, Birkh\"auser Boston, Boston, MA, 1987.



\bibitem{Kac1967}
{  V.G. Kac and  \`{E}.B. Vinberg}, 
    \newblock Quasi-homogeneous cones,
    \newblock{{\em Math. Zametki}} 1 (1967), 347--354.
    
    \bibitem{KL} {M. Kapovich, B. Leeb,}
\newblock{Finsler bordifications of symmetric and certain locally symmetric spaces,}
\newblock{arXiv:1505.03593 }

\bibitem{KLP} {M. Kapovich, B. Leeb, and J. Porti,}
A Morse Lemma for quasigeodesics in symmetric spaces and euclidean buildings
\newblock{    arXiv:1411.4176. }
    
    \bibitem{Katok} 
 {   A. Katok and B. Hasselblatt}, 
    \newblock{\em Introduction to modern theory of dynamical systems}, 
    \newblock{ Cambridge University Press} 1995.
    
    \bibitem{ink} 
    { I. Kim}, 
    \newblock Compactification of strictly convex real projective structures,
\newblock {\em Geom. Dedicata} 113 (2005), 185--195. 

\bibitem{Kobpaper} { S. Kobayashi},
\newblock Projectively invariant distances for affine and projective structures, 
\newblock In: {\em Differential geometry} (Warsaw, 1979), 127--152, 
Banach Center Publ., 12, PWN, Warsaw, 1984. 


\bibitem{BK} { B. Kostant}, 
\newblock{On convexity, the Weyl group and the Iwasawa decomposition}, 
\newblock{{\em Ann. ENS.} 4em s\'eree tome 6 no. 4 (1973), 413--455}. 

\bibitem{KS} { B. Kostant and D. Sullivan},
\newblock{The Euler characteristic of an affine space form is zero}, 
\newblock{{\em Bull. Amer. Math. Soc.} 81 (1975), no. 5, 937--938. }


\bibitem{Kos} { J. Koszul}, 
\newblock{Deformations de connexions localement plates},
\newblock{{\em Ann. Inst. Fourier {\rm (}Grenoble{\rm )}} 18 fasc. 1 (1968), 103--114. }

\bibitem{Lab} { F. Labourie}, 
\newblock{ Flat projective structures on surfaces and cubic holomorphic differentials,} 
\newblock{\em Pure and applied mathematics quaterly} 3 no. 4 (2007), 1057--1099.




\bibitem{Leitner1} {    A. Leitner}, 
\newblock{Limits under conjugacy of the diagonal subgroup in $\SL_{3}(\bR)$,}
\newblock{{ {\em Geom. Dedicata} 180 (2016), 135--149.} }


\bibitem{Leitner2} {   A. Leitner}, 
\newblock{Limits under conjugacy of the diagonal subgroup in $\SL_{n}(\bR)$,}
\newblock{ {{\em Proc. Amer. Math. Soc.} 144 (2016), no. 8, 3243--3254.} }



    \bibitem{Marquis}
{   L. Marquis},
    \newblock Espace des modules de certains poly\`edres projectifs miroirs,
    \newblock{\em Geom. Dedicata} 147 (2010), 47--86.
    
    \bibitem{Mess}
 {   G. Mess}, 
\newblock{Lorentz spacetimes of  curvature}, 
\newblock{{\em Geom. Dedicata} 126 (2007), 3--45. }

\bibitem{Mo1} 
{ P. Molino},
\newblock G\'eom\'etrie global des feuilletages riemanniens, 
\newblock {\em Nederl. Akad. Wetensch. Indag. Math.} 44 (1982), no. 1, 45--76.

\bibitem{Moore} 
{ C. Moore}, 
\newblock Distal affine transformation groups, 
\newblock{\em Amer. J. Math.} 90 (1968){,} 733--751.

 \bibitem{Mol} { P. Molino}, 
\newblock {\em Riemannian foliations},
\newblock Progress in Mathematics, vol 73, Birkh\"auser, Boston, Basel, 1988.

\bibitem{Rag} { M. S. Raghunathan}, 
\newblock{{\em Discrete subgroups of Lie groups}}, 
\newblock{Ergebnisse der Mathematik und ihrer Grenzgebiete, Band 68}, Springer Verlag, Berlin, 1972.


\bibitem{Shbook} { H. Shima},
\newblock{{\em The geometry of Hessian structures},}
\newblock World Scientific Publishing Co. Pte. Ltd., Hackensack, NJ, xiv+246 pp, 2007. 

\bibitem{ST} D. Sullivan and W. Thurston, 
\newblock Manifolds with canonical coordinate charts: some examples, 
\newblock {\em Enseign. Math.} (2) 29 no.1--2 (1983) 15--25.


\bibitem{Thnote} { W. Thurston}, 
\newblock{\em Geometry and topology of $3$-manifolds,} 
\newblock{available at \url{http://library.msri.org/books/gt3m/}.}

\bibitem{Thbook} { W. Thurston}, 
\newblock{\em Three-dimensional geometry and topology,} 
\newblock{Princeton University Press, Princeton NJ, 1997.}


\bibitem{Var} { V.S. Varadarajan}, 
\newblock{{\em Lie groups, Lie algebras, and their representations}, }
\newblock{GTM Vol 102, Springer, Berlin, 1972.}

 \bibitem{Vey68} { J. Vey},  
  \newblock Une notion d'hyperbolicit\'e sur les vari\'et\'es localement plates, 
  \newblock {\em C.R. Acad. Sc. Paris}, 266(1968), 622--624.

 \bibitem{Vey} { J. Vey}, 
  \newblock Sur les automorphismes affines des ouverts convexes saillants,
  \newblock {\em Ann. Scuola Norm. Sup. Pisa} (3) 24(1970), 641--665.



 \bibitem{Vinberg1971}
  { \`{E}.B. Vinberg}, 
    \newblock Discrete linear groups that are generated by reflections,
    \newblock {\em Izv. Akad. Nauk SSSR} Ser. Mat. 35 (1971), 1072--1112.

  \bibitem{Vinberg1985}
 {  \`{E}.B. Vinberg}, 
    \newblock Hyperbolic reflection groups,
    \newblock {\em Uspekhi Mat. Nauk} 40 (1985), 29--66.
       
\bibitem{vin63}
{  \`{E}.B. Vinberg},
\newblock Homogeneous convex cones,
{\em Trans. Moscow Math. Soc.} 12 (1963), 340--363.

\bibitem{Yaman} A. Yaman,
\newblock{ A topological characterization of relatively hyperbolic groups},
\newblock{{\em J. Reine Angew. Math.} 566 (2004), 41--89.}


 \bibitem{Weil1962}
  {  A. Weil}, 
    \newblock On discrete subgroups of Lie groups II,
    \newblock{\em Ann. of Math.} 75 (1962), 578--602.

   \bibitem{Weil1964}
   { A. Weil}, 
    \newblock Remarks on the cohomology of groups,
    \newblock{\em Ann. of Math.} 80 (1964), 149--157.


\bibitem{DW} 
{ D. Witte}, 
\newblock{Superrigidity of lattices in solvable Lie groups},
\newblock{ {\em Inv. Math. } 122 (1995), 147--193.}

\end{thebibliography}



\end{document}